\documentclass[12pt]{amsart}
\usepackage{mathrsfs}  
\usepackage[english]{babel}
\usepackage[T1]{fontenc}
\usepackage[utf8]{inputenc} 

\usepackage{relsize}
\usepackage{amssymb,latexsym}
\usepackage{mathtools}
\usepackage{enumerate}  
\usepackage{a4wide}    

\usepackage[colorlinks=true,citecolor=blue,linkcolor=blue]{hyperref}
\usepackage{latexsym,amsmath,amssymb}
\usepackage{accents}
\usepackage{color}  
\title[Regularity at the Neumann boundary]{Boundary regularity for conformally invariant variational problems with Neumann data}

\author{Armin Schikorra}

\address[Armin Schikorra]{Mathematisches Institut, 
Abt. f\"ur Reine Mathematik, 
Albert-Ludwigs-Universit\"at,
Eckerstra\ss{}e 1,
79104 Freiburg im Breisgau,
Germany}
\email{armin.schikorra@math.uni-freiburg.de}

plus 6pt minus 12pt
\def\eps{\varepsilon}

\def\N{{\mathbb N}}

\def\S{{\mathbb S}}

\renewcommand{\div}{\dv}

\newcommand{\Err}{\mathscr{E}}
\newcommand{\err}{\epsilon}

\newtheorem{theorem}{Theorem}
\newtheorem{lemma}[theorem]{Lemma}
\newtheorem{corollary}[theorem]{Corollary}
\newtheorem{proposition}[theorem]{Proposition}

\theoremstyle{definition}
\newtheorem{remark}[theorem]{Remark}

\newtheorem{example}[theorem]{Example}

\newtheorem{cond}[theorem]{Condition}

\def\dist{{\rm dist\,}}

\def\curl{{\rm curl\,}}

\def\supp{{\rm supp\,}}

\newcommand{\R}{\mathbb{R}}

\newcommand{\brac}[1]{\left (#1 \right )}

\newcommand{\Ep}{\bigwedge\nolimits}

\newcommand{\ve}[1]{{ \boldsymbol{#1}}}

\newcommand{\tail}[5]{{\operatorname{Tail}}_{#1}(#2; #3, #4, #5)}

\newcommand{\barint}{
\rule[.036in]{.12in}{.009in}\kern-.16in \displaystyle\int }

\newcommand{\barcal}{\mbox{$ \rule[.036in]{.11in}{.007in}\kern-.128in\int $}}

\newcommand{\solU}{{\bf U}}
\newcommand{\solu}{{\bf u}}
\newcommand{\solW}{{\bf W}}
\newcommand{\solV}{{\bf V}}

\def\mvint_#1{\mathchoice
          {\mathop{\vrule width 6pt height 3 pt depth -2.5pt
                  \kern -8pt \intop}\nolimits_{\kern -3pt #1}}%
          {\mathop{\vrule width 5pt height 3 pt depth -2.6pt
                  \kern -6pt \intop}\nolimits_{#1}}%
          {\mathop{\vrule width 5pt height 3 pt depth -2.6pt
                  \kern -6pt \intop}\nolimits_{#1}}%
          {\mathop{\vrule width 5pt height 3 pt depth -2.6pt
                  \kern -6pt \intop}\nolimits_{#1}}}

\numberwithin{theorem}{section} \numberwithin{equation}{section}

\newcommand{\lap}{\Delta }
\newcommand{\laph}{\laps{1}}
\newcommand{\lapv}{(-\lap)^{\frac{1}{4}}}
\newcommand{\aleq}{\precsim}
\newcommand{\ageq}{\succsim}
\newcommand{\aeq}{\approx}
\newcommand{\Rz}{\mathscr{R}}
\newcommand{\Hz}{\mathscr{H}}
\newcommand{\Max}{\mathscr{M}}
\newcommand{\mfdM}{\mathcal{M}}

\newcommand{\hardy}{{\mathscr{H}^1}}
\newcommand{\laps}[1]{(-\lap) ^{\frac{#1}{2}}}

\newcommand{\lapms}[1]{I^{#1}}
\newcommand{\lapmv}{\lapms{\frac{1}{2}}}

\newcommand{\dv}{\operatorname{div}}
\begin{document}
\setcounter{tocdepth}{1}


\begin{abstract}
We study boundary regularity of maps from two-dimensional domains into manifolds which are critical with respect to a generic conformally invariant variational functional and which, at the boundary, enter perpendicularly into a support manifold. For example, harmonic maps, or $H$-surfaces, with a partially free boundary condition.

In the interior it is known, by the celebrated work of Rivi\`{e}re, that these maps satisfy a system with an antisymmetric potential, from which one can derive regularity of the solution. We show that these maps satisfy along the boundary a system with a nonlocal antisymmetric boundary potential which contains information from the interior potential and the geometric Neumann boundary condition. We then proceed to show boundary regularity for solutions to such systems.
\end{abstract}

\maketitle
\tableofcontents

\nocite{*}
\section{Introduction}
In the last three decades the interior regularity theory for maps between a surface and a manifold which are critical with respect to a conformally invariant variational energy has seen tremendous progress. One example of such an energy is the Dirichlet energy,
\[
 \mathscr{E}(\solU) := \frac{1}{2} \int_{D} |\nabla \solU|^2
\]
acting on maps $\solU: D \subset \R^2 \to \mfdM$, where $\mfdM\subset \R^N$ is a smooth, compact submanifold without boundary. Indeed, this energy is conformally invariant, since for conformal transforms $\tau: D \to \tau(D)$,
\[
 \int_{D} |\nabla (\solU\circ \tau)|^2 := \int_{\tau(D)} |\nabla \solU|^2.
\]

Maps $\solU$ which are critical with respect to the energy $\mathscr{E}$ in the class of maps from $D$ into $\mfdM$ are called (weakly) \emph{harmonic maps} from $D$ into $\mfdM$, and they are characterized by the \emph{harmonic map equation}
\begin{equation}\label{eq:harmonicmapeq}
 \lap \solU \perp T_\solU \mfdM \quad \mbox{in $D$}.
\end{equation}
For the case of a target sphere $\mfdM = \S^{N-1}$ one can rewrite this equation
\begin{equation}\label{eq:harmmapeqSphere}
 -\lap \solU = \solU |\nabla \solU|^2 \quad \mbox{in $D$}.
\end{equation}
This is a vectorial equation which can be equivalently written as
\[
 -\lap \solU^i = \solU^i |\nabla \solU|^2, \quad \mbox{in $D$,}\quad i \in \{1,\ldots,N\}.
\]
Equation \eqref{eq:harmmapeqSphere} is critical in two dimensions, and for a long time was not accessible to classical potential regularity theory, since for a finite-energy solution $\solU \in H^1(D,\R^N)$ the right-hand side seems to belongs only to $L^1(D)$. Indeed, it is not possible to conclude boundedness or continuity from the growth properties of this equation: for example the classical counterexample $u(x) := \log \log 2/|x|$ is an $H^1$-solution on $B(0,1)$ to an equation of the same growth properties,
\[
 |\lap u| = |\nabla u|^2 \quad \mbox{in $B(0,1) \subset \R^2$}
\]
Also a priori assumptions on the boundedness of the solution yield no advantage, similar equations are satisfied by $u(x) := \sin \log \log 2/|x|$. On the other hand, any \emph{continuous} solution of \eqref{eq:harmmapeqSphere} is as smooth as the manifold allows, and in contrast to the initial regularity this fact follows directly from the growth of the equation, see \cite{Tomi-1969}. The equations \eqref{eq:harmonicmapeq}, \eqref{eq:harmmapeqSphere} are thus \emph{critical} and proving initial regularity such as continuity for solutions is the only analytic obstacle to a full regularity theory.

For sphere-valued harmonic maps in two dimensions this initial regularity was obtained by H\'elein in \cite{Helein-1990}. For $D \subset \R^2$, he showed that any solution $\solU \in H^1(D,\S^{N-1})$ to \eqref{eq:harmmapeqSphere} is continuous (and thus smooth). He used the conservation laws for sphere-valued harmonic maps discovered by Shatah \cite{Shatah-1988}: for $\solU \in H^1(D,\S^{N-1})$ Equation~\eqref{eq:harmmapeqSphere} is equivalent to
\begin{equation}\label{eq:divomegaz}
 \div(\Omega_{ij}) = 0 \quad \mbox{for all }i,j =1,\ldots,N,
\end{equation}
where
\[
 \Omega_{ij} := \solU^i\nabla \solU^j-\solU^j \nabla \solU^i \in L^2(D,\R^2)\quad i,j =1,\ldots,N
\]
This allows to rewrite equation~\eqref{eq:harmmapeqSphere},
\begin{equation}\label{eq:harmmapsphererew}
 -\lap \solU^i = \solU^i \nabla \solU^j\cdot\nabla \solU^j = \Omega_{ij}\cdot \nabla \solU^j \quad i =1,\ldots,N,
\end{equation}
since $\solU^j \nabla \solU^j = \frac{1}{2} \nabla |\solU|^2 = 0$.

Observe that in view of \eqref{eq:divomegaz}, $\Omega_{ij}\cdot \nabla \solU^j \in L^1(D)$ is the product of a divergence-free and curl-free quantity. This, in turn, implies that $\Omega_{ij}\cdot \nabla \solU^j$ belongs to a strictly smaller space than $L^1(D)$, the Hardy space $\hardy(D)$. The latter was shown by Coifman, Lions, Meyer, Semmes \cite{CLMS-1993} which gave a concluding explanation for the ``compensation effects'' that had been observed for Jacobians as early as Wente's \cite{Wente69} and Reshetnyak's work \cite{Reshetnyak-1967}, then later in terms of the so-called Wente inequality \cite{BrC84,Tartar-1982} and $L^1 \log L$-integrability \cite{Mueller-1990}. Since the Riesz potential $\lapms{2}$  maps $\hardy(\R^2)$ into the continuous functions $C^0(\R^2)$, solutions of \eqref{eq:harmmapsphererew} are indeed continuous.

For harmonic maps into general smooth, closed manifolds $\mfdM$ H\'elein developed the so-called moving frame technique in \cite{Helein-1991}, see also his monograph \cite{Helein-book}: if there is a frame $(e_\sigma)_{\sigma=1}^{\dim \mfdM}$, that is a smooth orthonormal basis of the tangent vector fields in $T\mfdM$, then \eqref{eq:harmonicmapeq} becomes
\begin{equation}\label{eq:movingframe}
 -\div(\langle e_\sigma(\solU), \nabla \solU \rangle_{\R^N}) = \langle \nabla e_\sigma(\solU), e_\tau(\solU) \rangle_{\R^N}\cdot \langle e_\tau(\solU), \nabla \solU \rangle_{\R^N}.
\end{equation}
Of course there is a degree of freedom to choose such vector fields (once one exists), and H\'elein showed that one can find one frame $(e_\sigma)_{\sigma=1}^{\dim \mfdM}$ such that 
\[
 \div \langle \nabla e_\sigma(\solU), e_\tau(\solU) \rangle = 0.
\]
Therefore \eqref{eq:movingframe} becomes again an equation with a div-curl term on the right-hand side (up to the multiplicative $e_\tau(\solU) \in H^1\cap L^\infty(D,\R^N)$). Using again that such a div-curl term belongs to the Hardy space $\mathcal{H}^1(D)$, one can show continuity of solutions $\solU$ to \eqref{eq:movingframe}.

Another example is the \emph{prescribed mean curvature equation}. Let $\solU \in H^1(D,\R^3)$ be a solution to
\begin{equation}\label{eq:Hsurfaces}
 \lap \solU = 2H(\solU)\, \solU_x \wedge \solU_y \quad \mbox{in $D$}
\end{equation}
where $\wedge$ denotes the wedge product for vectors in $\R^3$ and $H: \R^3 \to \R$ is a bounded map. If one additionally assumes that $\solU$ is conformal, i.e. $|\solU_x| = |\solU_y|$ and $\solU_x \perp \solU_y$ then $\solU$ parametrizes a surfaces $\solU(D) \subset \R^3$ with mean curvature $H(\solU(x))$ in $\solU(x)$, hence the name of the equation. From a variational point of view, solutions $\solU$ to \eqref{eq:Hsurfaces} appear as critical points of the functional
\begin{equation}\label{eq:Qharmmap}
 \mathscr{E}_Q(\solU) = \int_{D} \frac{1}{2}|\nabla \solU|^2 + Q(\solU) \cdot \solU_x \wedge \solU_y
\end{equation}
where $Q \in C^1(\R^3,\R^3)$ is a vector field with the property $2H \equiv \div Q$. Again \eqref{eq:Hsurfaces} is a critical equation: initial regularity does not follow from direct potential methods, but once continuity is obtained, the solution is as smooth as the $H$ allows. Using the div-curl theory, Bethuel \cite{Bethuel-1992} showed initial regularity under the assumption that $H$ is bounded and Lipschitz. An elementary computation shows that 
\[
 \solU_x \wedge \solU_y = \left (\begin{array}{c}\det(\nabla \solU^{2}, \nabla \solU^{3})\\ \det(\nabla \solU^{1}, \nabla \solU^{3})\\ \det(\nabla \solU^{1}, \nabla \solU^{2})\\ \end{array} \right )
\]
That is, \eqref{eq:Hsurfaces} is a system with Jacobians on the right-hand side, up to the multiplicative $H(\solU)$. If $H$ is bounded and Lipschitz then $H(\solu) \in H^1(D,\R^3)$, and thus, as above in \eqref{eq:movingframe} we find a Jacobian (i.e. div-curl term) in the Hardy-space up to multiplicative function in $H(\solU) \in H^1\cap L^\infty(D,\R^3)$, and one can prove regularity.

Both, harmonic maps and $H$-surfaces, are special cases of critical points of conformally invariant variational functionals of the form
\begin{equation}\label{eq:conformallyinvariantvf}
 \mathscr{E}(\solU) = \int_{D} \frac{1}{2} |\nabla \solU|^2 + \solU^\ast(\lambda)
\end{equation}
where $\lambda \in C^1(\Ep^2\R^3)$ is a two-form and $\solU^\ast(\lambda)$ denotes the pullback of $\lambda$ under $\solU$. Indeed, Gr\"uter \cite{Grueter-1984} proved that under some ``natural conditions'' all elliptic, conformally invariant elliptic variational functional in two dimensions are of the form \eqref{eq:conformallyinvariantvf}, up to a conformal transform of the domain.

It was Rivi\`ere \cite{Riviere-2007} who discovered that the Euler-Lagrange equation of the Dirichlet functional, the $H$-surface functional \eqref{eq:Qharmmap}, and in general all generic conformally invariant variational functionals such as \eqref{eq:conformallyinvariantvf}, possibly restricted to maps into a closed manifold $\mfdM$, share a crucial regularizing structure: the right-hand side may not be a Jacobian, but it can always be seen as an \emph{antisymmetric} potential acting on the solution $\solU$. More precisely, any critical point $\solU \in H^1(D,\mfdM)$ of a variational functional of the form \eqref{eq:conformallyinvariantvf}, where $\mfdM$ can be a closed manifold or $\R^N$, solves an equation of the form
\begin{equation}\label{eq:rivieresystem}
 \lap \solU^i = \Omega_{ij} \cdot \nabla \solU^j \quad \mbox{in $D$},
\end{equation}
where $\Omega_{ij} \in L^2(D,\R^2)$ is antisymmetric:
\[
 \Omega_{ij}  = - \Omega_{ji} \quad \mbox{almost everywhere in $D$}.
\]
He then showed interior initial regularity for solutions of \eqref{eq:rivieresystem}. More precisely we have
\begin{theorem}[Rivi\`{e}re \cite{Riviere-2007}]
Let $\solU \in H^1(D,\R^N)$ be a solution of \eqref{eq:rivieresystem}. Then $\solU$ is continuous in $D$. 
\end{theorem}
Rivi\`ere's interior regularity result has seen many extensions, among them generalizations to biharmonic maps \cite{Lamm-Riviere-2008}, to half-harmonic maps \cite{DaLio-Riviere-1Dmfd}, polyharmonic maps \cite{DaLio-nmfd}, and to Schroedinger-type systems \cite{Riviere-2011}. Again these equations are all critical: once \emph{initial} regularity is obtained, \emph{higher} regularity follows from a bootstrap argument, see \cite{Sharp-Topping-2013,Moser-2015,Schikorra-eps}. 

It is natural to investigate the regularity up to the boundary for these kind of geometric equations. In \cite{Mueller-Schikorra-2009} M\"uller and the author considered the Dirichlet problem and showed
\begin{theorem}[M\"uller-S. \cite{Mueller-Schikorra-2009}] Let $D$ be a smoothly bounded domain. Then any solution $\solU \in H^1(D,\R^N)$ of
\[
\begin{cases}
 \lap \solU^i = \Omega_{ij} \cdot \nabla \solU^j \quad &\mbox{in $D$},\\
 \solU = \solu \quad &\mbox{on $\partial D$},
 \end{cases}
\]
is continuous up to the boundary, $\solU \in C^0(\overline{D})$, if $\solu$ is continuous on $\partial D$.
\end{theorem}
In this work we investigate boundary regularity for prescribed geometric Neumann data, namely when the solution $\solU$ at the boundary penetrates a support manifold $\mathcal{N}$ perpendicularly. 
\begin{theorem}\label{th:geometryreg}
Let $D \subset \R^2$ be a smoothly bounded domain and $\mathcal{N}$ be a closed manifold in $\R^N$. Assume that $\solU \in H^1(D,\R^{N})$ with trace $\solu \in H^{\frac{1}{2}}(\partial D,\mathcal{N})$ is a solution to
\begin{equation}\label{eq:omegafreebd}
 \begin{cases}
\lap \solU = \Omega \cdot \nabla \solU \quad &\mbox{in $D$}\\
\partial_\nu \solU \perp T_{\solu} \mathcal{N} \quad &\mbox{on $\partial D$}
 \end{cases}
\end{equation}
where $\Omega \in L^2(D)$ is antisymmetric.

Then $\solU$ is H\"older continuous in $D$ up to the boundary $\partial D$.
\end{theorem}
A Neumann condition for systems with antisymmetric potential has been considered already in \cite{Sharp-Zhu-2016} where $\partial_\nu \solU$ solves a subcritical equation along the boundary, which is elliptic in the sense of boundary equations as studied by \cite{Agmon-Douglis-Nirenberg-1959}. But observe that in our case, $\partial_{\nu} \solU \in L^1(\partial D)$ only. That is, \eqref{eq:omegafreebd} is a critical equation both, in the interior and at the boundary.

Let us also mention the special case $\Omega \equiv 0$. Then the boundary equation 
\begin{equation}\label{eq:frbdcond}\partial_\nu \solU \perp T_{\solu} \mathcal{N} \quad \mbox{on $\partial D$}\end{equation}
is equivalent to the half-harmonic map equation
\[\laph_{\partial D} \solu \perp T_{\solu} \mathcal{N} \quad \mbox{in $\partial D$},\] 
where $\solu$ is the trace of $\solU$ at the boundary $\partial D$ and $\laph_{\partial D}$ denotes the half-laplacian along $\partial D$. The regularity theory for the half-harmonic map equation was proven by Da Lio and Rivi\`ere in their seminal work \cite{DaLio-Riviere-1Dmfd}. In particular they showed, that this equation also exhibits an antisymmetric potential on the right-hand side.

It is thus reasonable to suspect that the equation \eqref{eq:omegafreebd} can be reformulated into two coupled systems, one in the interior and one at the boundary, which both exhibit an antisymmetric potential on the right-hand side. We confirm this suspicion, and reduce \eqref{eq:omegafreebd} to an interior equation coupled with an equation along the boundary, and then show that the boundary equation exhibits a nonlocal antisymmetric potential which contains information of the Neumann condition \eqref{eq:frbdcond} and the interior potential $\Omega$. More precisely,  Theorem~\ref{th:geometryreg} will be a consequence of the following result, see Theorem~\ref{th:transform} below.
\begin{theorem}\label{th:mastereq}
Let $\solU \in H^1(\R^2_+,\R^N)$ and its trace $\solu \in H^{\frac{1}{2}}(\R,\R^N)$ be a solution to
\[
 \begin{cases}
 \lap \solU^i  = \Omega_{ij}\cdot \nabla \solU^j + \Err^i(\solU) \quad &\R^2_+\\
 \laph \solu^i = \omega_{ij}(\lapv \solu^j) + \err^i(\solU) \quad &\mbox{in }(-2,2).
 \end{cases}
\]
Here $\Err^i$ and $\err^i$ are benign error terms satisfying conditions~\ref{cond:Err} and \ref{cond:err}, respectively. 

Moreover, $\Omega$ is a pointwise antisymmetric potential ${\Omega}_{ij} = -{\Omega}_{ij} \in L^2(\R^2_+,\R^2)$, and $\omega_{ij}$ is a nonlocal potential, $\omega_{ij}: L^2(\R) \to L^1(\R)$, which is a linear operator given as
\[
  \int_{\R} \omega_{ij}(f)\ \varphi := \int_{\R}\int_{\R} \omega_{ij}(x,y)\, f(y)\ \varphi(x)\ dy\, dx
\]
whose kernel is antisymmetric, $\omega_{ij}(x,y) = -\omega_{ji}(x,y)$, and satisfies the boundedness and localization conditions~\ref{cond:omega} below.

Then $\solU$ is H\"older continuous in $\R^2_+ \cup \brac{(-1,1)\times \{0\}}$. 
\end{theorem}
Since we reformulated equation \eqref{eq:omegafreebd} into a system of local and nonlocal equations with antisymmetric potential, it should be possible to base higher regularity arguments on the related results for nonlocal equations, see \cite{Schikorra-eps}. This will be a future project of study.

Before we comment in the next Section more on the strategy of the proof for Theorem~\ref{th:geometryreg} and Theorem~\ref{th:mastereq}, let us have a look at consequences: the conditions $\solu: \partial D \to \mathcal{N}$ and $\partial_\nu \solU \perp T_{\solu} \mathcal{N}$ is motivated by partially free boundary problems. The first example is known from the Plateau problem.
\begin{corollary}
Assume that $\solU$ is a critical point of 
\[
 \int_{D} |\nabla \solU|^2 \quad \mbox{s.t. $\solu = \solU\Big|_{\partial D} \subset \mathcal{N}$}
\]
Then
\[
 \begin{cases}
\lap \solU = 0 \quad &\mbox{in $D$}\\
\partial_\nu \solU \perp T_{\solu} \mathcal{N} \quad &\mbox{on $\partial D$}
 \end{cases}
\]
and consequently, $\solU$ is H\"older continuous in $D$ up to the boundary.
\end{corollary}
As an interesting side remark, let us mention that Douglas' proof of the Plateau problem is actually related to our approach of computing an intrinsic nonlocal equation along the boundary, cf. \cite[equation (1.4)]{Douglas-1931}.
 
From Theorem~\ref{th:geometryreg} we also recover the regularity for harmonic maps with partially free boundary, which was originally obtained by Scheven, \cite{Scheven-2006}. Actually Scheven even proved partial regularity in dimensions $n \geq 3$ for domains $D \subset \R^n$.
\begin{corollary}
Assume that $\solU: D \to \mfdM \subset \R^N$ is a harmonic map with free boundary in $\mathcal{N} \subset \mfdM$, where $\mfdM$, $\mathcal{N}$ are smooth, closed manifolds in $\R^N$. That is,
let $\solU \in H^1(D,\mfdM)$ with trace $\solu \in H^{\frac{1}{2}}(\partial D,\mathcal{N})$ be a critical point of the Dirichlet energy
\[
 \int_{D} |\nabla \solU|^2: \quad \mbox{s.t. } \solU \in H^1(D,\mfdM) \mbox{ with trace }\solu = \solU \Big |_{\partial D} \in H^{\frac{1}{2}}(\partial D,\mathcal{N}).
\]
Then $\solU$ satisfies an equation of the form \eqref{eq:omegafreebd} and consequently $\solU$ is H\"older continuous in $D$ up to the boundary.
\end{corollary}
For $H$-surfaces, Theorem~\ref{th:geometryreg} implies regularity at a partially free boundary.
\begin{corollary}
Let $\mathcal{N}$ be a closed two-dimensional submanifold of $\R^3$. Assume that $\solU \in H^1(D,\R^3)$ is a critical point of the energy \eqref{eq:Qharmmap} subject to the partial free boundary condition $\solu = \solU \Big|_{\partial D} \subset \mathcal{N}$, i.e., let $\solU$ be a solution to
\[
 \begin{cases}
  \lap \solU = 2H(\solU)\, \solU_x \wedge \solU_y \quad &\mbox{in $D$}\\
  \partial_\nu \solU \perp T_\solu \mathcal{N} \quad &\mbox{on $\partial D$}.
 \end{cases}
\]
If $H = \frac{1}{2}\div Q$ is bounded, $H \in L^\infty(\R^3)$ and $Q$ satisfies the condition
\begin{equation}\label{eq:Qncondition}
 Q(p) \cdot {\bf n}(p) = 0 \quad \mbox{for all $p \in \mathcal{N}$}
\end{equation}
where ${\bf n}(p)$ denotes the unit normal of $\mathcal{N}$,
then $\solU$ is H\"older continuous in $D$ up to the boundary $\partial D$.
\end{corollary}
Condition \eqref{eq:Qncondition} was already used to to show regularity up to the free boundary under the assumption of conformal parametrization of $\solU$ in \cite{Grueter-Hildebrandt-Nitsche-1986} and \cite{Mueller-2005}.

Finally, Theorem~\ref{th:geometryreg} also implies the following free boundary version of Rivi\`ere's regularity theorem for critical maps of conformally invariant variational functionals, \cite[Theorem I.2]{Riviere-2007}.
\begin{corollary}
Let $\mfdM$ be a $C^2$-submanifold of $\R^N$, and $\lambda$ a $C^1$ 2-form on $\mfdM$ such that the $L^\infty$-norm of $d\lambda$ in bounded on $\mfdM$. Assume that $\mathcal{N} \subset \mfdM$ is a closed submanifold of $\mfdM$. Assume moreover that, similarly to \eqref{eq:Qncondition}, $\lambda$ satisfies the \emph{orthogonal angle condition} 
\begin{equation}\label{eq:anglecondition}
|\lambda(\ve{z})(\ve{v},\ve{w})| = 0 \quad \quad \forall\ \ve{v},\, \ve{w} \in T_{\ve{z}} \mathcal{N},\ \ve{z} \in \mathcal{N}.
\end{equation}
Then any critical point $\solU \in H^1(D,\mfdM)$, $\solu = \solU \Big|_{\partial D} \in H^{\frac{1}{2}}(\partial D,\mathcal{N})$ of the energy
\[
 \int_{D} |\nabla \solU|^2 + \solU^\ast \lambda,
\]
satisfies an equation of the form \eqref{eq:omegafreebd}, and therefore $\solU$ is H\"older continuous in $D$ up to the boundary $\partial D$.
\end{corollary}

\subsection*{Outline of the paper}
In the next section, Section~\ref{s:strategy}, we state in Theorem~\ref{th:transform} the reduction result which related Theorem~\ref{th:geometryreg} and Theorem~\ref{th:mastereq}; moreover we introduce our notation. In Section~\ref{s:proofthtransform} we give the proof of Theorem~\ref{th:transform}. In Section~\ref{s:proofoftheoremmaster} we proof Theorem~\ref{th:mastereq} using a suitable gauge for certain nonlocal antisymmetric functionals. This gauge is constructed in Section~\ref{s:optimalgauge}. Finally, in Section~\ref{s:extcom} we introduce new commutator estimates for extension operators, which are needed in the proof of Theorem~\ref{th:mastereq}, but which are interesting in their own right.

\section{The master equation with antisymmetric potentials: Proof of Theorem~\ref{th:geometryreg}}\label{s:strategy}
Recall the Dirichlet-to-Neumann property: Assume that $\solV \in H^1(\R^2_+,\R^N)$ is harmonic in the upper half plane $\R^2_+$,
\begin{equation}
\label{eq:Vdef}
 \begin{cases}
  \lap \solV = 0 \quad &\mbox{in }\R^2_+,\\
  \solV = \solu \quad &\mbox{on }\R \times \{0\}.\\
 \end{cases}
\end{equation}
with zero boundary conditions at infinity. That is, $\solV$ is the Poisson extension of $\solu$,
\[
 \solV(x,t)= p_t \ast \solu(x) := c \int_{\R} \frac{t}{|x-z|^2+t^2}\ \solu(z)\ dz.
\]
Here, $c$ is some dimensional constant (and in general $c$ may change from line to line).
In semigroup language, $\solV(x,t) = e^{-t \laph} \solu(x)$. Here, the $\alpha$-Laplacian is defined as
\[
 \laps{\alpha} \solu(x) := c \int_{\R} \frac{\solu(x)-\solu(y)}{|x-y|^{1+\alpha}}\ dx\ dy.
\]
Then we have the Dirichlet-to-Neumann property
\begin{equation}\label{eq:dtn}
 \partial_\nu \solV(x,0) \equiv -\partial_t \solV(x,0) =  \laph \solu(x) \quad x \in \R.
\end{equation}
That is, a condition on $\partial_\nu \solV$ at the boundary is simply a condition on $\laph \solu$.

This was used, e.g., by Millot and Sire in \cite{Millot-Sire}, to re-interpret the half-harmonic map equation 
\[
 \laph \solu \perp T_\solu \mathcal{N} \quad \mbox{in }\R
\]
as a minimal surface $\solV: \R^2_+ \to \R^N$ with partial free boundary
\[
 \begin{cases}
\lap \solV = 0 \quad \mbox{in $\R^2_+$}\\
\partial_\nu \solV \perp T_\solu \mathcal{N} \quad \mbox{in $\R \times \{0\}$}.
 \end{cases}
\]
This way they obtained partial regularity for half-harmonic maps from the work of Scheven~\cite{Scheven-2006}.

Our strategy is the reverse. In order to study solutions $\solU \in H^1(\R^2_+,\R^N)$ of the equation
\[ \begin{cases}
\lap \solU = \Omega \cdot \nabla \solU \quad &\mbox{in $\R^2_+$}\\
\partial_\nu \solU \perp T_{\solu} \mathcal{N} \quad &\mbox{on $\R \times \{0\}$},
 \end{cases}
\]
we interpret $\partial_\nu \solU$ as a nonlocal operator along the boundary. Observe, however, that since $\solU$ is not harmonic in the interior, the relation \eqref{eq:dtn} fails for $\solU$, and in general
\[
 \partial_\nu \solU(x,0) \neq \laph \solu(x) 
\]
Nevertheless we obtain the following
\begin{theorem}\label{th:transform}
Let $D \subset \R^2$ be a smoothly bounded domain and $\mathcal{N}$ be a closed manifold in $\R^N$. If $\tilde{\solU} \in H^1(D,\R^{N})$ with trace $\tilde{\solu} \in H^{\frac{1}{2}}(\partial D,\mathcal{N})$ is a solution to
\[
 \begin{cases}
\lap \tilde{\solU} = \tilde{\Omega} \cdot \nabla \tilde{\solU} \quad &\mbox{in $D$}\\
\partial_\nu \tilde{\solU} \perp T_{\tilde{\solu}} \mathcal{N} \quad &\mbox{on $\partial D$}
 \end{cases}
\]
where $\tilde{\Omega}_{ij} = -\tilde{\Omega}_{ji} \in L^2(D,\R^2)$ is antisymmetric, then for any point $x_0 \in \partial D$ we find a small radius $\rho > 0$, $\solU \in H^1(\R^2_+,\R^N)$, and a diffeomorphism $\tau: B(0,1) \to \R^2$ with
\[
\tilde{\solU} = \solU\circ \tau^{-1} \quad \mbox{in $B(x_0,\rho) \cap \overline{D}$},
\]
so that $\solU$ satisfies the following conditions:

The map $\solU \in H^1(\R^2_+,\R^N)$ has compact support. $\solU$ and its trace $\solu = \solU \Big|_{\R \times \{0\}} \in L^\infty \cap H^{\frac{1}{2}}(\R,\R^N)$ are a solution to
\begin{equation}\label{eq:trU}
 \begin{cases}
 \lap \solU^i  = \Omega_{ij}\cdot \nabla \solU^j + \Err^i \quad &\R^2_+\\
 \solU = {\solu} \quad &\mbox{on $\R \times \{0\}$},
 \end{cases}
\end{equation}
for some $\Err^i$ satisfying the conditions~\ref{cond:Err} below. Here $\Omega$ is a pointwise antisymmetric potential ${\Omega}_{ij} = -{\Omega}_{ij} \in L^2(\R^2_+,\R^2)$.

On the other hand, the trace $\solu \in L^\infty \cap H^{\frac{1}{2}}(\R,\R^N)$ satisfies
\begin{equation}\label{eq:transformedeqV}
  \laph \solu^i = \omega_{ij}(\lapv \solu^j) + \err^i(\solU) \quad \mbox{in }(-2,2)
\end{equation}
for the nonlocal, boundary antisymmetric potential $\omega_{ij} = -\omega_{ji}: \dot{H}^{\frac{1}{2}}(\R) \to L^1(\R)$ which is a linear operator given via
\[
  \int_{\R} \omega_{ij}(f)\ \varphi := \int_{\R}\int_{\R} \omega_{ij}(x,y)\, f(y)\ \varphi(x)\ dy\, dx
\]
whose kernel $\omega_{ij}(x,y)$ satisfies the boundedness and localization conditions~\ref{cond:omega} below. Moreover $\err(\solU)$ depends on $\solU$, i.e. on interior \emph{and} boundary values, but is an benign error term satisfying the conditions~\ref{cond:err} below.
\end{theorem}
\eqref{eq:trU} is a consequence of the usual flattening of the boundary argument. The main work is to obtain the boundary condition \eqref{eq:transformedeqV}. Thus, we successfully reformulated the Neumann boundary equation \eqref{eq:omegafreebd} into a coupled system, the interior system \eqref{eq:trU} which is local, and the boundary system \eqref{eq:transformedeqV}. Both equations are critical, but with antisymmetric potentials.
Theorem~\ref{th:geometryreg} is then a consequence of the regularity theorem Theorem~\ref{th:mastereq} for systems with antisymmetric potential in the interior and at the boundary. \qed

\subsection{Notation and conditions on $\omega$, $\Err$, $\err$}
We denote by $I(x_0,r) \subset \R$ one-dimensional open balls and with $B(x_0,r) \subset \R^2$ two-dimensional open balls each centered at $x_0$ with radius $r$. For $x_0 \in \R \times \{0\}$ we define $B(x_0,r)^+$ the upper semi-ball $B(x_0,r) \cap \R^2_+$.

We will use the notion of Lorentz spaces, for a gentle introduction we refer to \cite{GrafakosC}. 
For measurable functions $f: \R^n \rightarrow \mathbb R$ and $\Omega \subset \R^n$ the \emph{decreasing rearrangement} of $f$ is
\[
 f^\ast (t):= \inf\left\{s>0: \left |\{x \in \R^n:\ |f(x)|> s\} \right |\leq t\right\}.
\]
Then the Lorentz space $L^{(p,q)}$ is induced by the pseudo-norm $\|\cdot\|_{(p,q),\R^n)}$, defined for $p \in [1,\infty)$, $q \in [1,\infty)$ by
\[
 \|f\|_{(p,q)} \equiv \|f\|_{(p,q),\R^n} := \left(\int\limits_0^\infty \left (t^{\frac{1}{p}} f^\ast (t) \right) ^q \frac {dt}{t}\right)^{\frac{1}{q}}.
\]
For $q = \infty$ we define
\[
\|f\|_{(p,\infty)} \equiv  \|f\|_{(p,\infty),\R^n} := \sup\limits_{t>0} t^{\frac{1}{p}}\, f^\ast (t)
\]
Note that $\|\cdot\|_{(p,q),\R^n}$ does not satisfy the triangular inequality with constant one, but otherwise it is a norm. For a measurable subset $D \subset \R^n$,
\[
  \|f\|_{(p,q),D} :=  \|\chi_D\, f\|_{(p,q)}.
\]
The Lorentz space provide a finer scale of Lebesgue spaces, in particular it holds $L^{(p,p)}(D) = L^p(D)$ with equivalent norms. We also have the embedding $L^{p,q_1}(D) \subset L^{p,q_2}(D)$ for any $D \subset \R^n$ if $q_1 > q_2$. Indeed,
\begin{equation}\label{eq:lorentzembedding}
 \|f\|_{(p,q_2)} \aleq C\ \|f\|_{(p,q_1)}.
\end{equation}
Moreover the Lorentz space version of H\"older inequality holds: for $p_1,p_2,p \in [1, \infty)$ and $q_1,q_2,q \in [1,\infty]$ with $1/p_1 + 1/p_2 = 1/p$ and $1/q_1 + 1/q_2 = 1/q$ we have
\begin{equation}\label{eq:hoelder}
  \|fg\|_{(p,q)} \aleq \|f\|_{(p_1,q_1)}\ \|g\|_{(p_2,q_2)}.
\end{equation}
Also, we have a version of Young inequality away from $L^1$ and $L^\infty$: for $p_1, p_2, p \in (1,\infty)$ and $q_1,q_2 \in [1, \infty]$
with $1/p_1 + 1/p_2 = 1 /p + 1$ and $1/q_1 +1/q_2 = 1/q$ we have
\begin{equation}\label{eq:young}
 \|f \ast g\|_{(p,q)} \aleq \|f\|_{(p_1,q_1)}\; \|g\|_{(p_2,q_2)}.
\end{equation}
Since we are working with nonlocal quantities, tails cannot be avoided. We write
\[
\tail{\sigma}{\|f\|_{(2,\infty)}}{x_0}{R}{k_0} = \sum_{k=k_0}^\infty 2^{-\sigma k} \|f\|_{(2,\infty),I(x_0,2^{k} R)},
\]
or
\[
\tail{\sigma}{\|f\|_{(2,\infty)}}{x_0}{R}{k_0} = \sum_{k=k_0}^\infty 2^{-\sigma k} \|f\|_{(2,\infty),B(x_0,2^{k} R)},
\]
depending on the dimension.

Here, $\sigma >0$ is a constant that will change from line to line.

\begin{cond}[Conditions on $\omega$]\label{cond:omega}
The kernel $\omega(x,y): \R \times \R \to \R$ is admissible if it measurable, and bounded in the following sense
\[
 \int_{\R} \int_{\R}f(y)\, \varphi(x)\ \omega(x,y)\ dx\, dy\ \aleq \|f\|_{L^2(\R)}\ \|\varphi\|_{L^\infty(\R)}.
\]
Moreover we require the following \emph{localization properties}:

For any $\eps > 0$, there is an $R \in (0,1)$ so that for any sufficiently large $k_0 \geq 2$ and any $r \in (0,2^{-k_0} R)$ the following holds for some uniform $\sigma > 0$.

For any $\|g\|_{\infty} + \|\lapv g\|_{2} \leq 1$ and any $\varphi \in C_c^\infty(I(x_0, r))$ with $\|\varphi\|_{\infty} + \|\lapv \varphi\|_{2} \leq 1$ and 
\begin{equation}\label{eq:oa:1}
\begin{split}
 &\int_{\R} \int_{\R} f(y)\, g(x)\, (\varphi(x)-\varphi(y))\, \omega_{kj} (x,y)\, dx\, dy\\
 \aleq & 
 \ \eps \|f\|_{(2,\infty),I(x_0,2^{k_0})} + \tail{\sigma}{\|f\|_{(2,\infty)}}{x_0}{r}{k_0}.
 \end{split}
\end{equation}
For any $\|g\|_{\infty} + \|\lapv g\|_{2} \leq 1$, any $\|\varphi\|_{\infty} + \|\lapv \varphi\|_{2} \leq 1$, and any $f \in C_c^\infty(I(x_0, r))$
\begin{equation}\label{eq:oa:2}
 \int_{\R} \int_{\R} f(y)\, g(x)\, (\varphi(x)-\varphi(y))\, \omega_{kj} (x,y)\, dx\, dy \aleq 
 \eps \|f\|_{(2,\infty)},
\end{equation}
and for any $\|g\|_{\infty} + \|\lapv g\|_{2} \leq 1$, any $\varphi \in C_c^\infty(I(x_0, r))$ and any $f \in C_c^\infty(\R)$
\begin{equation}\label{eq:oa:3}
 \int_{\R} \int_{\R} f(y)\, g(x)\, (\varphi(x)-\varphi(y))\, \omega_{kj} (x,y)\, dx\, dy \leq 
 \eps \|f\|_{2,\R}\ \|\lapv \varphi\|_{(2,\infty)}
\end{equation}
\end{cond}

\begin{example}
For our setup, $\omega$ will be a composition of the following examples
\begin{enumerate}
\item A first example of a kernel $\omega(x,y)$ satisfying Condition~\ref{cond:omega} is \[\omega(x,y) :=\tilde{\omega}(x)\, \delta_{x,y} \] for some $\tilde{\omega} \in L^2(\R)$. Indeed, all the localization conditions are trivially satisfied since the left-hand sides are zero.
\item A second example is
 \[
 \int_{\R} \int_{\R} \omega(x,y)\, f(y)\, \varphi(x)\ dx\, dy= \int_{\R^2_+} \Omega(z,t)\,  \kappa_t \ast f(z)\, p_t\ast \varphi(z)\, dz,
\]
that is
\[
 \omega(x,y) = \int_{\R^2_+} \Omega(z,t)\, \kappa_t(z-y)\,   p_t(z-x),
\]
where $\Omega \in L^2(\R^2_+)$, $\kappa,p \in C^\infty\cap L^1 \cap L^\infty(\R)$ and $\int \kappa = 0$. Here $\kappa_t(z-x) = t^{-1}\kappa((z-x)/t)$, and $p_t(z-x) = t^{-1}p((z-x)/t)$. In our application, $p$ is the Poisson kernel and $\kappa$ is a derivative of the Poisson kernel.
The proof can be found below in Section~\ref{ss:boundaryaction}.
\end{enumerate}
\end{example}

\begin{cond}[Conditions on $\Err$]\label{cond:Err}
\[
 \Err = \Err_1 + \Err_2
\]
where 
\begin{itemize} 
\item $\Err_1 \in L^p(\R^2_+)$ for some $p > 1$ and $\Err_1$ has compact support.
\item For any $\Phi \in C^\infty(\overline{\R^2_+})$ with compact support
\[
 \int \Err_2 \Phi \aleq \|\nabla \Phi\|_{2,\R^2_+}
\]
and if $\supp \Phi \subset B(x_0,r)$ for some $x_0 \in \R \times \{0\}$, then
\[
 \int \Err_2 \Phi \aleq r^\sigma\ \|\nabla \Phi\|_{2,\R^2_+}
\]
\end{itemize}
\end{cond}

\begin{cond}[Conditions on $\err$]\label{cond:err}
For $\solU \in H^1(\R^2_+,\R^N)$ and $\solu \in L^\infty H^{\frac{1}{2}}(\R,\R^N)$ and $\solW := \solU - p_t \ast \solu$ we assume that $\err(\solU)$ is so that the following holds.

For any $\eps > 0$, there is an $R \in (0,1)$ so that for any large enough $k_0 \geq 2$ and any $r \in (0,2^{-k_0} R)$ the following holds for some uniform $\sigma > 0$.
For $\varphi \in C_c^\infty(I(x_0,r))$ for any ball $I(x_0,r) \subset (-2,2)$ so that $\|\varphi\|_{\infty,\R} + \|\lapv \varphi\|_{2,\R} \leq 1$, 
\[
\begin{split}
 \int \err(\solU) \varphi \aleq& \eps\, \brac{\|\lapv \solu\|_{(2,\infty),I(x_0,2^{k_0} r)} + \|\nabla \solW\|_{(2,\infty),B(2^{k_0}r,x_0)}}\\
 &+ \tail{\sigma}{\|\lapv \solu\|_{(2,\infty)}}{x_0}{r}{k_0}+\tail{\sigma}{\|\nabla \solW\|_{(2,\infty)}}{x_0}{r}{k_0}.\\
 &+\brac{2^{k_0} r}^\sigma.
 \end{split}
\]
\end{cond}

\section{Antisymmetric nonlocal boundary potentials: Proof of Theorem~\ref{th:transform}}\label{s:proofthtransform}
Let $D \subset \R^2$ be a smoothly bounded domain and $\mathcal{N}$ be a closed manifold in $\R^N$. In order to prove Theorem~\ref{th:transform} we need to transform the following equation
for $\tilde{\solU} \in H^1(D,\R^{N})$ with trace $\tilde{\solu} \in H^{\frac{1}{2}}(\partial D,\mathcal{N})$,
\begin{equation}\label{eq:utildeeq}
 \begin{cases}
\lap \tilde{\solU} = \tilde{\Omega} \cdot \nabla \tilde{\solU} \quad &\mbox{in $D$},\\
\partial_\nu \tilde{\solU} \perp T_{\tilde{\solu}} \mathcal{N} \quad &\mbox{on $\partial D$},
 \end{cases}
\end{equation}
where $\tilde{\Omega}_{ij} = -\tilde{\Omega}_{ji} \in L^2(D,\R^2)$ is antisymmetric.

First, by a standard argument around any point $x_0 \in \partial D$ we can transform the equation into an equation of the half-space.
\begin{lemma}\label{la:transformed}
For any $x_0 \in \partial D$ we find a small radius $\rho > 0$, some $\solU \in H^1(\R^2_+,\R^N)$ and a diffeomorphism $\tau: B(0,1) \to \R^2$ with
\[
\tilde{\solU} = \solU\circ \tau^{-1} \quad \mbox{in $B(x_0,\rho) \cap \overline{D}$},
\]
so that $\solU$ satisfies the following equation
\[
\begin{cases}
 \lap \solU^i  = \Omega_{ij}\cdot \nabla \solU^j + \Err^i\quad& \mbox{in $\R^2_+$},\\
 \solU{(-5,5)\times \{0\}} \subset \mathcal{N}\\
 \partial_\nu\solU \perp T_\solu \mathcal{N} \quad &\mbox{in $(-5,5)\times \{0\}$},\\
 \solU(x) = 0 \quad \mbox{for $|x| > 10$}.
\end{cases}
\]
Here $\Err^i \in L^1 \cap L^p(\R^2_+)$ for any $p < 2$, and $\Err^i$ and $\Omega$ have support in $\overline{B(0,10)^+}$. Moreover, 
and $g \in C^2(\overline{\R^2_+},GL(2))$, and $g_{\alpha\beta} \equiv \delta_{\alpha \beta}$ on $\R \times \{0\}$.
\end{lemma}
\begin{proof}
Fix $x_0 \in \partial D$. Since $\partial D$ is a smooth manifold, there exists a small neighborhood $B(\delta,x_0)$ of $x_0$ in $\R^2$ where the orthogonal projection $\pi_{\partial D}: B(x_0,\delta) \to \partial D$ is well-defined. 
We may assume that we have a parametrization of $\partial D$ around $x_0$. For possibly smaller $\delta$ let $\varphi: [-20,20] \to \R^2$ be this parametrization of $\partial D$ around $x_0$, say of constant speed $|\varphi '| \equiv \sigma$.

Then we have the following diffeomorphism $\tau: [-20,20]^2 \to \R^2$
\[
 \tau(x,t) := \varphi(x) - \sigma\, t\, \nu(\varphi(x)).
\]
We have
\begin{equation}\label{eq:nablatauinSON}
 \frac{1}{\sigma} \nabla \tau (x,0) = (\frac{1}{\sigma}\varphi'(x),\nu(\varphi(x)) \in  SO(N).
\end{equation}
Also $\tau$ maps the upper cylinder $[-20,20]\times (0,20)$ into $D$ and the lower cylinder $[-20,20]\times (-20,0)$ is mapped into $\R^2 \backslash D$.
Finally we have
\begin{equation}\label{eq:pttau}
 \partial_t \tau(x,t) \Big |_{t = 0} = -\sigma \nu(\varphi(x)).
\end{equation}
Set $\solU := \eta\, \tilde{\solU} \circ \tau$ for a smooth cutoff function $\eta \in C_c^\infty([-10,10]^2)$ with $\eta \equiv 1$ on $[-7,7]^2$. 

Let 
\[
 A_{\alpha \beta} := \sigma\ \partial_\alpha \brac{\tau^{-1}}^\beta \circ \tau,
\]
then from
\[
\sigma(\partial_\alpha f) \circ \tau = A_{\alpha \beta}\, \partial_\beta (f \circ \tau)
\]
we find
\[
 A_{\alpha \gamma} \partial_\gamma \brac{A_{\alpha \beta}  \partial_\beta (\tilde{\solU} \circ \tau)} = \sigma^2\brac{\lap \tilde{\solU}}\circ \tau.
\]
Thus we have from \eqref{eq:utildeeq}
\[
 A_{\alpha \gamma} \partial_\gamma \brac{A_{\alpha \beta}  \partial_\beta (\tilde{\solU}^i \circ \tau)} = \sigma\, \tilde{\Omega}_{ij}^\alpha\circ \tau\ A_{\alpha \beta}\, \partial_\beta \brac{\tilde{\solU}^j\circ \tau}\quad \mbox{in $(-20,20)\times (0,20)$}.
\]
Multiplying this with $\eta$ and using the product rule we find
\[
\partial_\gamma \brac{A_{\alpha \gamma} A_{\alpha \beta}  \partial_\beta \solU^i}= \Omega_{ij}\cdot \tilde{\solU}^j+ \Err^i_1\quad \mbox{in $\R^2_+$}.
\]
for 
\[
 \Omega_{ij}^\beta := \chi_{B(0,10)}\, \sigma A_{\alpha \beta}\, \tilde{\Omega}_{ij}^\alpha\circ \tau \in L^2(\R^2_+),
\]
which satisfies $\Omega_{ij} = -\Omega_{ji}$, and where pointwise a.e.
\[
 |\Err_1| \aleq \chi_{B(0,10)}\ \brac{|\tilde{\solU}\circ \tau| + |\nabla \tilde{\solU}\circ \tau| + |\tilde{\Omega} \circ \tau|\ |\tilde{\solU} \circ \tau|}
\]
Setting
\[
 \Err_2^i := \lap \solU^i-\partial_\gamma \brac{A_{\alpha \gamma} A_{\alpha \beta}  \partial_\beta \solU^i},
\]
we found
\[
 \lap \solU^i= \Omega_{ij}\cdot \tilde{\solU}^j+ \Err^i_1+\Err^i_{2}\quad \mbox{in $\R^2_+$}.
\]
In order to show that $\Err^i_{2}$ satisfies the condition~\ref{cond:Err}, observe that from \eqref{eq:nablatauinSON} we have 
\[
 A_{\alpha \gamma} A_{\alpha \beta} - \delta_{\gamma \beta} = 0 \quad \mbox{on $(-20,20) \times \{0\}$}
\]
and the support we can assume that this holds on $\R \times \{0\}$ without changing the equation for $\solU$, because of the support of $\solU$.
Thus, for any $\Phi \in C_c^\infty(\R^2)$ the boundary terms vanish in the following integration by parts,
\[
 \int_{\R^2_+} \Err_2^i\ \Phi = \int_{\R^2_+} \brac{A_{\alpha \gamma} A_{\alpha \beta} - \delta_{\gamma \beta}} \partial_\beta \solU^i\ \partial_\gamma \Phi.
\]
Consequently, for any $\Phi \in C_c^\infty(B(x_0,r))$ for some $x_0 \in \R \times \{0\}$
\[
 \int_{\R^2_+} \Err_2^i\ \Phi \aleq \|A_{\alpha \gamma} A_{\alpha \beta} - \delta_{\gamma \beta}\|_{L^\infty(B(x_0,r)}\ \|\nabla \solU\|_{2,\R^2_+}\ \|\nabla \Phi\|_{2,\R^2_+}
\]
But since $A_{\alpha \gamma} A_{\alpha \beta} - \delta_{\gamma \beta} = 0$ on $\R \times \{0\}$ we have
\[
 \|A_{\alpha \gamma} A_{\alpha \beta} - \delta_{\gamma \beta}\|_{L^\infty(B(x_0,r)} \aleq r\ \|\nabla \Phi\|_{2,\R^2_+}.
\]
Thus $\Err_2^i$ satisfies condition~\ref{cond:Err}.

Finally from \eqref{eq:pttau} and the fact that $\solU = \tilde{\solU} \circ \tau$ on $[-7,7] \times [0,1)$,
\[
 -\partial_t \solU = \sigma \partial_\nu \tilde{\solU} \circ \tau \perp T_{\solU} \mathcal{N} \quad \mbox{on }(-5,5) \times \{0\}.
\]
\end{proof}
From Lemma~\ref{la:transformed} we obtained a compactly supported $\solU \in H^1(\R^2_+,\R^N)$ with trace $\solu \in H^{\frac{1}{2}}(\R,\R^N)$ satisfying \eqref{eq:trU}. 

It remains to compute \eqref{eq:transformedeqV}. For this we denote again with $\solV$ the harmonic Poisson extension to $\R^2_+$ of $\solu$,
\[
 \solV(x,t) := p_t \ast \solu(x),
\]
that is the solution to
\[
\begin{cases}
 \lap \solV = 0\quad& \mbox{in $\R^2_+$}\\
 \solV = \solu&\mbox{in $\R \times \{0\}$}\\
 \lim_{|x| \to \infty} V(x) = 0.
\end{cases}
\]
In view of the Dirichlet-to-Neumann property \eqref{eq:dtn} for \eqref{eq:transformedeqV} we need to show
\begin{equation}\label{eq:trafoVgoal}
 \partial_\nu \solV^i = \omega_{ij}(\lapv \solu^j) + \err^i(\solU) \quad (-2,2) \times \{0\}.
\end{equation}
Since we only have information about $\partial_\nu \solU$, we also introduce $\solW := \solU - \solV$ which solves
\begin{equation}\label{eq:Weq}
\begin{cases}
 \lap \solW^i =  \Omega_{ij}\cdot \nabla \solU^j+\Err^i(\solU)& \mbox{in $\R^2_+$}\\
 \solW = 0 \quad &\mbox{on $\R \times \{0\}$}\\
\lim_{|x| \to \infty} \solW(x) = 0.
\end{cases}
\end{equation}
By $\Pi: \mathcal{N} \to \R^{N \times N}$ we denote the projection on the tangent plane of $\mathcal{N}$, that is for $\solu \in \mathcal{N}$ we have a matrix $\Pi(\solu) \in \R^{N \times N}$ which is symmetric, $\Pi(\solu)^2 = \Pi(\solu)$ and 
\[
 \operatorname{im}(\Pi(\solu)) = T_\solu \mathcal{N}, \quad \ker(\Pi(\solu)) = (T_{\solu} \mathcal{N})^\perp.
\]
By the nearest point projection $\pi$ from a tubular neighborhood $B_\delta(\mathcal{N})$ into $\mathcal{N}$ we can assume $\Pi$ to be defined first in this small neighborhood and then extended to all of $\R^N$, $\Pi : \R^N \to \R^{N\times N}$ and $\Pi \in C^\infty \cap W^{1,\infty}(\R^N,\R^N\times N)$. W.l.o.g. $\Pi(0) = 0$. With $\Pi^\perp(\solu)$ we denote $I-\Pi(\solu)$ for the identity matrix $I \in \R^{N \times N}$.
 
From the condition $\partial_\nu \solU \perp T_\solu \mathcal{N}$ in $(-5,5) \times \{0\}$ we then have
\begin{equation}\label{eq:solueq}
\laph \solu \equiv \partial_\nu \solV = \partial_\nu (\solU-\solW) =  \Pi^\perp(\solu) \partial_\nu \solV - \Pi(\solu)\partial_\nu \solW \quad \mbox{in $(-2,2) \times \{0\}$}.
\end{equation}
The first term $\Pi^\perp(\solu) \partial_\nu \solV = \Pi^\perp(\solu) \laph \solu$ is essentially known from the theory of half-harmonic maps into manifolds. We will treat it in Section~\ref{ss:halfharmpart}.
The second term $\Pi(\solu)\partial_\nu \solW$ takes into account the interior equation \eqref{eq:trU}. It involves the antisymmetric action \[\Omega \cdot \nabla \solU = \Omega \cdot \nabla \solW + \Omega \cdot \nabla \solV.\]
The interior action part $\Omega \cdot \nabla \solW$ can essentially be estimated as the antisymmetric system treated in Rivi\`{e}re's celebrated \cite{Riviere-2007}. We will treat the interior action part in Section~\ref{ss:interioraction}. The remaining part, the boundary action part $\Omega \cdot \nabla \solW$, induces a (nonlocal) antisymmetric potential acting on the trace $\solu$. We will treat the boundary action part in Section~\ref{ss:boundaryaction}. 
\subsection{The half-harmonic map part, $\Pi^\perp(\solu) \partial_\nu \solV$}\label{ss:halfharmpart}
We begin with the first term \[\Pi^\perp(\solu) \partial_\nu \solV \equiv \Pi^\perp(\solu) \laph \solu.\] An antisymmetric structure from this term was first derived by Da Lio and Rivi\`{e}re \cite{DaLio-Riviere-1Dmfd} who studied the equation
\[
 \laph \solu \perp T_\solu \mathcal{N} \quad \mbox{in $\R$},
\]
or equivalently,
\[
 \laph \solu = \Pi^\perp(\solu) \laph \solu.
\]
There are several antisymmetric potentials that can be derived for this equation, the one that was found in \cite{DaLio-Riviere-1Dmfd}, see also \cite{DaLio-Riviere-CAG}, or a nonlocal one as in \cite{Mazowiecka-Schikorra-2017}. Here we have a slightly different one from all of those.
\begin{lemma}\label{la:piperplaphu}
For $\varphi \in C_c^\infty((-2,2))$,
\[
 \int_{\R}\ \varphi\ \Pi^\perp(\solu)\ \partial_\nu \solV = \int_{\R}\int_{\R}  \varphi(x)\, \omega^1(x,y)\, \lapv \solu(y)\  dx\ dy + \err(\solU)
\]
where $\omega^1$ satisfies conditions~\ref{cond:omega} and $\err$ satisfies conditions~\ref{cond:Err}.
\end{lemma}
\begin{proof}
We have
\[
 \Pi^\perp(\solu) \laph \solu= \Pi^\perp(\solu) \lapv \brac{\Pi(\solu) \lapv \solu} + \err_1(\solU)
\]
where 
\[
 \err_1(\solU) := \Pi^\perp(\solu) \lapv \brac{\Pi^\perp(\solu) \lapv \solu}.
\]
The fact that $\err_1(\solU)$ satisfies the conditions~\ref{cond:Err}, follows from known arguments, see \cite{DaLio-Riviere-1Dmfd,Schikorra-2012,Blatt-Reiter-Schikorra-2016}. Essentially one localizes the following estimates. We only mention the global steps, and skip the details. Firstly,
\[
\begin{split}
 \int \err_1(\solU)  \varphi =& \int \brac{\Pi^\perp(\solu) \lapv \solu} \lapv (\Pi^\perp(\solu) \varphi) \aleq \|\Pi^\perp(\solu) \lapv \solu\|_{2,\R}\ \|\lapv (\Pi^\perp(\solu) \varphi)\|_{2}.
\end{split}
 \]
Since $\solu$ maps into a manifold, in view of Lemma~\ref{la:Piperplapvu},
\[
  \|\Pi^\perp(\solu) \lapv \solu\|_{2,(-2,2)} \aleq \|\lapv \solu\|_{(2,\infty),\R}\ \|\lapv \solu\|_{2,\R}.
\]
Moreover,
\[
\begin{split}
 \|\lapv (\Pi^\perp(\solu) \varphi)\|_{2} \aleq& \brac{\|\Pi(\solu)\|_{\infty} + \|\lapv \Pi^\perp(\solu)\|_{2}}\ \brac{\|\varphi\|_{\infty} + \|\lapv \varphi\|_{2}}\\
 \aleq& \brac{1 + \|\lapv \solu\|_{2}}\ \brac{\|\varphi\|_{\infty} + \|\lapv \varphi\|_{2}}.
\end{split}
 \]
Working with cutoff arguments and using the pseudo-locality of the fractional Laplacian, see, e.g., \cite[Lemma A.1]{Blatt-Reiter-Schikorra-2016}, one obtains
\[
\begin{split}
 \int \err_1  \varphi \aleq& \brac{\|\lapv \solu\|_{2,I(x_0,2^{k_0}r)} \|\lapv \solu\|_{(2,\infty),I(x_0,2^{k_0}r)} + \tail{\sigma}{\|\lapv \solu\|_{(2,\infty)}}{x_0}{R}{k_0} } \\
 &\quad \cdot \brac{1 + \|\lapv \solu\|_{2}}\ \brac{\|\varphi\|_{\infty} + \|\lapv \varphi\|_{2}}
\end{split}
 \]
Thus, for any ball in $(-2,2)$ with radius $2^{k_0}r$ small enough so that
\[
 \|\lapv \solu\|_{2,I(x_0,2^{k_0}r)} \brac{1 + \|\lapv \solu\|_{2}} \leq \eps
\]
we have the estimate as required for condition~\ref{cond:Err}.

For the remaining term, we denote by
\[
 H_{\frac{1}{2}} (a,b) := \lapv (a b) - a\, \lapv b-b\, \lapv a.
\]
Then, since $\Pi^\perp(\solu) \Pi(\solu) \equiv 0$,
\[
\begin{split}
 &\int \varphi\ \Pi^\perp(\solu) \lapv \brac{\Pi(\solu) \lapv \solu}\\
 =&\int \lapv (\varphi\ \Pi^\perp(\solu)) \brac{\Pi(\solu) \lapv \solu}\\
 =&\int (\varphi\ \lapv \Pi^\perp(\solu)) \brac{\Pi(\solu) \lapv \solu} + \int H_{\frac{1}{2}}(\varphi, \Pi^\perp(\solu)) \brac{\Pi(\solu) \lapv \solu}\\
 =:&\int (\varphi\ \lapv \Pi^\perp(\solu)) \brac{\Pi(\solu) \lapv \solu} + \err_2.
\end{split}
 \]
The error term $\err_2$ again satisfies condition~\ref{cond:err} as above.
Now we go into coordinates, and find for
\[
 \tilde{\omega}_{ik} := \lapv \Pi^\perp_{ij}(\solu) \ \Pi(\solu)_{jk} - \lapv \Pi^\perp_{kj}(\solu) \ \Pi(\solu)_{ji} \in L^2(\R^2),
\]
and
\[
 \err_3 := \Pi(\solu)_{ij}  \lapv \Pi^\perp_{jk}(\solu) \ \lapv \solu^k
\]
the representation
\[ 
 \lapv \Pi^\perp_{ij}(\solu) \ \Pi(\solu)_{jk} \lapv \solu^k = \tilde{\omega}_{ik}\, \lapv \solu^k + \err_3.
 \]
Clearly 
\[
 \omega_{ik}(x,y) := \tilde{\omega}_{ik}(x) \delta_{x,y}
\]
is antisymmetric and satisfies the conditions \ref{cond:omega}.

As for the last error term $\err_3$, again since $\Pi^\perp(\solu) \Pi(\solu) \equiv 0$,
\[
  \err_3 = -H_{\frac{1}{2}}(\Pi_{ij}(\solu),\Pi^\perp(\solu)_{jk})\ \lapv \solu^k - \lapv \Pi(\solu)_{ij}\ \Pi^\perp(\solu)_{jk} \lapv \solu.
\]
From the three-commutator estimates, see \cite[Theorem 7.1.]{Lenzmann-Schikorra-commutators},
\[
 \|\err_3\|_{L^1(\R)} \aleq \|\lapv \Pi(\solu)\|_{2}^2\ \|\lapv \solu\|_{(2,\infty)} + \|\lapv \Pi(\solu)_{ij}\|_{2}\ \|\Pi^\perp(\solu)_{jk} \lapv \solu\|_{2}.
\]
With the help of Lemma~\ref{la:Piperplapvu} and a suitable localization we find that $\err_3$ satisfies the conditions~\ref{cond:err}.
\end{proof}
\subsection{The interior action $\Omega \cdot \nabla \solW$}\label{ss:interioraction}
It remains to reformulate for $\varphi \in C_c^\infty((-2,2))$
\[
\int_{\R} \Pi_{\ell j}(\solu)\, \partial_\nu \solW^j\ \varphi.
\]
Observe that the antisymmetric potential in \eqref{eq:Weq}, $\Omega \cdot \nabla \solU$, acts on $\solU=\solV + \solW$. So we have to control the interior action of $\Omega$, namely $\Omega \cdot \nabla \solW$, and the boundary action of $\Omega$, namely $\Omega \cdot \nabla \solV$.

Clearly, the interior action $\Omega \cdot \nabla \solW$ can, in general, \emph{not} be represented as an antisymmetric potential of the boundary data $\solu$. It is a purely interior object. 

Since we are in the process to find a reformulation for $\partial_\nu \solV$ on the boundary, i.e., we are reformulating the boundary equation, we will see that this seemingly critical term $\Omega \cdot \nabla \solW$ is actually subcritical with respect to its influence on the boundary. Note that we did not choose any gauge on the boundary yet, so this subcriticality might be surprising at first. The reason is that we can choose an \emph{interior} gauge which only transforms the interior equation and does not touch the boundary. This is the subject of this section.

The remaining action $\Omega \cdot \nabla \solV$ involves the boundary data $\solu$ and has an antisymmetric structure, and we will treat it in Section~\ref{ss:boundaryaction}.

From Proposition~\ref{pr:gaugenormal} we find an interior gauge adapted to $\Omega \in L^2(\R^2_+,so(N)\otimes \R^2)$, namely $P \in \dot{H}^1(\R^2_+,SO(N))$, $P \equiv I$ on $\R \times \{0\}$ satisfying
\begin{equation}\label{eq:V:Pchoice}
 \dv (\nabla P\, P^T + P \Omega P^T) = 0 \quad \mbox{in $\R^2_+$}.
\end{equation}
Observe that $\Pi(\solu) \in L^2(\R)$ since $\solu \in H^{\frac{1}{2}}\cap L^\infty(\R,\R^n)$, $\solu$ has compact support, and $\Pi(0) = 0$. Denote by $\Pi^h(\solu)$ the harmonic Poisson-extension of $\Pi(\solu)$ to $\R^2_+$, and with $\varphi^h$ the harmonic Poisson-extension of $\varphi \in C_c^\infty((-2,2))$ to $\R^2_+$.  An integration by parts with boundary data in $\R^2_+$ gives 
\[
\begin{split}
 \int_{\R} \Pi_{\ell j}(\solu)\, \partial_\nu \solW^j\ \varphi \equiv&\int_{\R \times \{0\}} \Pi^h_{\ell k}(\solu)\, P_{k j}\, \, \partial_\nu \solW^j\ \varphi^h \\
=&\int_{\R^2_+} \nabla \solW^j\cdot \nabla \brac{ \Pi^h_{\ell k}(\solu)\, P_{k j}\ \varphi^h} + \int_{\R^2_+} P_{k j}\lap_{\R^2} \solW^j\ \Pi^h_{\ell k}(\solu)\, \ \varphi^h.
\end{split}
\]
By the equation for $\solW$, \eqref{eq:Weq}, 
\[
\begin{split}
\int_{\R} \Pi_{\ell j}(\solu)\, \partial_\nu \solW^j\ \varphi=&\int_{\R^2_+} \nabla \solW^j\cdot \nabla \brac{ \Pi^h_{\ell k}(\solu)\, P_{k j}\ \varphi^h} + \int_{\R^2_+} P_{k j} \Omega_{ji}\cdot \nabla \solU^i\ \Pi^h_{\ell k}(\solu)\, \ \varphi^h\\
&+ \int_{\R^2_+} P_{k j} \Err^j(\solU) \ \Pi^h_{\ell k}(\solu)\, \ \varphi^h.
\end{split}
\]
With $\solU = \solV + \solW$ we arrive at
\[
\begin{split}
\int_{\R} \Pi_{\ell j}(\solu)\, \partial_\nu \solW^j\ \varphi=&\int_{\R} \brac{\err^\ell_1(\solU) + \err^\ell_2(\solU) + \err^\ell_3(\solU)+\err^\ell_4(\solU)}\, \varphi + \int_{\R^2_+} \Omega_{ji} \nabla \solV^i\ \Pi^h_{\ell k}(\solu)\, \ \varphi^h\\
 \end{split}
 \]
for $\err^\ell_1(\solU), \ldots,\err^\ell_4(\solU)$ distributions defined as
\[
 \begin{split}
\int_{\R} \err^\ell_1(\solU)\ \varphi &:= \int_{\R^2_+} \brac{\nabla P_{k i} + P_{k j} \Omega_{ji} }\cdot \nabla \solW^i \ \Pi^h_{\ell k}(\solu)\ \varphi^h,\\
\int_{\R} \err^\ell_2(\solU)\ \varphi &:= \int_{\R^2_+} P_{k j} \nabla \solW^j\cdot \nabla \brac{ \Pi^h_{\ell k}(\solu)\, \ \varphi^h},\\
\int_{\R} \err^\ell_3(\solU)\ \varphi &:= \int_{\R^2_+} P_{k j}\, \Err^j\ \Pi^h_{\ell k}(\solu)\, \ \varphi^h,\\
\int_{\R} \err^\ell_4(\solU)\ \varphi &:= \int_{\R^2_+} \brac{P_{k j}-\delta_{kj}}\, \Omega_{ji} \nabla \solV^i\ \Pi^h_{\ell k}(\solu)\, \ \varphi^h.\\
 \end{split}
\]
The antisymmetric boundary action term
\[
 \int_{\R^2_+} \Omega_{ji} \nabla \solV^i\ \Pi^h_{\ell k}(\solu)\, \ \varphi^h
\]
will be treated in Section~\ref{ss:boundaryaction}.

The following Lemmata~\ref{la:V:err1}, \ref{la:V:err2}, \ref{la:V:err3}, \ref{la:V:err4} show that $\err^\ell_1(\solU), \ldots,\err^\ell_4(\solU)$ satisfy the conditions~\ref{cond:err}.
\begin{lemma}\label{la:V:err1}
For any $\eps > 0$ there exists $R \in (0,1)$ so that whenever $2^{k_0} r < R$ for some $k_0 \in \N$: For any $x_0 \in (-1,1)$, $r \in (0,1)$ so that $I(x_0,r) \subset (-2,2)$ and for any $\varphi \in C_c^\infty(I(x_0,r))$,
\[
\begin{split}
\int_{\R} \varphi\ \err_1(\solU) \aleq&\ \eps\ \|\nabla \solW\|_{(2,\infty),B(2^{k_0}r,x_0)}
(\|\varphi\|_{\infty} + \|\lapv \varphi\|_{2})\\
&+ \tail{\sigma}{\|\nabla \solW\|_{(2,\infty)}}{x_0}{r}{k_0}\ (\|\varphi\|_{\infty} + \|\lapv \varphi\|_{2}).
\end{split}
\]
\end{lemma}
\begin{proof}
Let 
\[
 \Omega^P := \nabla P\, P^T + P \Omega P^T.
\]
Then we have
\[
\int_{\R} \err^\ell_1(\solU)\ \varphi = \int_{\R^2_+} \Omega^P_{ki} P_{ij}\cdot \nabla \solW^j \ \Pi^h_{\ell k}(\solu)\ \varphi^h,\\
\]
Integrating by parts, using that $\solW = 0$ on $\R \times \{0\}$,
\[
 \int_{\R} \err^\ell_1(\solU)\ \varphi = -\int_{\R^2_+} \Omega^P_{ki} \cdot \nabla \brac{P_{ij}\Pi^h_{\ell k}(\solu)\ \varphi^h}\ \solW^j.\\
\]
Since $\solW = 0$ on $\R \times \{0\}$ we may apply the div-curl lemma on the upper halfplane, Theorem~\ref{th:clmshalfspace}. Before we do so, observe that by Lemma~\ref{la:harmextest} for any large enough $k \geq k_0$,
\[
  \|\nabla  \brac{P\, \Pi^h(\solu)\, \varphi^h}\|_{2, B(x_0,2^{k+5} r) \backslash B(x_0,2^{k-5} r)} \aleq 2^{-k}\ (1+\|\nabla P\|_{2,\R^2_+}+\|\lapv \solu\|_{2,\R})\, \|\varphi\|_{\infty}.
\]
Moreover, by the estimates for the Poisson extension, $\|\varphi^h\|_{\infty,\R^2_+} \leq \|\varphi\|_{\infty}$ and $\|\nabla \varphi^h\|_{2,\R^2_+} \aleq \|\lapv \varphi \|_{2,\R}$, and thus
\[
  \|\nabla  \brac{P\, \Pi^h(\solu)\, \varphi^h}\|_{2, B(x_0,2^{k_0+5} r) } \aleq (1+\|\nabla P\|_{2,\R^2_+}+\|\lapv \solu\|_{2,\R})\, \brac{\|\varphi\|_{\infty}+\|\lapv \varphi\|_{2}}.
\]
Thus Theorem~\ref{th:clmshalfspace} implies,
\[
\begin{split}
 \int_{\R} \err^\ell_1(\solU)\ \varphi \aleq& \|\Omega^P_{ki}\|_{2,B(x_0,2^{k_0} r)}\ \|\nabla \solW\|_{(2,\infty),B(x_0,2^{k_0} r)}\  (\|\varphi\|_{\infty} + \|\lapv \varphi\|_{2})\\
 &+ \sum_{k = k_0}^\infty k\, 2^{-k}\,  \|\nabla \solW\|_{(2,\infty),B(x_0,2^{k} r)}\  (\|\varphi\|_{\infty} + \|\lapv \varphi\|_{2})
\end{split}
 \]
Observe that the constants may depend on $P$, $\Omega$, and $\solU$, but not on the radius $r>0$ the point $x_0 \in \R \times \{0\}$ or $k_0$. In particular, by absolute continuity of the integral, we find some $R > 0$ so that 
\[
  \sup_{x_0 \in \R \times \{0\}} \|\Omega^P_{ki}\|_{2,B(x_0,R)} < \eps.
\]
The lemma is proven.
\end{proof}

\begin{lemma}\label{la:V:err2}
For any $\eps > 0$ there exists $R \in (0,1)$ so that whenever $2^{k_0} r < R$: For any $x_0 \in (-1,1)$, $r \in (0,1)$ and $\varphi \in C_c^\infty(I(x_0,r))$,
\[
\begin{split}
 \int_{\R} \varphi\ \err_2(\solU)=
\aleq & \eps \|\nabla \solW\|_{(2,\infty),B(2^{k_0}r,x_0)}
(\|\varphi\|_{\infty} + \|\lapv \varphi\|_{2})\\
&+ \tail{\sigma}{\|\nabla \solW\|_{(2,\infty)}}{x_0}{r}{k_0}\ (\|\varphi\|_{\infty} + \|\lapv \varphi\|_{2})
 \end{split}
 \]
\end{lemma}
\begin{proof}
Since $\solW=0$ on $\R^1 \times \{0\}$, and $\varphi^h$ and $\Pi^h(\solu)$ are harmonic, integration by parts yields
\[
\int_{\R} \varphi\ \err_2(\solU)= -\int_{\R^2_+} \nabla \brac{\Pi^h_{\ell k}(\solu) P_{k j}} \cdot \nabla \varphi^h\ \solW^j -\int_{\R^2_+} \nabla \brac{\varphi^h P_{k j} } \cdot \nabla \Pi^h_{\ell k}(\solu)\ \solW^j  
\]
Observe that $\div(\nabla \varphi^h) = 0$ and $\div \nabla \Pi^h_{\ell k}(\solu) = 0$. We can then apply Theorem~\ref{th:clmshalfspace} and conclude as for Lemma~\ref{la:V:err1}.
\end{proof}

\begin{lemma}\label{la:V:err3}
For a uniform $\sigma > 0$ the following holds for any $x_0 \in \R$, any $r \in (0,\frac{1}{2})$ and any $\varphi \in C_c^\infty(I(x_0,r))$
\[
 \int_{\R} \err^\ell_3(\solU)\ \varphi \aleq r^{\sigma}\, \brac{\|\varphi\|_{L^\infty(\R)} + \|\lapv \varphi\|_{L^2(\R)}}\\
\]
\end{lemma}
\begin{proof}
We have to consider two cases. Firstly, assume that $\Err^j \in L^1\cap L^{\frac{3}{2}}(\R^2_+)$, then  
\[
 \int_{\R^2_+} P\, \Err\ \Pi^h(\solu)\, \ \varphi^h \aleq \sqrt{r}^{\frac{2}{3}}\ \|\varphi^h\|_{L^{\infty}(\R^2_+)}+\|\varphi^h\|_{L^{\infty}(\R^2_+ \backslash B^+(x_0,\sqrt{r}))}
\]
By Lemma~\ref{la:harmextest},
\[
\|\varphi^h\|_{L^{\infty}(\R^2_+ \backslash B^+(x_0,\sqrt{r}))} \aleq \sqrt{r}^{-1} \|\varphi\|_{L^1(\R)} \aleq r^{\frac{1}{2}}\ \|\varphi\|_{\infty}.
\]
Secondly, let us assume that for any $\Phi \in C^\infty(\overline{\R^2_+})$ with compact support
\[
 \int \Err_2 \Phi \aleq \|\nabla \Phi\|_{2,\R^2_+}
\]
and if $\supp \Phi \subset B(x_0,r)$ for some $x_0 \in \R \times \{0\}$, then
\[
 \int \Err_2 \Phi \aleq r\ \|\nabla \Phi\|_{2,\R^2_+}.
\]
Let $\eta_{B(x_0,\sqrt{r})} = \eta((\cdot-x_0)/\sqrt{r})$ for the usual bump-function $\eta$. Then
\[
 \int_{\R^2_+} P\, \Err\ \Pi^h(\solu)\, \ \varphi^h \aleq \sqrt{r} \|\nabla \brac{\eta_{B(x_0,\sqrt{r})} P\,\Pi^h(\solu)\, \varphi^h}\|_{2,\R^2_+}\ + \|\nabla \brac{\brac{1-\eta_{B(x_0,\sqrt{r})}} P\,\Pi^h(\solu)\, \varphi^h}\|_{2,\R^2_+}.
\]
Now
\[
 \|\nabla \brac{\eta_{B(x_0,\sqrt{r})} P\,\Pi^h(\solu)\, \varphi^h}\|_{2,\R^2_+} \aleq \|\varphi^h\|_{\infty} + \|\nabla \varphi^h\|_{2} \aleq \|\varphi\|_{\infty} + \|\lapv \varphi\|_{2}.
\]
Moreover, in view of Lemma~\ref{la:harmextest}, since $r < \frac{1}{2}$ and thus $\sqrt{r} > r$,
\[
\|\nabla \brac{\brac{1-\eta_{B(x_0,\sqrt{r})}} P\,\Pi^h(\solu)\, \varphi^h}\|_{2,\R^2_+} \aleq r^{-\frac{1}{2}}\ \|\varphi\|_{1} \aleq \sqrt{r} \|\varphi\|_{\infty}
\]
The constants depend on the fixed values $\|\lapv \solu\|_{2}$ and $\|\nabla P\|_{2}$. 
\end{proof}

\begin{lemma}\label{la:V:err4}
For any $\eps > 0$ there exists $R \in (0,1)$ so that whenever $2^{k_0} r < R$: For any $x_0 \in (-1,1)$, $r \in (0,1)$ and $\varphi \in C_c^\infty(I(x_0,r))$,
\[
\begin{split}
\int_{\R} \err^\ell_4(\solU)\ \varphi\aleq & \eps \brac{\|\lapv \solu\|_{(2,\infty),B(2^{k_0}r,x_0)}}
(\|\varphi\|_{\infty} + \|\lapv \varphi\|_{2})\\
&+ \tail{\sigma}{\|\lapv \solu\|_{(2,\infty)}}{x_0}{r}{k_0}\ (\|\varphi\|_{\infty} + \|\lapv \varphi\|_{2})
 \end{split}
 \]
\end{lemma}
\begin{proof}
By Hodge decomposition on $\R^2$, we find $F_{ji},G_{ji} \in \dot{H}^1(\R^2)$ so  that
\[
 \Omega_{ji}\chi_{\R^2_+} = \nabla G_{ji} + \nabla^\perp F_{ji}.
\]
Here
\[
 \nabla^\perp = \left ( \begin{array}{c} -\partial_2\\\partial_1 \end{array} \right ).
\]
Since $\brac{P_{k j}-\delta_{kj}} = 0$ on $\R \times \{0\}$ we can apply Theorem~\ref{th:clmshalfspace} to the term
\[
\int_{\R^2_+} \brac{P_{k j}-\delta_{kj}}\, \nabla^\perp F_{ji}\cdot \nabla \solV^i\ \Pi^h_{\ell k}(\solu)\, \ \varphi^h = \int_{\R^2_+} \nabla \brac{\brac{P_{k j}-\delta_{kj}}\, \Pi^h_{\ell k}(\solu)\, \ \varphi^h} \cdot \nabla^\perp F_{ji}\ \solV^i.
\]
As in in the proof of Lemma~\ref{la:V:err1}, additionally using Proposition~\ref{pr:VtBMO}, we then find
\[
\begin{split}
 &\int_{\R^2_+} \brac{P_{k j}-\delta_{kj}}\, \nabla^\perp F_{ji}\cdot \nabla \solV^i\ \Pi^h_{\ell k}(\solu)\, \ \varphi^h\\
 \aleq& 
 \|\nabla F\|_{2,B(x_0,2^{k_0} r)}\ \|\lapv \solu\|_{(2,\infty),B(x_0,2^{k_0} r)}\  (\|\varphi\|_{\infty} + \|\lapv \varphi\|_{2})\\
 &+ \tail{\sigma}{\|\lapv \solu\|_{(2,\infty)}}{x_0}{r}{k_0}\ (\|\varphi\|_{\infty} + \|\lapv \varphi\|_{2}).
\end{split}
 \]
For suitably small $r >0$ we conclude as in the proof of Lemma~\ref{la:V:err1}.

For the remaining term
\[
\int_{\R^2_+} \brac{P_{k j}-\delta_{kj}}\, \nabla G_{ji}\cdot \nabla \solV^i\ \Pi^h_{\ell k}(\solu)\, \ \varphi^h 
\]
we observe that $\nabla \solV = \nabla^\perp \tilde{\solV}$, since $\nabla \solV$ is divergence free. Indeed, one can compute that $\tilde{\solV} = p_t \ast \Hz \solu$, where $\Hz$ denotes the Hilbert transform of $\solu \in L^2\cap L^\infty(\R)$.
We apply Theorem~\ref{th:clmshalfspace}, and as above with the help of Proposition~\ref{pr:VtBMO} obtain
\[
\begin{split}
 &\int_{\R^2_+} \brac{P_{k j}-\delta_{kj}}\, \nabla G_{ji}\cdot \nabla \solV^i\ \Pi^h_{\ell k}(\solu)\, \ \varphi^h \\
 \aleq& 
 \|\nabla G\|_{2,B(x_0,2^{k_0} r)}\ \|\Hz\lapv \solu\|_{(2,\infty),B(x_0,2^{k_0} r)}\  (\|\varphi\|_{\infty} + \|\lapv \varphi\|_{2})\\
 &+ \tail{\sigma}{\|\Hz\lapv \solu\|_{(2,\infty)}}{x_0}{r}{k_0}\ (\|\varphi\|_{\infty} + \|\lapv \varphi\|_{2}).
\end{split}
 \]
Finally, to remove the Hilbert transform $\Hz$, we apply Proposition~\ref{pr:localHilbert} and obtain the claim.
\end{proof}

\subsection{The boundary action $\Omega \cdot \nabla \solV$}\label{ss:boundaryaction}
From the above considerations we have arrived at the following. 
\begin{equation}\label{eq:boundaryactionstart1}
 \int_{\R} \Pi_{\ell j}(\solu)\, \partial_\nu \solW^j\ \varphi = \int_{\R^2_+} \Pi^h_{\ell j}(\solu)\, \Omega_{ji}\cdot \nabla \solV^i \,  \varphi^h+ \int_{\R}\err\ \varphi,
\end{equation}
for some $\err$ satisfying the conditions~\ref{cond:err}.

Since $\solV$ is the harmonic extension of the boundary values $\solu$, the expression $\int_{\R^2_+} \Omega_{ji}\cdot \nabla \solV^i\ \varphi^h$ would be a nonlocal antisymmetric potential acting on $\solu$. However, we have the additional symmetric term $\Pi^h_{\ell j}(\solu)$.

Denote by $(\Pi^\perp(\solu))^h$ the harmonic Poisson-extension of $\Pi^\perp(\solu)=I-\Pi(\solu)$. Then
\begin{equation}\label{eq:boundaryactionstart2}
 \int_{\R^2_+} \Pi^h_{\ell j}(\solu)\, \Omega_{ji}\cdot \nabla \solV^i \,  \varphi^h = \int_{\R^2_+} \Pi^h_{\ell j}(\solu)\, \Omega_{ji}\,  \Pi^h_{ik}(\solu) \cdot\nabla \solV^k \,  \varphi^h + \int_{\R} \err_5 \varphi,
\end{equation}
where
\[
 \int_{\R} \err_5 \varphi := \int_{\R^2_+} \Pi^h_{\ell j}(\solu)\, \Omega_{ji}\, (\Pi^\perp_{ik}(\solu))^h\cdot \nabla \solV^k\,\varphi^h.
\]
The following proposition shows that $\err_5$ satisfies the conditions~\ref{cond:err}.

Observe that there is no reason to believe that $\Pi^\perp_{ik}(\solu) \nabla \solV = 0$ even close to the boundary $\R \times \{0\}$. What we know by Lemma~\ref{la:Piperplapvu} is that $\Pi^\perp_{ik}(\solu) \lapv \solu$ is well behaved. So our strategy for showing the following lemma is that up to commutators, which are in the realm of Theorem~\ref{th:extensioncomm}, $\Pi^\perp_{ik}(\solu) \nabla \solV$ is somehow comparable to $\Pi^\perp_{ik}(\solu) \lapv \solu$.
\begin{proposition}\label{pr:V:err5}
For any $\eps > 0$ there exists $R \in (0,1)$ so that whenever $2^{k_0} r < R$: For any $x_0 \in (-1,1)$ and $\varphi \in C_c^\infty(I(x_0,r))$,
\[
\begin{split}
\int_{\R} \varphi\ \err_5(\solU)\aleq & \eps\, \|\lapv \solu\|_{(2,\infty),B(2^{k_0}r,x_0)}\,
\|\varphi\|_{\infty}\\
&+ \tail{\sigma}{\|\lapv \solu\|_{(2,\infty)}}{x_0}{r}{k_0}\ \|\varphi\|_{\infty}\\
&+r^\sigma\, \|\varphi\|_{\infty,\R} 
 \end{split}
 \]
\end{proposition}
Before we prove Proposition~\ref{pr:V:err5} we introduce the operator $T: L^2(\R) \to L^2(\R^2_+,\R^2)$ as
\[
 T^1f(x,t) :=  \partial_x \brac{p_t \ast \brac{\lapms{\frac{1}{2}} f}(x)} = p_t \ast \brac{\lapv \Hz f}(x)
\]
\[
 T^2f(x,t) :=  \partial_t \brac{p_t \ast \brac{\lapms{\frac{1}{2}} f}(x)} = -p_t \ast \brac{\lapv f}(x)
\]
In particular, we have the relation
\[
 \nabla \solV = T(\lapv \solu).
\]
In view of Lemma~\ref{la:lapvptboundedness} and the boundedness of the Hilbert transform $\Hz$ on $L^2$ we find that $T$ is indeed a bounded operator  from $L^2(\R)$ $L^2(\R^2_+,\R^2)$.
From Theorem~\ref{th:extensioncomm} we obtain in particular the following estimate, which we will use to compare $\brac{\Pi^\perp(\solu)}^h \nabla \solV$ with $T(\Pi^\perp(\solu) \lapv \solu)$.
\begin{theorem}\label{th:V:err5:Tcommie}
Let $f \in H^{\frac{1}{2}}(\R)$ and $g \in L^{(2,\infty)}(\R)$. Denoting then $f^h(x,t) := p_t\ast f(x)$ the Poisson extension of $f$ to $\R^2_+$ we have the following estimate.
\[
 \|f^h T(g) - T(fg)\|_{2,\R^2_+} \aleq \|\lapv f\|_{2,\R}\ \|g\|_{(2,\infty),\R}
\]
and
\[
 \|f^h T(g) - T(fg)\|_{2,\R^2_+} \aleq \|\lapv f\|_{(2,\infty),\R}\ \|g\|_{2,\R}
\]
\end{theorem}
\begin{proof}
We only prove the first claim, the second one is analogous.

Theorem~\ref{th:extensioncomm} is directly applicable to $T^2$. For $T^1$, note that 
\[
 |f^h T^1(g) - T^1(fg)| \leq |f^h T^2(\Hz(g))  - T^2(f\Hz(g))|+ |T^2(f\Hz(g)-\Hz(fg))|.
\]
So by the boundedness and commutator theorem, Theorem~\ref{th:extensioncomm}, for $T^2$
\[
 \|f^h T^1(g) - T^1(fg)\|_{2,\R^2_+} \aleq \|\lapv f\|_{2,\R}\ \|\Hz(g)\|_{(2,\infty),\R} + \|f\Hz(g)-\Hz(fg)\|_{2,\R}
\]
With boundedness and the commutator theorem, Theorem~\ref{th:Hzcomm}, for the Hilbert transform $\Hz$,
\[
 \|f^h T^1(g) - T^1(fg)\|_{2,\R^2_+} \aleq \|\lapv f\|_{2,\R}\, \|g\|_{(2,\infty),\R}.
\]
\end{proof}
Now we start gathering important estimates for Proposition~\ref{pr:V:err5}. 
For $x_0 \in \R \times \{0\}$, $r > 0$ let $\eta \in C_c^\infty(B(0,1))$ a typical bump function constantly one in $B(0,\frac{1}{2})$ and set
\begin{equation}\label{eq:etakxik}
 \eta_{k} := \eta((x_0-x)/{2^k r}),\ \xi_k := \eta_k - \eta_{k-1}.
\end{equation}
We can see these cutoff functions also as cutoff function on the real line $\R \times \{0\}$, simply by restriction.

First we estimate the situation where the support of the integral and the support of $\varphi^h$ are far away.
\begin{lemma}\label{la:V:err5:faraway}
Let $X \in L^2(\R^2_+)$. For any $x_0 \in \R$, $r \in (0,1)$ and $\varphi \in C_c^\infty(I(x_0,r))$, $k \geq 2$,
\[
 \int_{\R^2_+} X\cdot \nabla \solV\ \xi_k\, \varphi^h \aleq \|X\|_{2,\R^2_+}\, \|\varphi\|_{\infty}\, \sum_{\ell = k}^\infty 2^{-\sigma \ell} \|\lapv \solu\|_{(2,\infty),I(x_0,2^{\ell} r)}.
\]
\end{lemma}
\begin{proof}
Observe that since $k \geq 2$ and $\supp \varphi \subset I(x_0,r)$ we have that $\Psi := \xi_k \varphi^h$ vanishes on $\R \times \{0\}$. So we are in a similar situation to Lemma~\ref{la:V:err4}. Arguing as there with Hodge decomposition, and using Theorem~\ref{th:clmshalfspace} we find
\[
 \int_{\R^2_+} X\cdot \nabla \solV\ \xi_k\, \varphi^h \aleq k 2^{-\sigma k} \|X\|_{2,\R^2_+} \|\varphi\|_{\infty}\ \brac{[\solV]_{BMO,B(x_0,2^{k+10} r)}+[\tilde{\solV}]_{BMO,B(x_0,2^{k+10} r)}}.
\]
With the help of Proposition~\ref{pr:VtBMO} and Proposition~\ref{pr:localHilbert} this implies
\[
 \int_{\R^2_+} X\cdot \nabla \solV\ \xi_k\, \varphi^h \aleq \sum_{\ell = k}^\infty 2^{-\sigma \ell} \|X\|_{2,\R^2_+} \|\varphi\|_{\infty}\ \|\lapv \solu\|_{(2,\infty),I(x_0,2^{\ell} r)}.
\]
\end{proof}

\begin{lemma}\label{la:V:err5:disjointT}
For $\ell \geq k+10$, for some $\sigma > 0$,
\[
 \|T(\xi_\ell f)\, \eta_k\|_{2,\R^2_+} \aleq 2^{-\sigma\brac{\ell-k}} \|\xi_\ell f\|_{(2,\infty),\R}
\]
and
\[
 \|T(\eta_k f)\, \xi_\ell\|_{2,\R^2_+} \aleq 2^{-\sigma\brac{\ell-k}} \|\eta_k f\|_{(2,\infty),\R}.
\]
Moreover, we also have the inhomogeneous versions
\[
 \|T(\eta_k f)\, \xi_\ell\|_{\infty,\R^2_+} \aleq 2^{-\sigma\brac{\ell-k}}\ (2^\ell r)^{-1} \|\eta_k f\|_{(2,\infty),\R}
\]
and
\[
 \|T(\eta_k f)\, \xi_\ell\|_{2,\R^2_+} \aleq 2^{-\sigma\brac{\ell-k}}\ (2^k r)^{-\frac{1}{2}} \|\eta_k f\|_{1,\R}
\]
\end{lemma}
\begin{proof}
We only prove the first claim, the other ones follows analogously.

First we consider $T^2$. We apply Lemma~\ref{la:ptest}. Keep in mind that the cutoff function $\xi_\ell$ act on $\R \times \{0\}$ and the cutoff function $\eta_k$ acts on $\R^2_+$. However if $|(x,t) - (x_0,0)| \leq 2^k r$ and $|y-x_0| \geq 2^{\ell-2} r$ then $|x-y| \geq 2^\ell r$, since $\ell \geq k+10$. Thus, from Lemma~\ref{la:ptest},
\[
 \|T^2(\xi_\ell f)(x,t)\, \eta_k\|_{L^\infty(\R^2_+)} \aleq (2^\ell r)^{-\frac{3}{2}}\, \|\xi_\ell f\|_{1,\R} 
\]
Consequently, using H\"older inequality once on $\R^2_+$ and once on $\R$,
\[
 \|T^2(\xi_\ell f)(x,t)\, \eta_k\|_{2,\R^2_+} \aleq 2^{k-\ell}\,  \|\xi_\ell f\|_{(2,\infty),\R} 
\]
As for the estimate of $T^1$, 
\[
 \|T^1(\xi_\ell f)(x,t)\, \eta_k\|_{2,\R^2_+} \leq \|T^2(\eta_k \Hz(\xi_\ell f))(x,t)\, \eta_k\|_{2,\R^2_+}+\sum_{j=k}^\infty \|T^2(\xi_j \Hz(\xi_\ell f)) \eta_k\|_{2,\R^2_+}
\]
For the first term we apply the boundedness of $T^2$,
\[
 \|T^2(\eta_k \Hz(\xi_\ell f))(x,t)\, \eta_k\|_{2,\R^2_+} \aleq \|\eta_k \Hz(\xi_\ell f))(x,t)\|_{2,\R}
\]
Now we can use the disjoint support of $\eta_k$ and $\xi_\ell$ as functions on $\R \times \{0\}$ and find
\[
 \|\eta_k \Hz(\xi_\ell f))(x,t)\|_{2,\R^2_+} \aleq 2^{-\frac{1}{2}\brac{\ell-k}}\|\xi_\ell f\|_{(2,\infty),\R}.
\]
For the remaining term, since $\ell$ and $k$ are sufficiently far away, we split the sum
\[
 \sum_{j=k}^\infty = \sum_{j=k}^{k+4} + \sum_{j=k+5}^{\ell-5} + \sum_{j=\ell-4}^{\ell+4}+\sum_{j=\ell+5}^{\infty}.
\]
By boundedness of $T^2$ and the disjoint support of $\xi_j$ and $\xi_\ell$,
\[
 \sum_{j=k}^{k+4}  \|T^2(\xi_j \Hz(\xi_\ell f)) \eta_k\|_{2,\R^2_+} \aleq \sum_{j=k}^{k+4}  \|\xi_j \Hz(\xi_\ell f)\|_{2,\R} \aleq 2^{-\frac{1}{2} (\ell-k)}\, \|\xi_\ell f\|_{(2,\infty),\R}.
\]
By the disjoint support of $\xi_j$ and $\eta_k$, and again by the disjoint support of $\xi_j$ and $\xi_\ell$,
\[
\begin{split}
 \sum_{j=k+5}^{\ell-5}  \|T^2(\xi_j \Hz(\xi_\ell f)) \eta_k\|_{2,\R^2_+}\aleq& \sum_{j=k+5}^{\ell-5}  2^{-(j-k)}\, \|\xi_j \Hz(\xi_\ell f) \|_{(2,\infty),\R} \\
 \aleq& \sum_{j=k+5}^{\ell-5}  2^{-(j-k)}\, 2^{-\frac{1}{2}(\ell-j)}\, \|\xi_\ell f\|_{(2,\infty),\R}\\
 \aleq & 2^{-\frac{1}{2}(\ell-j)}\, \|\xi_\ell f\|_{(2,\infty),\R}.
\end{split}
 \]
By the disjoint support of $\xi_j$ and $\eta_k$ and boundedness of $\Hz$,
\[
 \sum_{j=\ell-4}^{\ell+4} \|T^2(\xi_j \Hz(\xi_\ell f)) \eta_k\|_{2,\R^2_+}\aleq 2^{-\frac{1}{2}(\ell-j)}\, \|\xi_\ell f\|_{(2,\infty),\R}.
\]
Finally, first by the disjoint support of $\xi_j$ and $\eta_k$ and then by the disjoint support of $\xi_j$ and $\xi_\ell$,
\[
 \sum_{j=\ell+5}^{\infty} \|T^2(\xi_j \Hz(\xi_\ell f)) \eta_k\|_{2,\R^2_+} \aleq \sum_{j=\ell+5}^{\infty}  2^{-(j-k)}\ 2^{-\frac{1}{2}(j-\ell)} \|\xi_\ell f\|_{(2,\infty),\R} \aleq 2^{-(\ell-j)}\, \|\xi_\ell f\|_{(2,\infty),\R}.
\]
The claim is proven, if we choose $\sigma = \frac{1}{2}$.
\end{proof}

Now we give
\begin{proof}[Proof of Proposition~\ref{pr:V:err5}]
In view of Lemma~\ref{la:V:err5:faraway},
\[
\begin{split}
 \left |\int_{\R^2_+} \Pi^h_{\ell j}(\solu)\, \Omega_{ji}\, (\Pi^\perp_{ik}(\solu))^h\cdot \nabla \solV^k\,\varphi^h\right | \aleq&\left |\int_{\R^2_+} \Pi^h_{\ell j}(\solu)\, \Omega_{ji}\, (\Pi^\perp_{ik}(\solu))^h\cdot \nabla \solV^k\,\eta_{k_0}\,\varphi^h \right |\\
 &+ \|\Omega\|_{2,\R^2_+}\, \|\varphi\|_{\infty} \sum_{k=k_0}^\infty 2^{-\sigma k} \|\lapv \solu\|_{(2,\infty),I(x_0,2^k r)}.
\end{split}
 \]
Using the representation $\nabla \solV = T(\lapv \solu)$ we have
\[
\begin{split}
 &\left |\int_{\R^2_+} \Pi^h_{\ell j}(\solu)\, \Omega_{ji}\, (\Pi^\perp_{ik}(\solu))^h\cdot T(\lapv \solu)\,\eta_{k_0}\,\varphi^h \right |\\
 \aleq&\left |\int_{\R^2_+} \Pi^h_{\ell j}(\solu)\, \Omega_{ji}\, (\Pi^\perp_{ik}(\solu))^h\cdot T(\eta_{2k_0}\lapv \solu)\,\eta_{k_0}\,\varphi^h \right |\\
 &+\|\Omega\|_{2,B^+(x_0,2^{k_0}r)}\, \|\varphi\|_{\infty,\R} \sum_{\ell = 2k_0}^\infty \|T(\xi_{\ell}\lapv \solu)\,\eta_{k_0}\|_{2,\R^2_+}.
\end{split}
 \]
By Lemma~\ref{la:V:err5:disjointT}, if $k_0$ is sufficiently large, we have found
\[
\begin{split}
 \left |\int_{\R^2_+} \Pi^h_{\ell j}(\solu)\, \Omega_{ji}\, (\Pi^\perp_{ik}(\solu))^h\cdot \nabla \solV^k\,\varphi^h\right | \aleq&
\|\Omega\|_{2,B^+(x_0,2^{k_0}r)}\, \|\varphi\|_{\infty,\R}\, \|(\Pi^\perp_{ik}(\solu))^h T(\eta_{2k_0}\lapv \solu)\|_{2,\R^2_+}\\
 &+ \tail{\sigma}{\|\lapv \solu\|_{(2,\infty)}}{x_0}{r}{k_0}\ \|\varphi\|_{\infty}.
\end{split}
 \]
Now we commute $(\Pi^\perp (\solu))^h$ and $T$, then with Theorem~\ref{th:V:err5:Tcommie}
\[
\begin{split}
  &\|(\Pi^\perp_{ik}(\solu))^h T(\eta_{2k_0}\lapv \solu)\|_{2,\R^2_+}\\
  \aleq& \|T(\eta_{2k_0}\Pi^\perp(\solu)\lapv \solu)\|_{2,\R^2_+}+\|\lapv \Pi(\solu)\|_{2,\R}\ \|\lapv \solu\|_{(2,\infty),I(x_0,2^{2k_0}r)}.
 \end{split}
 \]
With boundedness of $T$,
\[
  \|T(\eta_{2k_0}\Pi^\perp(\solu)\lapv \solu)\|_{2,\R^2_+} \aleq  \|\eta_{2k_0}\Pi^\perp(\solu)\lapv \solu\|_{2,\R}
\]
Now if $I(x_0,2^{2k_0}r) \subset (-1,1)$ and $r$ is small enough, by Lemma~\ref{la:Piperplapvu},
\[
\begin{split}
 &\|\eta_{2k_0}\Pi^\perp(\solu)\lapv \solu\|_{2,\R}\\
 \aleq& \|\lapv \solu\|_{2,\R}\ \brac{\|\lapv \solu\|_{(2,\infty), I(x_0,2^{3k_0}r)}+\tail{\sigma}{\|\lapv \solu\|_{(2,\infty)}}{x_0}{r}{k_0} + (2^{k_0}r)^{\sigma}}.
\end{split}
 \]
Together we have shown,
\[
\begin{split}
 &\left |\int_{\R^2_+} \Pi^h_{\ell j}(\solu)\, \Omega_{ji}\, (\Pi^\perp_{ik}(\solu))^h\cdot \nabla \solV^k\,\varphi^h\right | \\
 \aleq& \|\Omega\|_{2,B^+(x_0,2^{k_0}r)}\, \|\lapv \solu\|_{(2,\infty),I(x_0,2^{3k_0} r)}\, \|\varphi\|_{\infty}\\
 &+\tail{\sigma}{\|\lapv \solu\|_{(2,\infty)}}{x_0}{r}{k_0}\ \|\varphi\|_{\infty} + (2^{k_0}r)^{\sigma}\ \|\varphi\|_{\infty}.
\end{split}
 \]
So if $x_0 \in (-1,1)$ and $2^{k_0}r  \leq R$ for some $R$ so small that 
\[
  \sup_{y_0 \in \R^2_+} \|\Omega\|_{2,B^+(y_0,R)} < \eps,
\]
we have shown the claim.
\end{proof}
From \eqref{eq:boundaryactionstart1}, \eqref{eq:boundaryactionstart2} and Proposition~\ref{pr:V:err5} we have found 
\begin{equation}\label{eq:pisolwpart}
 \int_{\R} \Pi_{\ell j}(\solu)\, \partial_\nu \solW^j\ \varphi = \int_{\R^2_+} \Pi^h_{\ell j}(\solu)\, \Omega_{ji}\,  \Pi^h_{ik}(\solu) \cdot T[\lapv \solu^k] \,  \varphi^h + \int_{\R} \err \varphi,
\end{equation}
where $\err$ satisfies conditions~\ref{cond:err}.

We now observe that this can be written as an antisymmetric potential $\omega$ that satisfies the conditions~\ref{cond:omega}, namely we have
\[
 \int_{\R^2_+} \Pi^h_{\ell j}(\solu)\, \Omega_{ji}\cdot \nabla \solV^i \,  \varphi^h = \int_{\R} \omega^2_{\ell k}(\lapv \solu^k)(x)\ \varphi(x)\ dx+ \int_{\R} \err \varphi.
\]
Indeed, $\omega^2_{\ell k}$ is defined by
\begin{equation}\label{eq:defomega2}
\int_{\R} \omega^2_{\ell k}(f)(x)\ \varphi(x)\ dx :=  \int_{\R^2_+} \Pi^h_{\ell j}(\solu)\, \Omega_{ji}\,  \Pi^h_{ik}(\solu) \cdot T[f] \,  \varphi^h.
\end{equation}
Since $f \mapsto T(f)$ and $\varphi \mapsto \varphi^h$ are convolution operators, we find the representation
\[
\int_{\R} \omega^2_{\ell k}(f)(x)\ \varphi(x)\ dx := \int_{\R} \omega^2_{\ell k}(x,y) f(y)\ \varphi(x),
\]
for
\[
\begin{split}
 \omega^2_{\ell k}(x,y) :=& \int_{\R^2_+} \brac{\Pi^h_{\ell j}(\solu)\, \Omega^1_{ji}\,  \Pi^h_{ik}(\solu)}(z,t)\, \lapv(p_t)(z-y)\,  p_t(z-x)\, dz\, dt\\
 &-\int_{\R}\int_{\R^2_+} \brac{\Pi^h_{\ell j}(\solu)\, \Omega^1_{ji}\,  \Pi^h_{ik}(\solu)}(z,t)\, \lapv(p_t)(z-\tilde{y})\,  p_t(z-x)\, dz\ dt\, \frac{y-\tilde{y}}{|y-\tilde{y}|^2}d\tilde{y}.
\end{split}
 \]
Clearly, $\omega_{\ell k} = -\omega_{k\ell}$ since $\Pi^h(\solu)\, \Omega\,  \Pi^h(\solu)$ is an antisymmetric matrix. Indeed we have.
\begin{proposition}\label{pr:omega2correct}
$\omega^2$ as above satisfies the localization properties from condition~\ref{cond:omega}.
\end{proposition}
\begin{proof}
Firstly, the following estimate follows from the representation \eqref{eq:defomega2}, the boundedness of $T$ and the fact that $\|\varphi^h\|_{\infty,\R^2_+} \leq \|\varphi\|_{\infty,\R}$:
\[
 \int_{\R} \int_{\R}f(y)\, \varphi(x)\ \omega^2(x,y)\ dx\, dy\ \leq \|\Omega\|_{2,\R^2_+}\, \|f\|_{L^2(\R)}\, \|\varphi\|_{L^\infty(\R)}.
\]
As for the localization properties, fix $\eps > 0$. For some $k_0$ and $R> 0$ to be chosen below assume that $r \in (0,2^{-k_0} R)$. For simplicity denote by $\tilde{\Omega} := \Pi^h(\solu)\, \Omega\,  \Pi^h(\solu)$.

\underline{Proof of condition~\eqref{eq:oa:1}}. 
Assume that $\|g\|_{\infty} + \|\lapv g\|_{2} \leq 1$ and $\varphi \in C_c^\infty(I(x_0, r))$ with $\|\varphi\|_{\infty} + \|\lapv \varphi\|_{2} \leq 1$. Then from \eqref{eq:defomega2},
\[
\begin{split}
  &\int_{\R} \int_{\R} f(y)\, g(x)\, (\varphi(x)-\varphi(y))\, \omega^2_{kj} (x,y)\, dx\, dy\\
=&\int_{\R^2_+} \tilde{\Omega}  \cdot \brac{T[f] \,  \brac{g \varphi}^h - T[f\varphi] \,  g^h}\\
=&\int_{\R^2_+} \tilde{\Omega}  \cdot T[f]\ \brac{\brac{g \varphi}^h - \varphi^h\, g^h}+\int_{\R^2_+} \tilde{\Omega}  \cdot \brac{T[f] \,  \varphi^h  - T[f\varphi] }g^h.
\end{split}
\]
For the first term, recall that 
\[
 Tf(x,t) = \nabla_{\R^2} p_t \ast \lapms{\frac{1}{2}} f(x)
\]
In particular, $\div Tf = \curl Tf = 0$ in $\R^2_+$. On the other hand, $\brac{g \varphi}^h - \varphi^h\, g^h$ is zero on $\R \times \{0\}$. Thus we can proceed as above for Lemma~\ref{la:V:err4}, and obtain
\[
 \int_{\R^2_+} \tilde{\Omega}  \cdot T[f]\ \brac{\brac{g \varphi}^h - \varphi^h\, g^h} \aleq \|\Omega\|_{2,B^+(x_0,2^{k_0}r)}\ \|f\|_{(2,\infty),I(x_0,2^{k_0}r)} + \tail{\sigma}{\|f\|_{(2,\infty)}}{x_0}{r}{k_0}
\]
For the second term we argue as in the proof of Proposition~\ref{pr:V:err5}, and obtain again
\[
 \int_{\R^2_+} \tilde{\Omega}  \cdot \brac{T[f] \,  \varphi^h  - T[f\varphi] }g^h \aleq \|\Omega\|_{2,B^+(x_0,2^{k_0}r)}\ \|f\|_{(2,\infty),I(x_0,2^{k_0}r)} + \tail{\sigma}{\|f\|_{(2,\infty)}}{x_0}{r}{k_0}.
\]
This implies \eqref{eq:oa:1} if we choose $R$ so small so that 
\[
 \sup_{x_0 \in \R \times \{0\}} \|\Omega\|_{2,B^+(x_0,R)} < \eps.
\]
That is, condition~\eqref{eq:oa:1} is satisfied.

\underline{Proof of condition~\eqref{eq:oa:2}}. 
Assume $\|g\|_{\infty} + \|\lapv g\|_{2} \leq 1$, any $\|\varphi\|_{\infty} + \|\lapv \varphi\|_{2} \leq 1$, and $f \in C_c^\infty(B(x_0, r))$.
As above, from \eqref{eq:defomega2} we find
\[
\begin{split}
 &\int_{\R} \int_{\R} f(y)\, g(x)\, (\varphi(x)-\varphi(y))\, \omega_{kj} (x,y)\, dx\, dy\\
=&\int_{\R^2_+} \eta_{k_0}\tilde{\Omega}  \cdot T[f]\ \brac{\brac{g \varphi}^h - \varphi^h\, g^h}+\int_{\R^2_+} \eta_{k_0}\tilde{\Omega}  \cdot \brac{T[f] \,  \varphi^h  - T[f\varphi] }g^h\\
&+\sum_{k=k_0}^\infty \int_{\R^2_+} \xi_{k}\tilde{\Omega}  \cdot \brac{T[f]\brac{g \varphi}^h - T[f\varphi] g^h}.
\end{split}
 \]
where $\eta_k$ and $\xi_k$ are cutoff functions as above in \eqref{eq:etakxik}.
As above for condition~\eqref{eq:oa:1} we obtain
\[
 \int_{\R^2_+} \eta_{k_0}\tilde{\Omega}  \cdot T[f]\ \brac{\brac{g \varphi}^h - \varphi^h\, g^h} \aleq \|\Omega\|_{2,B^+(x_0,2^{k_0}r)}\ \|f\|_{(2,\infty)}.
\]
By Theorem~\ref{th:V:err5:Tcommie} we find
\[
 \int_{\R^2_+} \eta_{k_0}\tilde{\Omega}  \cdot \brac{T[f] \,  \varphi^h  - T[f\varphi] }g^h \aleq \|\Omega\|_{2,B^+(x_0,2^{k_0}r)}\ \|f\|_{(2,\infty)}.
\]
Since $\supp f \subset B(x_0,r)$, by Lemma~\ref{la:V:err5:disjointT} we have
\[
 \sum_{k=k_0}^\infty \int_{\R^2_+} \xi_{k}\tilde{\Omega}  \cdot \brac{T[f]\brac{g \varphi}^h - T[f\varphi] g^h} \aleq 2^{-\sigma k_0}\, \|\tilde{\Omega}\|_{2,\R^2_+}\, \|f\|_{(2,\infty)}.
\]
Thus we have shown
\[
 \int_{\R} \int_{\R} f(y)\, g(x)\, (\varphi(x)-\varphi(y))\, \omega_{kj} (x,y)\, dx\, dy \aleq \brac{\|\Omega\|_{2,B^+(x_0,2^{k_0}r)}+2^{-\sigma k_0}\, \|\tilde{\Omega}\|_{2,\R^2_+} } \|f\|_{(2,\infty)}.
\]
Choosing $k_0$ large enough so that 
\[
 2^{-\sigma k_0}\, \|\tilde{\Omega}\|_{2,\R^2_+} \leq \frac{\eps}{2}
\]
and then $R$ small enough, so that 
\[
 \sup_{x_0 \in \R^2_+} \|\tilde{\Omega}\|_{2,B(x_0,R)} < \frac{\eps}{2},
\]
we conclude that condition~\eqref{eq:oa:2} is satisfied.

\underline{Proof of condition~\eqref{eq:oa:3}}. 
Assume that $\|g\|_{\infty} + \|\lapv g\|_{2} \leq 1$, $\varphi \in C_c^\infty(I(x_0, r))$ is arbitrary and $f \in C_c^\infty(\R)$. This time, we write
\[
\begin{split}
 &\int_{\R} \int_{\R} f(y)\, g(x)\, (\varphi(x)-\varphi(y))\, \omega_{kj} (x,y)\, dx\, dy\\
=&\int_{\R^2_+} \eta_{k_0}\tilde{\Omega}  \cdot T[f]\ \brac{\brac{g \varphi}^h - \varphi^h\, g^h}+\int_{\R^2_+} \eta_{k_0}\tilde{\Omega}  \cdot \brac{T[f] \,  \varphi^h  - T[f\varphi] }g^h\\
&+\sum_{k=k_0}^\infty \int_{\R^2_+} \xi_{k}\tilde{\Omega}  \cdot \brac{T[f]\brac{g \varphi}^h - T[f\varphi] g^h}\\
\end{split}
 \]
Let $\varphi^e$ be the even reflection of $\varphi^h$ to $\R^2_-$. From \cite[Proposition 10.5.]{Lenzmann-Schikorra-commutators} we have
\begin{equation}\label{eq:bmoreflect}
 [\varphi^e]_{BMO,\R^2} \aleq [\varphi]_{BMO,\R}.
\end{equation}
Thus, from Proposition~\ref{pr:ptfgmptfptg} and the boundedness of $T$, we find
\[
 \int_{\R^2_+} \eta_{k_0}\tilde{\Omega}  \cdot T[f]\ \brac{\brac{g \varphi}^h - \varphi^h\, g^h} \aleq \|\Omega\|_{2,B(x_0,2^{k_0} r)}\ \|f\|_{2,\R}\ [\varphi]_{BMO,\R}
\]
From Theorem~\ref{th:V:err5:Tcommie},
\[
 \int_{\R^2_+} \eta_{k_0}\tilde{\Omega}  \cdot \brac{T[f] \,  \varphi^h  - T[f\varphi] }g^h \aleq \|\Omega\|_{2,B(x_0,2^{k_0} r)}\ \|f\|_{2,\R}\ \|\lapv \varphi\|_{(2,\infty),\R}
\]
Next, by boundedness of $T$, we have
\[
\begin{split}
 &\int_{\R^2_+} \xi_{k}\tilde{\Omega}  \cdot \brac{T[f]\brac{g \varphi}^h - T[f\varphi] g^h}\\
 \aleq& \|g\|_{\infty,\R} \|\Omega\|_{2,\R^2_+}\ \|f\|_{2,\R} \|\xi_k \brac{g \varphi}^h\|_{\infty} + \|g\|_{\infty,\R}\, \|\Omega\|_{2,\R^2_+}\ \|\xi_k T[f\varphi]\|_{2,\R^2_+}.
\end{split}
 \]
On the one hand, by Lemma~\ref{la:harmextest}, the support of $\varphi$ and Poincar\'e inequality,
\[
 \|\xi_k \brac{g \varphi}^h\|_{\infty} \aleq (2^{k} r)^{-1} \|g\varphi\|_{1,\R} \aleq 2^{-k} \|g\|_{\infty,\R}\ \|\lapv \varphi\|_{(2,\infty),\R}.
\]
On the other hand, by Lemma~\ref{la:V:err5:disjointT}, the support of $\varphi$ and Poincare inequality,
\[
\|\xi_k T[f\varphi]\|_{2,\R^2_+} \aleq 2^{-\sigma k} \, \|f\|_{2,\R}\, \|\lapv \varphi\|_{(2,\infty),\R}.
\]
In conclusion,
\[
 \sum_{k=k_0}^\infty \int_{\R^2_+} \xi_{k}\tilde{\Omega}  \cdot \brac{T[f]\brac{g \varphi}^h - T[f\varphi] g^h} \aleq 2^{-\sigma k_0} \|\Omega\|_{2}\, \|f\|_{2,\R}\ \|\lapv \varphi\|_{(2,\infty),\R}
\]
Choosing again $k_0$ sufficiently large and then $R$ sufficiently small so that 
\[
 \brac{\|\Omega\|_{2,B^+(x_0,2^{k_0}r)}+2^{-\sigma k_0}\, \|\tilde{\Omega}\|_{2,\R^2_+} } < \eps
\]
we conclude that \eqref{eq:oa:3} is satisfied.
\end{proof}

Finally, all the ingredients of Theorem~\ref{th:transform} are available.
\begin{proof}[Proof of Theorem~\ref{th:transform}]
Take \[\omega := \omega^1 + \omega^2\] where $\omega^1$ is from Lemma~\ref{la:piperplaphu} and $\omega^2$ is from Proposition~\ref{pr:omega2correct}. Then we have from \eqref{eq:solueq}, Lemma~\ref{la:piperplaphu}, and \eqref{eq:pisolwpart}
\[
 \laph \solu^i = \omega_{ij}(\lapv \solu^j) + \err^i(\solU) \quad \mbox{in $(-2,2)$}
\]
where $\omega$ is antisymmetric and satisfies condition~\ref{cond:omega} and $\err$ satisfies condition~\ref{cond:err}.
 This proves Theorem~\ref{th:transform}.
\end{proof}
\section{Regularity theory for systems with antisymmetric potential at the boundary: Proof of Theorem~\ref{th:mastereq}}\label{s:proofoftheoremmaster}
Assume that $\solU \in H^1(\R^2_+,\R^N)$ has compact support and has the trace $\solu = \solU \Big|_{\R \times \{0\}} \in L^\infty \cap H^{\frac{1}{2}}(\R,\R^N)$, which are solutions to \eqref{eq:transformedeqV} and \eqref{eq:trU}. 

That is 
\[
 \tag{\ref{eq:trU}} 
 \lap \solU^i  = \Omega_{ij}\cdot \nabla \solU^j + \Err^i(\solU) \quad \mbox{in }\R^2_+\\
\]
for some $\Omega_{ij} = -\Omega_{ij} \in L^2(\R^2_+,\R^2)$ and for some $\Err^i(\solU)$ satisfying the conditions~\ref{cond:Err} below.

\[
  \laph \solu^i = \omega_{ij}(\solu^j) + \err^i(\solU) \quad \mbox{in }(-2,2) \tag{\ref{eq:transformedeqV}}
\]
for a nonlocal, boundary antisymmetric potential $\omega_{ij} = -\omega_{ji}: \dot{H}^{\frac{1}{2}}(\R) \to L^1(\R)$ which is a linear operator given via
\[
  \int_{\R} \omega_{ij}(f)\ \varphi := \int_{\R}\int_{\R} \omega_{ij}(x,y)\, \lapv f(y)\ \varphi(x)\ dy\, dx
\]
whose kernel $\omega_{ij}(x,y)$ satisfies the localization conditions~\ref{cond:omega} below. Moreover $\err(\solU)$ satisfies the conditions~\ref{cond:err}.

With $\solV \in \dot{H}^1(\R^2_+)\cap L^\infty(\R^2_+)$, we denote the Poisson extension $\solV = \solu^h = p_t \ast \solu$, which satisfies
\[
 \begin{cases}
  \lap \solV = 0 \quad &\mbox{in $\R^2_+$}\\
  \solV = \solu \quad &\mbox{on $\R \times \{0\}$}\\
  \lim_{|(x,t)| \to \infty} \solV(x,t) = 0.
 \end{cases}
\]
Then $\solU = \solW + \solV$ for $\solW \in \dot{H}^1(\R^2_+)$ satisfying
\[
 \begin{cases}
  \lap \solW^i = \Omega_{ij}\cdot \nabla \solU^j + \Err^i(\solU) \quad &\mbox{in }\R^2_+\\
  \solW = 0 \quad &\mbox{on $\R \times \{0\}$}\\
  \lim_{|(x,t)| \to \infty} \solW(x,t) = 0.
 \end{cases}
\]
Let $x_0 \in (-1,1) \times \{0\}$, $r > 0$. We are going to prove a decay estimate for 
\[
 \mathscr{G}(x_0,r) := \|\nabla \solW\|_{(2,\infty),B^+(x_0,r)} + \|\lapv \solu\|_{(2,\infty),I(x_0,r)}.
\]
\begin{proposition}\label{pr:decay}
Let $\solU$, $\solW$, $\solu$ be as above. Then for any $\eps > 0$ there exists a radius $R \in (0,1)$, a constant $k_0 \in \N$ so that for any $x_0 \in (-1,1) \times \{0\}$ and any $r \in (0,2^{-k_0} R)$ it holds
\[
 \mathscr{G}(x_0,r) \leq \eps\, \mathscr{G}(x_0,2^{k_0} r) + r^\sigma + \sum_{k=k_0}^\infty 2^{-\sigma k}\, \mathscr{G}(x_0,2^{k}r).
\]
Here, $\sigma > 0$ is a uniform constant.
\end{proposition}
Proposition~\ref{pr:decay} is a consequence of Propositions~\ref{pr:West} and \ref{pr:Vgoal} below. Propositions~\ref{pr:West} estimates the interior quantity $\|\nabla \solW\|_{(2,\infty)}$, and Proposition~\ref{pr:Vgoal} estimates the boundary quantity $\|\lapv \solu\|_{(2,\infty)}$.
\begin{proposition}\label{pr:West}
Let $\solU$, $\solW$ be as above. For any $\eps > 0$ there exists $R  \in (0,1)$ so that for any $x_0 \in \R \times \{0\}$ and $k_0 \geq 10$ and any $r < 2^{-k_0}R$ we have
\[
\begin{split}
\|\nabla \solW\|_{(2,\infty),B^+(x_0,r)}  \aleq& + \eps\ \brac{\|\nabla \solW\|_{(2,\infty),B(x_0,\Lambda r)}+\|\lapv \solu\|_{(2,\infty),B(x_0,\Lambda r)}}\\
&+  \sqrt{r}\ \brac{\|\Omega\|_{2,\R^2_+}+1}\, \|\nabla \solU\|_{2,\R^2_+}\\
&+ (1+\|\Omega\|_{2})\ \tail{\sigma}{\|\lapv \solu\|_{(2,\infty)}+\|\nabla \solW\|_{(2,\infty)}}{x_0}{R}{k_0}.
\end{split}
\]
Here, $\sigma$ is a uniform constant.
\end{proposition}

\begin{proposition}\label{pr:Vgoal}
For any $\eps > 0$ there exists $R \in (0,1)$ so that whenever $2^{k_0} r < R$, $x_0 \in (-1,1)$ so that
\begin{equation}\label{eq:lapvuestgoal}
\begin{split}
 \|\lapv \solu\|_{(2,\infty),I(x_0,r)} \aleq& \eps \brac{\|\lapv \solu\|_{(2,\infty),I(x_0,2^{k_0}r)} + \|\lapv \solW\|_{(2,\infty),B^+(x_0,2^{k_0}r)}}\\
 &+ \tail{\sigma}{\|\lapv \solu\|_{(2,\infty)}+\|\nabla \solW\|_{(2,\infty)}}{x_0}{r}{k_0} + (2^{k_0}r)^\sigma.
\end{split}
 \end{equation}
Here, $\sigma$ is a uniform constant.
\end{proposition}
From Proposition~\ref{pr:decay} we obtain Theorem~\ref{th:mastereq} in a standard way.
\begin{proof}[Proof of Theorem~\ref{th:mastereq}]
Applying Proposition~\ref{pr:decay} on successively smaller radii we obtain for some $\tau = \tau(\eps,k_0,\sigma) > 0$ for any $x_0 \in (-1,1)\times \{0\}$ that for any $r > 0$,
\[
 \mathscr{G}(x_0,r) \leq C\ r^\tau,
\]
where $C$ depends on $\mathscr{G}(x_0,\infty)$. See for example \cite[Lemma A.8]{Blatt-Reiter-Schikorra-2016}. That is, for any $x_0 \in (-1,1)\times \{0\}$ and any $r > 0$,
\[
 \|\lapv \solu\|_{(2,\infty),I(x_0,r)}  + \|\nabla \solW\|_{(2,\infty),B^+(x_0,r)} \aleq r^\tau.
\]
This is a Morrey space condition, and estimates on Riesz potentials on Morrey spaces, see~\cite{Adams-1975}, implies that $\solu \in C^\alpha((-1,1))$ and $\solW \in C^\alpha(B^+(0,1))$. 

Continuity up to the boundary follows now from \cite{Mueller-Schikorra-2009}. H\"older continuity follows from a reflection argument:
Since $\solu$ is H\"older continuous, so is $\solV$. Thus, for any $x_0 \in (-1,1) \times \{0\}$,
\[
 \int_{B^+(x_0,r)} |\solV-(\solV)_{B^+(x_0,r)}|^2  \aleq r^{2+2\tau},
\]
and
\[
 \int_{B^+(x_0,r)} |\solW-(\solW)_{B^+(x_0,r)}|^2  \aleq r^{2+2\tau},
\]
Denote by $U^e$ the even reflection of $U$ across $\R \times \{0\}$. Then, since since $\solU = \solW + \solV$ in $\R^2_+$ for any $x_0 \in (-1,1) \times \{0\}$
\[
 \int_{B(x_0,r)} |\solU^e-(\solU^e)_{B(x_0,r)}|^2  \aleq r^{2+2\tau}.
\]
On the other hand, from the interior regularity theory due to Rivi\`{e}re, \cite{Riviere-2007}, we have
\[
 \int_{B(y_0,r)} |\solU-(\solU)_{B(y_0,r)}|^2 \aleq r^{2+2\tau} \quad \mbox{whenever $B(y_0,2r) \subset \R^2_+$}.
\]
For any $z_0 \in (-1,1) \times (0,\infty)$ and any $r > 0$ or $B(z_0,2r)\subset \R^2_+$ or $B(z_0,r) \subset B(\pi(z_0),5r)$, where $\pi(z_0)$ is the projection to $\R \times \{0\}$ of $z_0$. Thus, for any $z_0 \in (-1,1) \times \R$,
\[
\int_{B(z_0,r)} |\solU^e-(\solU^e)_{B(y_0,r)}|^2 \aleq r^{2+2\tau}
\]
By the characterization of H\"older spaces by Campanato spaces, see, e.g., \cite{Giaquinta-MI}, we have $\solU^e \in C^{\tau}((-1,1) \times \R)$, can consequently $\solU \in C^\tau\brac{(-1,1) \times [0,\infty)}$.
\end{proof}
\subsection{Interior decay estimate for W: Proof of Proposition~\ref{pr:West}}
Recall that $\solU = \solW + \solV$. Thus $\solW \in \dot{H}^1(\R^2_+)$ is a solution of
\begin{equation}\label{eq:West:Weq}
 \begin{cases}
  \lap \solW^i = \Omega_{ij}\cdot \nabla \solW^j + \Omega_{ij}\cdot \nabla \solV^j  + \Err^i(\solU) \quad &\mbox{in }\R^2_+\\
  \solW = 0 \quad &\mbox{on $\R \times \{0\}$}\\
  \lim_{|(x,t)| \to \infty} \solW(x,t) = 0.
 \end{cases}
\end{equation}
For the proof of Proposition~\ref{pr:West} we adapt carefully of the interior regularity theory of Rivi\`{e}re \cite{Riviere-2007}. Observe that his theory would be directly applicable to a solution $\tilde{W}$ of an equation of the form
\[
 \begin{cases}
  \lap \tilde{\solW}^i = \Omega_{ij}\cdot \nabla \tilde{\solW}^j  \quad &\mbox{in }\R^2_+\\
  \tilde{\solW} = 0 \quad &\mbox{on $\R \times \{0\}$}\\
  \lim_{|(x,t)| \to \infty} \tilde{\solW}(x,t) = 0.
 \end{cases}
\]
For our situation \eqref{eq:West:Weq}, however, we have two distortions on the right-hand side. On the one hand there is the (harmless) term $\Err^i(\solU)$. More importantly, in \eqref{eq:West:Weq} we have the term $\Omega \cdot \nabla \solV$, i.e. a boundary action term.

Nevertheless, it will suffice to essentially follow the interior arguments. From Proposition~\ref{pr:gaugenormal} we find an optimal gauge $P \in H^1(\R^2_+,SO(N))$ so that
\begin{equation}\label{eq:choicegauge}
 \dv (P \nabla P^T + P \Omega P) = 0 \quad \mbox{in $\R^2_+$}.
\end{equation}
and
\begin{equation}\label{eq:West:NPest}
 \|\nabla P\|_{2,\R^2_+} \aleq\|\Omega\|_{2,\R^2_+} .
\end{equation}
We choose $R \in (0,1)$ so that
\begin{equation}\label{eq:West:eps}
 \sup_{x \in \R^2_+} \|\nabla P\|_{2,B^+(x,R)} + \|\Omega\|_{2,B^+(x,R)} < \eps.
\end{equation}
In order to estimate $\nabla\solW$ we argue by duality. Namely, we find $F \in C^\infty(\R^2_+,\R^2)$ with support $\supp F \subset \overline{B^+(x_0,r)}$ so that $\|F\|_{(2,1)} \leq 1$ and so that
\[
 \|\nabla \solW\|_{(2,\infty),B^+(x_0,r)} \aeq \|P\, \nabla \solW\|_{(2,\infty),B^+(x_0,r)} \aleq \int_{\R^2_+} P\, \nabla \solW \cdot F.
\]
With Hodge decomposition we find $\Phi \in C^\infty(\overline{\R^2_+})$ and a vector field $H \in C^\infty(\overline{\R^2_+},\R)$
\[
 F = \nabla \Phi + H \quad \mbox{in $\R^2_+$},
\]
and we may assume that $\Phi = 0$ on $\R \times \{0\}$ and $\dv H = 0$ in $\R^2_+$. Moreover \begin{equation}\label{eq:West:Hest}\|H\|_{(2,1),\R^2_+} + \|\nabla \Phi\|_{(2,1),\R^2_+} \aleq 1.\end{equation} See Proposition~\ref{pr:hodgewithestimate}.

That is,
\begin{equation}\label{eq:West:nablaWest}
 \|\nabla \solW\|_{(2,\infty),B^+(x_0,r)} \aleq \int_{\R^2_+} P\, \nabla \solW \cdot H + \int_{\R^2_+} P\, \nabla \solW \cdot \nabla \Phi =: I_H + I_{\Phi}.
\end{equation}
First we treat $I_H$.
\begin{lemma}\label{la:West:IH}
For possibly smaller $R \in (0,1)$, all sufficiently large $k_0 \in \N$, all radii $r \in (0,2^{-k_0} R)$, and all $x_0 \in \R \times \{0\}$ we have
\[
 |I_H| \aleq \eps \|\nabla \solW\|_{(2,\infty),B^+{x_0,2^{k_0}r)}} + \|\Omega\|_{2}\ \tail{\sigma}{\|\nabla \solW\|_{(2,\infty)}}{x_0}{R}{k_0}.
\]
Here, $\sigma > 0$ is a uniform constant.
\end{lemma}
\begin{proof}
Observe that $W$ has zero-boundary, and thus
\[
 \int_{\R^2_+} P\nabla \solW \cdot H =  \int_{\R^2_+} \solW\, \nabla P\cdot H. 
\]
Since $\div H = 0$ on $\R^2_+$, the term $\nabla P \cdot H$ is a div-curl quantity. In view of Theorem~\ref{th:clmshalfspace} we find
\[
\begin{split}
  |I_H| \aleq& \|H\|_{2,\R^2_+}\, \|\nabla P\|_{2,B^+(x_0,2^{k_0}r)} \|\nabla \solW\|_{(2,\infty),B^+{x_0,2^{k_0}r)}}\\
  &+ \|\nabla P\|_{2}\, \sum_{k=k_0}^\infty k \|H\|_{2,B^+(x_0,2^{k+5}r)\backslash B^+(x_0,2^{k-5}r)}\ \|\nabla \solW\|_{(2,\infty),B^+{x_0,2^{k+5}r)}}.
  \end{split}
\]
By \eqref{eq:West:Hest} and \eqref{eq:West:eps}
\[
 \|H\|_{2,\R^2_+} \|\nabla P\|_{2,B^+(x_0,2^{k_0}r)} \|\nabla \solW\|_{(2,\infty),B^+{x_0,2^{k_0}r)}} \aleq \eps\, \|\nabla \solW\|_{(2,\infty),B^+{x_0,2^{k_0}r)}}
\]
Moreover, in view of $\supp F \subset B^+(x_0,r)$, Proposition~\ref{pr:hodgewithestimate}, and with $\|F\|_{2} \aleq \|F\|_{(2,1)} \leq 1$,
\[
 \|H\|_{2,B^+(x_0,2^{k+5}r)\backslash B^+(x_0,2^{k-5}r)} \aleq 2^{-k}.
\]
Thus, since $k 2^{-k} \aleq 2^{-\frac{1}{2}k}$ for $k$ large enough and $\|\nabla P\|_{2} \aleq \|\Omega\|_{2}$,
\[
 |I_H| \aleq \eps \|\nabla \solW\|_{(2,\infty),B^+{x_0,2^{k_0}r)}} + \|\Omega\|_{2}\ \tail{\sigma}{\|\nabla \solW\|_{(2,\infty)}}{x_0}{R}{k_0}.
\]
This concludes the estimate of $I_H$.
\end{proof}
As for the estimate of $I_\Phi$, we have
\begin{lemma}\label{la:West:IPhi}
For possibly smaller $R \in (0,1)$, all sufficiently large $k_0 \in \N$, all radii $r \in (0,2^{-k_0} R)$, and all $x_0 \in \R \times \{0\}$ we have
\[
 |I_\Phi| \aleq \eps \|\nabla \solW\|_{(2,\infty),B^+{x_0,2^{k_0}r)}} + \tail{\sigma}{\|\nabla \solW\|_{(2,\infty)}}{x_0}{R}{k_0}.
\]
Here, $\sigma > 0$ is a uniform constant.
\end{lemma}
\begin{proof}
In view of \eqref{eq:West:Weq},
\[
\begin{split}
 I_\Phi = \int_{\R^2_+} P_{ij}\nabla \solW^j \cdot \nabla \Phi =&-\int_{\R^2_+} (\nabla P_{ik} + P_{ij} \Omega_{jk}) \cdot \nabla \solW^k\ \Phi\\ 
 & -\int_{\R^2_+} P_{ij} \Omega_{jk} \cdot \nabla \solV^k \Phi \\
 & + \int_{\R^2_+} \Err^i(\solU) P_{ij}\Phi
 \end{split}
\]
The claim is now a a consequence of Lemma~\ref{la:West:1}, Lemma~\ref{la:West:2}, Lemma~\ref{la:West:3} below.
\end{proof}
\begin{lemma}\label{la:West:1}
For all sufficiently large $k_0 \in \N$, all radii $r \in (0,2^{-k_0} R)$, and all $x_0 \in \R \times \{0\}$ we have
\begin{equation}\label{eq:OmegaPWterm}
  |\int_{\R^2_+} (\nabla P_{ik} + P_{ij} \Omega_{jk}) \cdot \nabla \solW^k\ \Phi| \aleq \eps \|\nabla \solW\|_{(2,\infty),B^+(x_0,2^{k_0} r)} + \tail{\sigma}{\|\nabla \solW\|_{(2,\infty)}}{x_0}{R}{k_0}
\end{equation}
Here, $\sigma > 0$ is a uniform constant.
\end{lemma}
\begin{proof}
Denote by
\[\Omega^P_{i\ell} := \nabla P_{ik}\ P_{\ell k} + P_{ij} \Omega_{jk} P_{\ell k}.\]
By \eqref{eq:choicegauge}, $\div \Omega^P = 0$, and thus the fact that $\solW = 0$ on $\R \times \{0\}$ implies
\[
 |\int_{\R^2_+} (\nabla P_{ik} + P_{ij} \Omega_{jk}) \cdot \nabla \solW^k\ \Phi| = |\int_{\R^2_+} \Omega^P_{i\ell} \cdot  \nabla \brac{P_{\ell k}\Phi}\ \solW^k|
\]
By assumption, $\|\nabla \Phi\|_{(2,1),\R^2_+}$ and $\Phi \equiv 0$ on $\R \times \{0\}$. Sobolev embedding implies that
\[
 \|\Phi\|_{\infty,\R^2_+} \aleq 1.
\]
Consequently,
\[
 \|\nabla (P\Phi)\|_{2,B^+(x_0,2^{k_0} r} \aleq 1.
\]
Moreover, in view of Proposition~\ref{pr:hodgewithestimate} and $\supp F \subset B(x_0,r)$ as well as $\|F\|_{2,\R^2_+} \aleq 1$, for any $k$ large enough we have
\[
 \|\nabla (P\Phi)\|_{2,B^+(x_0,2^{k+5} r)\backslash B^+(x_0,2^{k-5} r)} \aleq 2^{-k} \brac{\|\nabla P\|_{2,\R^2_+} + 1}.
\]
Consequently, using again that $\div \Omega^P = 0$ and the fact that $\solW$ has zero boundary values on $\R \times \{0\}$, we are able to apply Theorem~\ref{th:clmshalfspace}, and find
\[
\begin{split}
 |\int_{\R^2_+} (\nabla P_{ik} + P_{ij} \Omega_{jk}) \cdot \nabla \solW^k\ \Phi| \aleq& \|\Omega\|_{2,B^+(x_0,2^{k_0} r)}\, \|\nabla \solW\|_{(2,\infty),B^+(x_0,2^{k_0} r)}\\
 &+\brac{1+\|\Omega\|_{2,\R^2_+}} \tail{\sigma}{\|\nabla \solW\|_{(2,\infty)}}{x_0}{R}{k_0}.
\end{split}
 \]
In view of \eqref{eq:West:eps} we conclude.
\end{proof}

\begin{lemma}\label{la:West:2}
For possibly smaller $R \in (0,1)$, all sufficiently large $k_0 \in \N$, all radii $r \in (0,2^{-k_0} R)$, and all $x_0 \in \R \times \{0\}$ we have
\begin{equation}\label{eq:West:st}
 |\int_{\R^2_+} P_{ij} \Omega_{jk} \cdot \nabla \solV^k\, \Phi| \aleq \eps \|\lapv \solu\|_{(2,\infty),I(x_0,2^{k_0} r)} + \tail{\sigma}{\|\lapv \solu\|_{(2,\infty)}}{x_0}{R}{k_0}.
\end{equation}
Here, $\sigma > 0$ is a uniform constant.
\end{lemma}
\begin{proof}
The proof is analogous to Lemma~\ref{la:V:err4} and is a consequence of Theorem~\ref{th:clmshalfspace}. Observe that $\solV$ is harmonic and $\Phi = 0$ on $\R \times \{0\}$. 
\end{proof}

\begin{lemma}\label{la:West:3}
For all radii $r \in (0,\infty)$ and all $x_0 \in \R \times \{0\}$ we have
\[
 \int_{\R^2_+} \Err^i(\solU) P_{ij}\Phi \aleq r^\sigma
\]
Here, $\sigma > 0$ is a uniform constant.
\end{lemma}
\begin{proof}
By assumption, $\|\nabla \Phi\|_{(2,1),\R^2_+}$ and $\Phi \equiv 0$ on $\R \times \{0\}$. Sobolev embedding implies that
\[
 \|\Phi\|_{L^\infty(\R^2_+)} \aleq 1.
\]
Also recall that $P \in SO(N)$ almost everywhere, and thus $\|P\|_{\infty} \leq 1$.

Let $\tilde{\eta} := \eta(\brac{\cdot-x_0}/\sqrt{r})$ for a usual bump function $\eta\in C_c^\infty(B(0,2))$ constantly one $B(0,1)$. Then by conditions~\ref{cond:Err}, for some $p > 1$,
\[
\begin{split}
 \int_{\R^2_+} \Err^i(\solU) P_{ij}\Phi \aleq& \|\tilde{\eta}\Phi\|_{\frac{p}{p-1}} + \|(1-\tilde{\eta})\, |\Phi|\|_{\infty} +r^{\frac{\sigma}{2}} \|\tilde{\eta}\, \nabla( P_{ij}\, \Phi)\|_{2,\R^2_+}+r^{\frac{\sigma-1}{2}} \|\Phi\|_{2,B^+(x_0,2\sqrt{r})\backslash B^+(x_0,\sqrt{r})}\\
 &+ \|\nabla \brac{P_{ij}\, \Phi}\|_{2,\R^2_+ \backslash B^+(x_0,\sqrt{r})} + \sqrt{r}^{-1}\, \|\Phi\|_{2,B^+(x_0,2\sqrt{r})\backslash B^+(x_0,\sqrt{r})}
\end{split}
 \]
Now,
\[
 \|\tilde{\eta}\Phi\|_{\frac{p}{p-1}}  \aleq r^{\frac{p-1}{p}}\, \|\Phi\|_{L^\infty(\R^2_+)} \aleq r^{\frac{p-1}{p}}.
\]
Moreover, in view of Proposition~\ref{pr:hodgewithestimate} and with $\supp F \subset B^+(x_0,r)$ and $\|F\|_{2} \leq 1$,
\[
\|(1-\tilde{\eta})\, |\Phi|\|_{\infty} \aleq r^{-\frac{1}{2}}\ \|F\|_{1} \aleq r^{\frac{1}{2}}.
\]
Next, by from the above estimate we obtain in particular,
\[
 \|\Phi\|_{2,B^+(x_0,2\sqrt{r})\backslash B^+(x_0,\sqrt{r})} \aleq \sqrt{r}\|\Phi\|_{\infty,\R^2_+ \backslash B^+(x_0,\sqrt{r})} \aleq r,
\]
which implies
\[
 r^{\frac{\sigma-1}{2}} \|\Phi\|_{2,B^+(x_0,2\sqrt{r})\backslash B^+(x_0,\sqrt{r})} + \sqrt{r}^{-1}\, \|\Phi\|_{2,B^+(x_0,2\sqrt{r})\backslash B^+(x_0,\sqrt{r})} \aleq r^{\frac{\sigma+1}{2}} + r^{\frac{1}{2}}.
\]
Finally,
\[
 \|\nabla \brac{P_{ij}\, \Phi}\|_{2,\R^2_+ \backslash B^+(x_0,\sqrt{r})} \aleq \|\nabla P\|_{2,\R^2_+} \|\Phi\|_{\infty,\R^2_+ \backslash B^+(x_0,\sqrt{r})}  + \|\nabla \Phi\|_{2,\R^2_+ \backslash B^+(x_0,\sqrt{r})}, 
\]
which again in view of Proposition~\ref{pr:hodgewithestimate} implies
\[
 \|\nabla \brac{P_{ij}\, \Phi}\|_{2,\R^2_+ \backslash B^+(x_0,\sqrt{r})} \aleq r^{+\frac{1}{2}}.
\]
\end{proof}

\begin{proof}[Proof of Proposition~\ref{pr:West}]
This follows directly from Lemma~\ref{la:West:IH} and Lemma~\ref{la:West:IPhi}.
\end{proof}

\subsection{Boundary decay estimate for V: Proof of Proposition~\ref{pr:Vgoal}}
Recall that $\solu \in H^{\frac{1}{2}}(\R,\R^N)$ is a solution of
\begin{equation}\label{eq:laphupde}
 \laph \solu^i = \omega_{ij}(\lapv \solu^j) + \err^i(\solU) \quad \mbox{in (-2,2)}.
\end{equation}
Take $R \in (0,1)$ and $p \in \dot{H}^{\frac{1}{2}}(\R,SO(n))$ from Theorem~\ref{th:gaugenonlocal}. Choosing $R$ possibly even smaller, we may assume that
\begin{equation}\label{eq:lapvpsmall}
 \sup_{x_0 \in \R} \|\lapv p\|_{2,I(x_0,R)} \leq \eps.
\end{equation}
From now on, we assume that $x_0 \in (-1,1)$ and that for some $k_0 \in \N$ large enough it holds that $r \in (0,2^{-k_0} R)$.

Denote by
\[
 H_{\frac{1}{2}}(a,b) := \lapv (ab) - \lapv a\ b- a\lapv b.
\]
Then equation \eqref{eq:laphupde} implies that for any $\varphi \in C_c^\infty(-2,2)$,
\begin{equation}\label{eq:laphupdetr}
\begin{split}
 \int_{\R} p_{ij} \lapv \solu^j \lapv \varphi =& \int_{\R} \brac{\lapv p_{ij}\ \lapv \solu^j + p_{ij}\omega_{jk}(\lapv\solu^k)}\varphi\\
 &+ \int_{\R} p_{ij}\err^j(\solU)\, \varphi - \int_{\R} H_{\frac{1}{2}}(p_{ij},\varphi)\, \lapv \solu^j.
\end{split}
 \end{equation}
We have the following estimate, see e.g. \cite[Lemma C.1]{Schikorra-eps}.
\begin{lemma}\label{la:lhslaph}
For any $r > 0$, $x_0 \in \R$, $k_0 \geq 5$ and any $f \in L^2(\R)$ we find $\varphi \in C_c^\infty(I(2^{k_0}r,x_0))$,
\begin{equation}\label{la:lhs:infty21}
\|\varphi\|_{\infty,\R} + \|\lapv \varphi\|_{(2,1),\R} \leq 1
\end{equation}
so that for any 
\[
\begin{split}
\|f\|_{(2,\infty),I(x_0,r)}  \aleq \int_{\R} f\ \lapv \varphi  + \tail{\frac{1}{2}}{\|f\|_{(2,\infty)}}{x_0}{r}{k_0}.
\end{split}
\]
\end{lemma}
\begin{proof}[Proof of Proposition~\ref{pr:Vgoal}]
We apply Lemma~\ref{la:lhslaph} to $f := p_{ij} \lapv \solu^j$, and in view of \eqref{eq:laphupdetr} we find
\[
\begin{split}
 \|\lapv \solu\|_{(2,\infty),I(x_0,r)} \aeq& \|p_{ij} \lapv \solu^j\|_{(2,\infty),I(x_0,r)} \\
 \aleq& \int \brac{\lapv p_{ij}\ \lapv \solu^j + p_{ij}\omega_{jk}(\lapv \solu^k)} \varphi\\
 &+ \int p_{ij}\err^j(\solU)\varphi - \int H_{\frac{1}{2}}(p_{ij},\varphi) \lapv \solu^j\\
 &+\tail{\sigma}{\|\lapv \solu\|_{(2,\infty)}}{x_0}{r}{k_0}\\
\end{split}
 \]
Proposition~\ref{pr:Vgoal} then follows from Lemma~\ref{la:Vgoal:1}, \ref{la:Vgoal:2}, and \ref{la:Vgoal:3} below.
\end{proof}
\begin{lemma}\label{la:Vgoal:1}
\[
\begin{split}
 &\int \brac{\lapv p_{ij}\ \lapv \solu^j + p_{ij}\omega_{jk}(\lapv \solu^k)} \varphi\\
 \aleq& \eps\, \|\lapv \solu\|_{(2,\infty),I(x_0,2^{k_0} r)} + \tail{\sigma}{\|\lapv \solu \|_{(2,\infty)}}{x_0}{r}{k_0}.
\end{split}
 \]
\end{lemma}
\begin{proof}
With the definition of $\omega$ we find
\[
 \begin{split}
 & \int \brac{\lapv p_{ij}\ \lapv \solu^j + p_{ij}\omega_{jk}(\lapv \solu^k)} \varphi\\
  =&\int\int \brac{\lapv p_{ik}(x)\ \delta_{xy}  +  p_{ij}(x)\omega_{jk}(x,y)}\ \lapv \solu^k(y)\ \varphi(x) dx\ dy.
  \end{split}
\]
By Theorem~\ref{th:gaugenonlocal} and \eqref{la:lhs:infty21} we therefore conclude
\[
\begin{split}
&\int \brac{\lapv p_{ij}\ \lapv \solu^j + p_{ij}\omega_{jk}(\lapv \solu^k)} \varphi\\
\aleq &\eps\, \|\lapv \solu\|_{(2,\infty),I(x_0,2^{k_0} r)} + \tail{\sigma}{\|\lapv \solu \|_{(2,\infty)}}{x_0}{r}{k_0}.
\end{split}
 \]
\end{proof} 
\begin{lemma}\label{la:Vgoal:2}
For $R$ possibly smaller, for any large enough $k_0 \in \N$ and any $r \in (0,2^{-k_0} R)$ the following holds for some uniform $\sigma > 0$.
\[
\begin{split}
 \left |\int_{\R}  \err^j(\solU)\, p_{ij}\varphi \right | \aleq& \eps\,  \brac{\|\lapv \solu\|_{(2,\infty),I(x_0,2^{2k_0} r)} + \|\nabla \solW\|_{(2,\infty),B(2^{2k_0}r,x_0)}}\\
 &+ \tail{\sigma}{\|\lapv \solu\|_{(2,\infty)}}{x_0}{r}{2k_0}+\tail{\sigma}{\|\nabla \solW\|_{(2,\infty)}}{x_0}{r}{2k_0}.\\
 &+\brac{2^{2k_0} r}^\sigma.
\end{split}
 \]
\end{lemma}
\begin{proof}
This follows from the condition on $\err$, condition~\ref{cond:err}, observing that $\supp \varphi  \subset I(x_0,2^{k_0}r)$ and
\[
 \|p_{ij}\varphi\|_{\infty} + \|\lapv \brac{p_{ij} \varphi}\|_{2,\R} \aleq 1+\|\lapv p\|_{2} \aleq 1.
\]
\end{proof}

\begin{lemma}\label{la:Vgoal:3}
For $R$ as above, for any large enough $k_0 \in \N$ and any $r \in (0,2^{-k_0} R)$ the following holds for some uniform $\sigma > 0$.
\[
 \left |\int_{\R} H_{\frac{1}{2}}(p_{ij},\varphi) \lapv \solu \right | \aleq \eps \|\lapv \solu\|_{(2,\infty),I(x_0,2^{2k_0})} + \tail{\sigma}{\|\lapv \solu\|_{(2,\infty)}}{x_0}{r}{2k_0}.
\]
\end{lemma}
\begin{proof}
This follows from the estimates of $H_{\frac{1}{2}}$, see, e.g., \cite[Lemma A.7.]{Blatt-Reiter-Schikorra-2016}, which imply
\[
\begin{split}
 &\left |\int_{\R} H_{\frac{1}{2}}(p_{ij},\varphi) \lapv \solu \right | \\
 \aleq& \brac{\|\lapv p\|_{2,I(x_0,2^{2k_0}r)}+2^{-k_0} \|\lapv p\|_{2,\R}}\, \|\lapv \varphi\|_{2,\R}\, \|\lapv \solu\|_{(2,\infty),I(x_0,2^{2k_0}r)}\\
 &+ \|\lapv p\|_{2,\R}\, \|\lapv \varphi\|_{2,\R}\, \tail{\sigma}{\|\lapv \solu\|_{(2,\infty)}}{x_0}{r}{2k_0}.
 \end{split}
\]
We conclude observing \eqref{eq:lapvuestgoal}.
\end{proof}
\section{Optimal gauge for nonlocal antisymmetric potentials}\label{s:optimalgauge}
For $D \subset \R^2$, take any orthonormal frame $e_i \in L^\infty\cap H^1(D,\R^N)$, $i=1,\ldots,K$ for some $K \leq N$. That is, assume that pointwise almost everywhere in $D$
\[
 \langle e_i , e_j \rangle_{\R^N} = \delta_{ij}.
\]
The moving frame technique by H\'elein \cite{Helein-1991}, see also \cite{Helein-book}, tells us, that one can transform each such frame $(e_i)_{i=1}^K$ into a different orthonormal frame $(\tilde{e}_i)_{i=1}^K$ so that pointwise almost everywhere
\[
 \operatorname{span} \left \{e_i,\ i=1, \ldots, K\right \} = \operatorname{span} \left \{\tilde{e}_i,\ i=1, \ldots, K \right \}
\]
and so that additionally, denoting by
\[
 \Omega_{ij} := \langle \nabla \tilde{e}_i, \tilde{e}_j,  \rangle_{\R^N} \in L^2(D,\R^2),
\]
we have
\begin{equation}\label{eq:divframezero}
 \div \brac{\langle \nabla \tilde{e}_i, \tilde{e}_j,  \rangle_{\R^N}} = 0 \quad \mbox{in $D$}.
\end{equation}
This is good news for the regularity theory of harmonic maps into manifolds: assuming the existence of an initial orthonormal tangent vector field of $\mfdM$, $\tau_i := \tau_i \in T\mfdM$, and setting $e_i := \tau_i \circ \solu$ one can find a new moving frame $\tilde{e}_i$ so that the harmonic map equation of $\solu$,
\begin{equation}\label{eq:harmmapeqfr}
 \lap \solu \perp T_\solu \mfdM \quad \mbox{in $D$}
\end{equation}
can be transformed into
\[
 \div \brac{\langle \tilde{e}_i,\nabla \solu \rangle_{\R^N} } \overset{\eqref{eq:harmmapeqfr}}{=} \langle \nabla \tilde{e}_i,\nabla \solu \rangle_{\R^N}  = \langle \nabla \tilde{e}_i,\tilde{e}_k \rangle_{\R^N}\ \langle \tilde{e}_k,\nabla \solu \rangle_{\R^N}.
\]
In view of \eqref{eq:divframezero}, the right-hand side of this equation has now a div-curl structure, up to the term $\tilde{e}_k$, and using the Hardy-space estimates of div-curl quantities by Coifman-Lions-Meyer-Semmes seminal \cite{CLMS-1993} one can obtain H\"older continuity of solutions.

In the celebrated work by Rivi\`{e}re \cite{Riviere-2007}, he discovered that by an adaption of Uhlenbeck's work on gauges \cite{Uhlenbeck-1982} the condition \eqref{eq:divframezero} can be obtained for any antisymmetric matrix $\Omega_{ij} = -\Omega_{ji} \in L^2(D,\R^2)$, and that, under a smallness condition on of $\|\Omega \|_{2,D}$ one find $P \in H^1(D,SO(N))$ so that 
\[
 \Omega^P := \nabla P\ P^T + P \Omega P^T
\]
satisfies 
\[
 \div(\Omega^P) = 0.
\]
Just as in the harmonic map case, this leads to a regularity theory for systems of the form \eqref{eq:rivieresystem}, which, as Rivi\`{e}re showed, is the general structure of many geometric equations, see also \cite{Riviere-Struwe-2008}.

One can also obtain $P \in H^1(D,SO(N))$ without the smallness assumption on $\|\Omega\|_{2}$, which was proven in \cite{Schikorra-2010} motivated by the arguments by H\'elein for moving frames \cite{Helein-1991,Chone-1995}. Indeed, we have the following.
\begin{proposition}\label{pr:gaugenormal}
Let $\Omega_{ij} \in L^2(\R^2_+)$, $1 \leq i,j \leq N$. Then there exists $P \in L^\infty(\R^2_+,SO(N))$, $\nabla P \in L^2(\R,\R^N)$, $P \equiv I$ on $\R \times \{0\}$ that minimizes
\[
 \mathscr{E}(P) := \int_{\R^2_+} |\nabla P + P \Omega|^2.
\]
If $\Omega_{ij} = - \Omega_{ji}$ almost everywhere for $1 \leq i,j \leq N$ then this minimizer satisfies
\begin{equation}\label{eq:divOmegaPzero}
 \dv (\nabla P\ P^T + P \Omega P^T) = 0 \quad \mbox{in $\R^2_+$}.
\end{equation}
\end{proposition}
\begin{proof}
Clearly, 
\[
 \|\nabla P\|^2_{2,\R^2_+} \aleq \int_{\R^2_+} |\nabla P + P \Omega|^2 + \|\Omega\|_{2,\R^2_+}^2.
\]
Observe that $P \in SO(N)$ implies that $P$ is uniformly bounded. In particular, for any minimizing sequence $P_k \in \dot{H}^1\cap L^\infty (\R^2_+,SO(N))$ of $\mathscr{E}$ and for any compact set $K \subset \R^2_+$ we find a subsequence converging weakly in $H^1(K,SO(N))$ and strongly almost everywhere. By a diagonal argument and dominated convergence theorem we thus find a limit map $P \in \dot{H}^1\cap L^\infty(\R^2_+,SO(N))$ so that $P \equiv I$ on $\R \times \{0\}$ and so that
\[
 \int_{\R^2_+} \brac{\nabla P + P \Omega}: F = \lim_{k \to \infty} \int_{\R^2_+} \brac{\nabla P_k + P_k \Omega}: F \quad \mbox{for any $F \in C_c^\infty(\R^2_+,\R)$}.
\]
Here $A:B$ denotes the Hilbert-Schmidt scalar product of matrices. Therefore, by duality, $P$ is the minimizer of $\mathscr{E}$.

We compute the Euler-Lagrange equations \eqref{eq:divOmegaPzero}. For any $\Phi \in C_c^\infty(\R^2_+)$ and any constant antisymmetric matrix $\alpha \in so(N)$ we define 
\[
 P_\delta := e^{\delta \alpha \varphi} P.
\]
Observe that $P_\delta$ belongs to $SO(N)$ pointwise almost everywhere since $P \in SO(N)$. Thus, $P_\delta \in\dot{H}^1\cap L^\infty(\R^2_+,SO(N))$ and 
\begin{equation}\label{eq:Pdeltazero}
 \frac{d}{d\delta} \Big|_{\delta = 0} \mathscr{E}(P_\delta) = 0.
\end{equation}
Now compute
\[
 \frac{d}{d\delta} \Big|_{\delta =0} P_\delta  = \varphi \alpha P,
\]
and thus
\[
 \frac{d}{d\delta} \Big|_{\delta =0} \nabla P_\delta  = \nabla \varphi \alpha P + \varphi \alpha \nabla P.
\]
In particular,
\[
 \frac{d}{d\delta} \Big |_{\delta = 0} \nabla P_\delta + P_\delta \Omega = \nabla \varphi\, \alpha P + \varphi\, \alpha \brac{\nabla P + P \Omega}.
\]
Observe that 
\[
 \alpha \brac{\nabla P + P \Omega} : \brac{\nabla P + P \Omega} = 0
\]
by the antisymmetry of $\alpha$. Thus,
\[
 \frac{d}{d\delta} \Big |_{\delta = 0} |\nabla P_\delta + P_\delta \Omega |^2 = 2 \nabla \varphi\, \alpha P :\brac{\nabla P + P \Omega} = 2 \nabla \varphi\, \alpha :\brac{\nabla P\, P^T + P \Omega P^T}
\]
Plugging this into \eqref{eq:Pdeltazero} we have found that for any constant antisymmetric matrix $\alpha \in so(N)$,
\[
 \div(\alpha :\brac{\nabla P\, P^T + P \Omega P^T}) = 0 \quad \mbox{in $\R^2_+$}.
\]
Since $\Omega$ is antisymmetric, so is $\nabla P\, P^T + P \Omega P^T$. Consequently, we have found 
\[
 \div\brac{\nabla P\, P^T + P \Omega P^T}) = 0 \quad \mbox{in $\R^2_+$}.
\]
\end{proof}

\subsection{Nonlocal version}
For functions $\Omega_{ij}$ the theory of finding an optimal gauge has been generalized to other operators. Most notably, Da Lio  and Rivi\`{e}re \cite{DaLio-Riviere-1Dmfd} showed that by adapting the arguments of Uhlenbeck and Rivi\`ere \cite{Uhlenbeck-1982,Riviere-2007} for \emph{functions} $\omega_{ij} = -\omega_{ji} \in L^2(\R)$ one can find a map $p \in \dot{H}^{\frac{1}{2}}(\R,SO(N))$ so that 
\begin{equation}\label{eq:lapvpomegal21}
\lapv p + p\, \omega \in L^{(2,1)}(\R).
\end{equation}
This has been extended to various situations, see \cite{DaLio-nmfd,Riviere-2011,Schikorra-unp,Schikorra-eps,DaLio-Riviere-CAG,Mazowiecka-Schikorra-2017}. 

In our setting, $\omega$ is not a function, but is a nonlocal functional, $\omega_{ij}: \dot{H}^{\frac{1}{2}}(\R) \to L^1(\R)$ given as
\[
  \int_{\R} \omega_{ij}(f)\ \varphi := \int_{\R}\int_{\R} \omega_{ij}(x,y)\, f(y)\ \varphi(x)\ dy\, dx.
\]
Here $\omega_{ij}(x,y)$ is supposed to be measurable and to satisfy certain localization properties, namely conditions~\ref{cond:omega}. In particular, we assume
\begin{equation}\label{eq:normomega}
 \|\omega\| := \sup_{\|f_{ij}\|_{2,\R} \leq 1, \|\zeta\|_{\infty,\R} \leq 1}  \int_{\R} \int_{\R} \omega_{ij}(x,y)\, f_{ij}(y)\, \zeta(x) dx\ dy< \infty.
\end{equation}
We say that $\omega$ is antisymmetric if $\omega_{ij}(x,y) = -\omega_{ji}(x,y)$ almost everywhere for any $i,j = 1\ldots,N$.

In Proposition~\ref{pr:gaugenormal}, following the strategy from \cite{Schikorra-frame, Schikorra-eps}, we found the good gauge $P$ by minimizing the energy $\mathscr{E}(P) := \int_{\R^2_+} |\Omega^P|^2$, where $\Omega^P = \nabla P + P \Omega$. Here, the role of $\Omega^P$ is replaced by $\omega^p_{ij} \in L^2(\R)$ which for $p \in \dot{H}^{\frac{1}{2}}(\R,SO(N))$ is defined as
\begin{equation}\label{eq:omegapij}
 \omega^p_{ij}(y):= \lapv p_{ij}(y) + \int p_{ik}(x)\ \omega_{kj} (x,y) dx \in L^2(\R).
\end{equation}
Here, we obtain $p$ as a minimizer of the energy
\[
 \mathscr{E}(p) := \|\omega^p\|_{L^2(\R,\R^{n\times n})}^2
\]
\begin{theorem}\label{th:gaugenonlocal}
Assume that $\omega$ is antisymmetric and satisfies conditions~\ref{cond:omega}. Then there exists a $p \in L^\infty \cap \dot{H}^{\frac{1}{2}}(\R,SO(N))$ that minimizes $\mathscr{E}(p)$ 
in the class of maps $p \in \dot{H}^{\frac{1}{2}}(\R,SO(N))$.

For this minimizer $p$ we have
\[
 \omega^p_{ij} \in L^{(2,1)}_{loc}(\R).
\]
Moreover the following estimate holds: for any $\eps > 0$ we find $R \in (0,1)$ so that for any $\varphi \in C_c^\infty(I(x_0, r))$ and any $f \in L^2(\R)$, where $x_0 \in \R$ and $r \in (0,2^{-k_0} R)$ for a sufficiently large $k_0 \in \N$,
\begin{equation}\label{eq:gaugenonlocal:estimate}.
\begin{split}
 &\int_{\R} \int_{\R} \brac{\lapv p_{\ell j}(x) \delta_{xy}+   \, p_{\ell k}(x)\, \omega_{kj} (x,y)}\,f(y)\, \varphi(x)dx\, dy\\
 \aleq &\brac{\eps\, \|f\|_{(2,\infty),B(x_0,2^{k_0} r)} + \tail{\sigma}{\|f\|_{(2,\infty)}}{x_0}{r}{k_0}}\ \brac{\|\varphi\|_{\infty,\R} + \|\lapv \varphi\|_{2,\R}}.
\end{split}
 \end{equation}
Here $\sigma$ is a uniform constant.
\end{theorem}
Firstly, we have the following observations:
\begin{lemma}\label{la:observationsmatscrEnl}
$\mathscr{E}$ satisfies
\begin{equation}\label{eq:coercv}
 \mathscr{E}(p) \ageq \|\lapv p\|_{2,\R} - \|\omega\|. 
\end{equation}
and
\begin{equation}\label{eq:nlgauge:bound}
 0 \leq \inf_{p \in \dot{H}^{\frac{1}{2}}(\R,SO(N))} \mathscr{E}(p) \leq \|\omega\|.
\end{equation}
Also, $\mathscr{E}$ is sequentially lower semicontinuous with respect to the weak topology on $\dot{H}^{\frac{1}{2}}(\R,SO(N))$.

In particular, there exists $p \in \dot{H}^{\frac{1}{2}}(\R,SO(N))$ minimizing the energy $\mathscr{E}$ among all maps in $\dot{H}^{\frac{1}{2}}(\R,SO(N))$.
\end{lemma}
\begin{proof}
By duality and in view of \eqref{eq:normomega} we have for any $p \in L^\infty \cap \dot{H}^{\frac{1}{2}}(\R,SO(N))$,
\[
 \|\omega^p\|_{L^2(\R)} = \sup_{f_{ij} \in C_c^\infty(\R), \|f_{ij}\|_{2,\R} \leq 1} \int_{\R} \omega^p_{ij}(y)\ f_{ij}(y) \geq \|\lapv p_{ij}\|_{2,\R} - \|\omega_{ij}\|.
\]
This shows \eqref{eq:coercv}. The second claim, \eqref{eq:nlgauge:bound}, is obvious by taking $p$ the identity, $p_{ij} := \delta_{ij}$.

As for lower semicontinuity, observe that, as in the local case, any bounded sequence $p_k \in \dot{H}^{\frac{1}{2}}(\R,SO(N))$ is for any compact $K \subset \R$ uniformly bounded in $H^{\frac{1}{2}}(K,\R^N)$, because $SO(N)$ is a compact manifold. Via a diagonal argument, up to taking a subsequence, we can assume that $p^k(y) \to p(y)$ pointwise almost everywhere in $\R$. In particular $p \in H^{\frac{1}{2}}(\R,SO(N))$. 
Also,
\[
 \mathscr{E}(p) \leq \liminf_{k \to \infty}\mathscr{E}(p_k)
\]
follows from the lemma of Fatou, or by duality, since we can write
\[
 \mathscr{E}(p) = \sup_{\|(f_{ij}) \|_{2,\R}\leq 1}\ \int_{\R} \lapv p_{ij}(y)\ f_{ij}(y)\ dy + \int_{\R} \int_{\R} p_{ik}(x)\, \omega_{kj}(x,y)\, f_{ij}(y)\ dy\ dx.
\]

With the observations of Lemma~\ref{la:observationsmatscrEnl}, we conclude the existence for a minimizer of $\mathscr{E}$ in the class $\dot{H}^{\frac{1}{2}}(\R,SO(N))$ by the direct method of the calculus of variations: any minimizing sequence has a subsequence which is weakly converging in $\dot{H}^{\frac{1}{2}}(\R,\R^{N \times N})$ and also almost everywhere converging to a map $p \in \dot{H}^{\frac{1}{2}}(\R,SO(N))$, which by sequential lower semicontinuity is the minimizer.
\end{proof}
Having a minimizer we can compute the Euler-Lagrange equations and find the following estimate.
\begin{lemma}\label{la:gaugeelest}
Let $p \in \dot{H}^{\frac{1}{2}}(\R,SO(N))$ be a minimizer of $\mathscr{E}$ in the class $\dot{H}^{\frac{1}{2}}(\R,SO(N))$. Then
\begin{equation}\label{eq:gaugeelest:2}
 \omega^p_{ij}(y) \in L^{(2,1)}_{loc}(\R).
\end{equation}
Moreover, for any $\eps > 0$ there exists $R > 0$ so that for any $\varphi \in C_c^\infty(I(x_0,r))$ where $x_0 \in \R$ and $r \in (0,R)$,
\begin{equation}\label{eq:gaugeelest:1}
 \int_{\R} p_{\ell j}(y)\, \omega^p_{ij}(y)\, \lapv \varphi(y) \leq C\ \eps\, \|\lapv \varphi\|_{(2,\infty),\R}.
\end{equation}
Here $C$ is a constant depending on the data $p$, $\omega$, but which is independent of $\varphi$ and $\eps$.
\end{lemma}
\begin{proof}
By a duality argument, \eqref{eq:gaugeelest:2} follows from \eqref{eq:gaugeelest:1}, see, e.g., \cite[Lemma C.1]{Schikorra-eps}.

Let us prove \eqref{eq:gaugeelest:2}. We define an admissible variation $p_\delta$ of $p$ as in Proposition~\ref{pr:gaugenormal}: for arbitrary $\varphi \in C_c^\infty(\R)$ and an arbitrary constant antisymmetric matrix $\alpha \in so(N)$ we set
\[
 p_\delta := e^{\delta \alpha \varphi} p \in \dot{H}^{\frac{1}{2}}(\R,SO(N)).
\]
The minimality condition for $p$ implies that
\begin{equation}\label{eq:gaugeeq0}
 0 = \frac{d}{d\delta} \Big |_{\delta =0} \mathscr{E}(p_\delta) = 2\int_{\R} \omega^p_{ij}(y) \frac{d}{d\delta} \Big|_{\delta =0}\omega^{p_\delta}_{ij}(y)\ dy.
\end{equation}
We have $p_\delta = p +\delta \varphi \alpha p + O(\delta^2)$, and therefore from the definition \eqref{eq:omegapij} of $\omega^p$ we obtain
\[
  \frac{d}{d\delta} \Big|_{\delta =0}\omega^{p_\delta}_{ij}(y) = \alpha_{i\ell} \lapv (\varphi(y)\, \ p_{\ell j}(y)) + \alpha_{i\ell}\int  \varphi(x)\ p_{\ell k}(x)\, \omega_{kj} (x,y)\, dx.
\]
Denoting again the failure of the Leibniz rule for $\lapv$ by $H_{\frac{1}{2}}$,
\begin{equation}\label{eq:h12ab:gauge}
 H_{\frac{1}{2}}(a,b) := \lapv(ab) - \lapv a\, b - a\, \lapv b
\end{equation}
we find
\[
 \begin{split}
 \frac{d}{d\delta} \Big|_{\delta =0}\omega^{p_\delta}_{ij}(y) =& \alpha_{i\ell}\lapv \varphi(y)\, p_{\ell j}(y)\\
 &+\alpha_{i\ell} \brac{\varphi(y)\, \lapv p_{\ell j}(y) + \int  \varphi(x)\, p_{\ell k}(x) \omega_{kj} (x,y) dx}\\
 &+ \alpha_{i\ell}H_{\frac{1}{2}}(\varphi(y),p_{\ell j}(y)).
 \end{split}
\]
Observe that
\[
\begin{split}
  &\varphi(y)\, \lapv p_{\ell j}(y) + \int_{\R} \varphi(x)\, p_{\ell k}(x) \omega_{kj} (x,y) dx\\
  =& \varphi(y)\, \omega^p_{\ell j}(y) + \int_{\R} \brac{\varphi(x)-\varphi(y)}\, p_{\ell k}(x)\, \omega_{kj} (x,y) dx
\end{split}
  \]
Using, that by antisymmetry of $\alpha$, 
\[
\omega^p_{ij}(y)\, \alpha_{i\ell}\, \omega^p_{\ell j}(y) \equiv 0,
\]
we find that 
\[
\begin{split}
0= \int_{\R} \omega^{p}_{ij}(y)\, \frac{d}{d\delta} \Big|_{\delta =0}\omega^{p_\delta}_{ij}(y)\, dy 
=&\int_{\R} \omega^{p}_{ij}(y)\, \alpha_{i\ell}\lapv \varphi(y)\, p_{\ell j}(y)\, dy \\
 &+\int_{\R} \omega^{p}_{ij}(y)\, \alpha_{i\ell} \int_{\R} \brac{\varphi(x)-\varphi(y)}\, p_{\ell k}(x)\, \omega_{kj} (x,y) dx\, dy \\
 &+ \int_{\R} \omega^{p}_{ij}(y)\, \alpha_{i\ell}H_{\frac{1}{2}}(\varphi(y),p_{\ell j}(y))\, dy .
\end{split}
 \]
This holds for any antisymmetric matrix $\alpha \in so(N)$, therefore in view of \eqref{eq:gaugeeq0} we have for any $i,\ell = 1,\ldots,N$,
\[
\begin{split}
0=&\int_{\R} \omega^{p}_{ij}(y)\, \lapv \varphi(y)\, p_{\ell j}(y)\, dy - \int_{\R} \omega^{p}_{\ell j}(y)\, \lapv \varphi(y)\, p_{i j}(y)\, dy\\
 &+\int_{\R} \omega^{p}_{ij}(y)\, \int_{\R} \brac{\varphi(x)-\varphi(y)}\, p_{\ell k}(x)\, \omega_{kj} (x,y) dx\, dy \\
 &-\int_{\R} \omega^{p}_{\ell j}(y)\, \int_{\R} \brac{\varphi(x)-\varphi(y)}\, p_{ik}(x)\, \omega_{kj} (x,y) dx\, dy \\
 &+ \int_{\R} \omega^{p}_{ij}(y)\, H_{\frac{1}{2}}(\varphi(y),p_{\ell j}(y))\, dy\\
  &- \int_{\R} \omega^{p}_{\ell j}(y)\, H_{\frac{1}{2}}(\varphi,p_{ij})(y)\, dy.
\end{split}
 \]
That is, for any $\varphi \in C_c^\infty(\R)$
\[
\begin{split}
2\int_{\R} p_{\ell j}(y) \omega^{p}_{ij}(y)\, \lapv \varphi(y)\,  dy=I_{\ell i}(\varphi) + II_{\ell i}(\varphi) - II_{i \ell}(\varphi) + III_{\ell i}(\varphi) - III_{i\ell}(\varphi),
\end{split}
 \]
where we set
\[
 I_{\ell i}(\varphi) := \int_{\R} \brac{p_{\ell j}(y)\, \omega^{p}_{ij}(y) + p_{i j}(y)\, \omega^{p}_{\ell j}(y)}\, \lapv \varphi(y)\, dy,
\]
\[
 II_{\ell i}(\varphi) := \int_{\R}\int_{\R} \omega^{p}_{\ell j}(y)\,  p_{ik}(x)\, \omega_{kj} (x,y)\, \brac{\varphi(x)-\varphi(y)}\, dx\, dy,
\]
and
\[
 III_{\ell i}(\varphi) = \int_{\R} \omega^{p}_{\ell j}(y)\, H_{\frac{1}{2}}(\varphi,p_{ij})(y)\, dy.
\]
In order to show the claim \eqref{eq:gaugeelest:1} we need to show that for any $\eps > 0$ we find some $R > 0$ so that 
\begin{equation}\label{eq:gaugeelest:goal}
 |I_{\ell i}(\varphi)|  + |II_{\ell i}(\varphi)| + |III_{\ell i}(\varphi)| \aleq \eps\, \|\lapv \varphi\|_{(2,\infty),\R}
\end{equation}
holds for any $\varphi \in C_c^\infty(I(x_0,r))$ for arbitrary $x_0 \in \R$ and $r \in (0,R)$.

We will choose below a large $k_0$ so that 
\begin{equation}\label{eq:nlgauge:k0choice}
 2^{-k_0 \sigma} \brac{\|\lapv p\|_{2,\R}+\|\omega^p\|_{2,\R}} < \frac{\eps}{2},
\end{equation}
and then $R$ sufficiently small so that
\begin{equation}\label{eq:nlgauge:Rchoice}
 \sup_{x_0 \in \R}\brac{\|\lapv p\|_{2,I(x_0,2^{k_0} R)} + \|\omega^p\|_{2,I(x_0,2^{k_0} R)}} < \frac{\eps}{2}.
\end{equation}
\underline{We begin to estimate $I_{\ell i}(\varphi)$.} Firstly, $\lapv (p_{\ell j} p_{ij}) =\lapv (\delta_{\ell i})= 0$ because $p \in SO(N)$ almost everywhere. Therefore, by the definition of $H_{\frac{1}{2}}$, see \eqref{eq:h12ab:gauge}, 
\[
  p_{\ell j}(y) \lapv p_{ij}(y) +p_{i j}(y)\lapv p_{\ell j}(y) = - H_{\frac{1}{2}}(p_{\ell j},p_{ij}).
\]
Moreover, the antisymmetry of $\omega$, $\omega_{kj}(x,y) = - \omega_{jk}(x,y)$, implies
\[
\int_{\R} p_{i j}(y)\, p_{\ell k}(x)\, \omega_{kj} (x,y)\, dx = - \int_{\R}  p_{i k}(y)\, p_{\ell j}(x)\, \omega_{kj} (x,y)\, dx.
\]
Thus,
\[
\begin{split}
 &\int p_{\ell j}(y)\, p_{ik}(x)\, \omega_{kj}(x,y) + p_{ij}(y)\, p_{\ell k}(x)\, \omega_{k j}(x,y)\ dx\\
 =& -\int \brac{p_{\ell j}(x)-p_{\ell j}(y)}\, p_{ik}(x)\, \omega_{kj} (x,y)\ dx\\
 &+ \int \brac{p_{i k}(x)\, -  p_{i k}(y)}p_{\ell j}(x)\, \omega_{kj} (x,y)\ dx\\
\end{split}
 \]
Consequently, by the definition of $\omega^p$, see \eqref{eq:omegapij},  we find
\[
\begin{split}
|I_{\ell i}(\varphi)|  \aleq&  \int_{\R} |H_{\frac{1}{2}}(p,p)|\ |\lapv \varphi|\\
 &+\left |\int_{\R} \int_{\R} \brac{p_{\ell j}(x)- p_{\ell j}(y) }\, p_{ik}(x)\ \omega_{kj}(x,y)\ \lapv \varphi(y)\ dy\ dx \right |\\
 &+\left |\int_{\R} \int_{\R} \brac{p_{ik}(x)- p_{ik}(y) }\, p_{\ell j}(x)\ \omega_{kj}(x,y)\ \lapv \varphi(y)\ dy\ dx \right |\\
\end{split}
 \]
Using the estimates for $H_{\frac{1}{2}}$, see e.g. \cite[Lemma A.7.]{Blatt-Reiter-Schikorra-2016}, and the support of $\varphi \in C_c^\infty(I(x_0,r))$, for any $k_0$ sufficiently large,
\[
\begin{split}
 &\int_{\R} |H_{\frac{1}{2}}(p,p)|\ |\lapv \varphi|\\
 \aleq &\brac{\|\lapv p\|_{2,\R} \|\lapv p\|_{2,I(x_0,2^{k_0} r)} + 2^{-k_0 \sigma} \|\lapv p\|_{2,\R}^2 } \|\lapv \varphi\|_{(2,\infty)}.
\end{split}
 \]
In view of \eqref{eq:nlgauge:Rchoice} and \eqref{eq:nlgauge:k0choice} we conclude that 
\[
 \int_{\R} |H_{\frac{1}{2}}(p,p)|\ |\lapv \varphi| \aleq \eps\, \|\lapv \varphi\|_{(2,\infty)}.
 \]
Next, letting again $\eta \in C_c^\infty(B(0,1))$ a typical bump function constantly one in $B(0,\frac{1}{2})$ and denoting
\[
 \eta_{k} := \eta((x_0-x)/{2^k r}),
\]
we have
\[
\begin{split}
 &\left |\int_{\R} \int_{\R} \brac{p_{\ell j}(x)- p_{\ell j}(y) }\, p_{ik}(x)\ \omega_{kj}(x,y)\ \lapv \varphi(y)\ dy\ dx \right |\\
 \aleq&\left |\int_{\R} \int_{\R} \brac{p_{\ell j}(x)- p_{\ell j}(y) }\, p_{ik}(x)\ \omega_{kj}(x,y)\ \eta_{k_0} \lapv \varphi(y)\ dy\ dx \right |\\
 &+\left |\int_{\R} \int_{\R} \brac{p_{\ell j}(x)- p_{\ell j}(y) }\, p_{ik}(x)\ \omega_{kj}(x,y)\ (1-\eta_{k_0}) \lapv \varphi(y)\ dy\ dx \right |\\
\end{split}
 \]
For the first term, from \eqref{eq:oa:2}, possibly choosing a larger $k_0$ and a smaller $R$,
 \[
\left |\int_{\R} \int_{\R} \brac{p_{\ell j}(x)- p_{\ell j}(y) }\, p_{ik}(x)\ \omega_{kj}(x,y)\ \eta_{k_0} \lapv \varphi(y)\ dy\ dx \right | \aleq \brac{\eps + 2^{-k_0 \sigma}}\, \|\lapv \varphi\|_{(2,\infty),\R}.
\]
For the second term, in view of the definition of $\|\omega\|$ in \eqref{eq:normomega} and $\mathcal{E}(p) \aleq \|\omega\|$ 
 \[
 \begin{split}
  &\left |\int_{\R} \int_{\R} \brac{p_{\ell j}(x)- p_{\ell j}(y) }\, p_{ik}(x)\ \omega_{kj}(x,y)\ (1-\eta_{k_0}) \lapv \varphi(y)\ dy\ dx \right |\\ 
  \aleq& \|\omega\| \|(1-\eta_{k_0}) \lapv \varphi\|_{2,\R}.
 \end{split}
 \]
By the disjoint support of $(1-\eta_{k_0})$ and $\varphi$ we have, see, e.g. \cite[Lemma A.1]{Blatt-Reiter-Schikorra-2016},
\[
 \|(1-\eta_{k_0}) \lapv \varphi\|_{2,\R}. \aleq 2^{-k_0 \sigma}\ \|\lapv \varphi\|_{(2,\infty),\R}.
\]
Thus, in view of \eqref{eq:nlgauge:Rchoice} and \eqref{eq:nlgauge:k0choice}, for large enough $k_0$ and small $R$,
\[
 \left |\int_{\R} \int_{\R} \brac{p_{\ell j}(x)- p_{\ell j}(y) }\, p_{ik}(x)\ \omega_{kj}(x,y)\ \lapv \varphi(y)\ dy\ dx \right | \aleq \eps\, \|\omega\|\, \|\lapv \varphi\|_{(2,\infty),\R}.
\]
This concludes the estimate of $|I_{\ell i}(\varphi)|$.

\underline{The estimate for $II_{\ell i}(\varphi)$} follows from directly from condition \eqref{eq:oa:3}.

\underline{We estimate $III_{\ell i}(\varphi)$.}
Recall our notation of cutoff functions. For a typical bump function $\eta \in C_c^\infty(B(0,1))$ which is constantly one in $B(0,\frac{1}{2})$ we denote
\[
 \eta_{k} := \eta((x_0-x)/{2^k r}),\ \xi_k := \eta_k - \eta_{k-1}.
\]
Then
\[
 |III_{\ell i}(\varphi)| \aleq \|\omega^{p}_{\ell j}\|_{2,I(x_0,2^{k_0}r)}\, \|H_{\frac{1}{2}}(\varphi,p_{ij})(y)\|_{2,\R} + \|\omega^{p}_{\ell j}\|_{2} \sum_{k=k_0}^\infty \|\xi_k H_{\frac{1}{2}}(\varphi,p_{ij})\|_{2}
\]
Now, by the estimates for $H_{\frac{1}{2}}$, see e.g. \cite[Theorem 7.1.]{Lenzmann-Schikorra-commutators},
\[
 \|H_{\frac{1}{2}}(\varphi,p_{ij})(y)\|_{2,\R} \aleq \|\lapv \varphi\|_{(2,\infty),\R}\ \|\lapv p\|_{2,\R}.
\]
Moreover, by the disjoint support of $\varphi$ and $\xi_k$, see e.g. \cite[Lemma A.7.]{Blatt-Reiter-Schikorra-2016}
\[
 \|\xi_k H_{\frac{1}{2}}(\varphi,p_{ij})\|_{2} \aleq 2^{-\sigma k}\ \|\lapv \varphi\|_{(2,\infty),\R}\ \|\lapv p\|_{2,\R}.
\]
Consequently, in view of \eqref{eq:nlgauge:Rchoice} and \eqref{eq:nlgauge:k0choice},
\[
 |III_{\ell i}(\varphi)| \aleq \eps\|\lapv \varphi\|_{(2,\infty),\R}.
\]
\end{proof}
Now we can prove Theorem~\ref{th:gaugenonlocal}
\begin{proof}[Proof of Theorem~\ref{th:gaugenonlocal}]
Lemma~\ref{la:observationsmatscrEnl} ensures the existence of a minimizer $p$. The $L^{(2,1)}_{loc}$-integrability follows directly from Lemma~\ref{la:gaugeelest}. 

It remains to prove the estimate \eqref{eq:gaugenonlocal:estimate}. Let $\eps > 0$ be given, $r \in (0,2^{-k_0} R)$ for an $R>0$ and $k_0 \in \N$ to be chosen later.

Let $x_0 \in \R$ and $\varphi \in C_c^\infty(I(x_0, r))$ and $f \in L^2(\R)$ be arbitrary. With the definition \eqref{eq:omegapij} of $\omega^p$ we find
\[
\begin{split}
 &\int_{\R} \int_{\R} \brac{\lapv p_{\ell j}(x) \delta_{xy}+   \, p_{\ell k}(x)\, \omega_{kj} (x,y)}\,f(y)\, \varphi(x)\, dx\, dy\\
  =&\int_{\R} f(y)\ \varphi(y)\, \omega^p_{\ell j}(y)\ dy +\int_{\R} \int_{\R} f(y)\, p_{\ell k}(x)\, (\varphi(x)-\varphi(y)) \omega_{kj} (x,y)\, dx\, dy.
 \end{split}
\]
By condition~\eqref{eq:oa:1}, for $k_0$ large enough and $R$ small enough, we have
\[
\begin{split}
 &\int_{\R} \int_{\R} f(y)\, p_{\ell k}(x)\, (\varphi(x)-\varphi(y)) \omega_{kj} (x,y)\, dx\, dy\\
 \aleq &\brac{\eps\, \|f\|_{(2,\infty),B(x_0,2^{k_0} r)} + \tail{\sigma}{\|f\|_{(2,\infty)}}{x_0}{r}{k_0}}\ \brac{\|\varphi\|_{\infty,\R} + \|\lapv \varphi\|_{2,\R}}.
\end{split}
 \]
We rewrite the remaining term. Since $p \in SO(N)$ pointwise a.e., 
\[
 \int_{\R} f(y)\, \varphi(y)\, \omega^p_{\ell j}(y)\ dy = \int_{\R} f(y)\, \varphi(y)\,  p_{k j}(y)\ \brac{p_{k i}(y) \omega^p_{\ell i}(y)}\ dy.
\]
Recall our notation of cutoff functions. For a typical bump function $\eta \in C_c^\infty(I(0,1))$ which is constantly one in $I(0,\frac{1}{2})$ we denote
\[
 \eta_{k} := \eta((x_0-x)/{2^k r}),\ \xi_k := \eta_k - \eta_{k-1}.
\]
Also, recall that with $\lapms{\frac{1}{2}} \equiv (-\lap)^{-\frac{1}{4}}$ we denote the Riesz potential. Set
\[
 \phi_{k} := \eta_{k_0}\lapmv (f\, \varphi\,  p_{k j}),\quad \psi_{k} := \xi_{k}\lapmv (f\, \varphi\,  p_{k j}).
\]
Then
\[
  \left |\int_{\R} f(y)\, \varphi(y)\, \omega^p_{\ell j}(y)\ dy \right | \leq\left | \int_{\R} \lapv \phi_{k_0}\ p_{k i}(y)\, \omega^p_{\ell i}(y)\ dy \right | +\sum_{k=k_0}^\infty \|\omega^p\|_{2,\R}\ \|\lapv \psi_k\|_{2,\R}
\]
In view of Lemma~\ref{la:gaugeelest}
\[
 \left | \int_{\R} \lapv \phi_{k_0}\ p_{k i}(y)\, \omega^p_{\ell i}(y)\ dy \right | \aleq \eps\ \|\lapv \phi_{k_0}\|_{(2,\infty),\R}.
\]
So we have shown that
\[
  \left |\int_{\R} f(y)\, \varphi(y)\, \omega^p_{\ell j}(y)\ dy \right | \leq \eps\ \|\lapv \phi_{k_0}\|_{(2,\infty),\R} +\sum_{k=k_0}^\infty \|\omega^p\|_{2,\R}\ \|\lapv \psi_k\|_{2,\R}
\]
Now, in view of the fractional Leibniz rule, see e.g. \cite[Theorem 7.1.]{Lenzmann-Schikorra-commutators}, by Sobolev embedding, by H\"older inequality and the support of $\varphi$
\[
\begin{split}
 \|\lapv \phi_{k_0}\|_{(2,\infty)} \aleq & \|\lapv \eta_{k_0}\, \lapmv (f\, \varphi\,  p)\|_{(2,\infty),\R} + \|\eta_{k_0} f\, \varphi\,  p\|_{(2,\infty),\R}\\
 \aleq& (2^{k_0} r)^{-\frac{1}{2}} \|f\, \varphi\,  p\|_{1,\R} + \|\eta_{k_0} f\, \varphi\,  p\|_{(2,\infty),\R}\\
 \aleq& \|\varphi\|_{\infty,\R}\, \|f\|_{(2,\infty),\R}.
\end{split}
 \]
On the other hand, see, e.g., \cite[Lemma 3.6.]{Maalaoui-Martinazzi-Schikorra-2016},
\[
 \|\lapv \psi_k\|_{2,\R}\aleq 2^{-\sigma k} \|f\, \varphi\,  p\|_{(2,\infty),\R}\aleq 2^{-\sigma k}  \|\varphi\|_{\infty,\R}\, \|f\|_{(2,\infty),\R}.
\]
Thus we have shown
\[
  \left |\int_{\R} f(y)\, \varphi(y)\, \omega^p_{\ell j}(y)\ dy \right | \leq \eps\ \|\varphi\|_{\infty,\R}\, \|f\|_{(2,\infty),\R} +\|\omega^p\|_{2,\R}\, \sum_{k=k_0}^\infty 2^{-\sigma k}\ \|\varphi\|_{\infty,\R}\, \|f\|_{(2,\infty),\R}.
\]
This proves \eqref{eq:gaugenonlocal:estimate}. The proof of Theorem~\ref{th:gaugenonlocal} is finished.
\end{proof}

\section{Extension operators and commutators}\label{s:extcom}
Commutator estimates have played a crucial role for regularity theory for geometric equations. The most famous one might be the Coifman-Rochberg-Weiss commutator theorem.
\begin{theorem}[Coifman-Rochberg-Weiss \cite{CRW-1976}]\label{th:CRW}
Denote by $\Rz_i$ the $i$-th Riesz transform, $i =1,\ldots,n$. Then, for any $q \in [1,\infty]$,
\[
 \|f \Rz_i g -  \Rz_i (fg) \|_{(2,q),\R^n} \aleq [f]_{BMO}\ \|g\|_{(2,q)}.
\]
For a definition of $BMO$ we refer to Section~\ref{s:hardyhalf}.
\end{theorem}
If $f$ belongs to some Sobolev space, one can obtain the following statement that for some situations is stronger. For a proof we refer to \cite[Theorem 6.1.]{Lenzmann-Schikorra-commutators}.
\begin{theorem}\label{th:Hzcomm}
Let $1 \leq q, q_1,q_2 \leq \infty$ so that $\frac{1}{q} = \frac{1}{q_1} + \frac{1}{q_2}$. Then
\[
 \|f \Rz_i g -  \Rz_i (fg) \|_{(2,q),\R^n} \aleq \|(-\lap)^{\frac{n}{4}} f\|_{(2,q_1),\R^n}\ \|g\|_{(2,q_2),\R^n}.
\]
\end{theorem}
In this section we aim at finding a suitable generalization to Theorem~\ref{th:Hzcomm} where instead of Riesz transforms $\Rz_i$ we consider certain extension operators which take functions from $\R^n$ to functions on the upper halfspace $\R^{n+1}_+$. 

The extension operators we consider are the Poisson extension \[f \mapsto f^h := p_t \ast f,\] as well as their $\lapv$ derivative \[
f \mapsto Tf := \lapv f^h = (\lapv p_t)\ast f.
\] Here, $p$ denotes the Poisson kernel for $\R^{n+1}_+$,
\[
 p(x) = c\frac{1}{\brac{1+|x|^2}^{\frac{n+1}{2}}},
\]
and we denote $p_t(x) = t^{-n} p(x/t)$, that is
\[
 p_t(x) = c\frac{t}{\brac{t^2+|x|^2}^{\frac{n+1}{2}}}.
\]
A collection of estimates on the operator $p_t \ast$ acting on Sobolev functions can be found in~\cite{Lenzmann-Schikorra-commutators}. It is a well-known fact, that $p_t \ast$ is a bounded map from $\dot{H}^{\frac{1}{2}}(\R^n) \to \dot{H}^1(\R^{n+1}_+)$,
\[
 \|\nabla_{\R^{n+1}} \brac{p_t \ast f}\|_{2,\R^{n+1}_+} =  c \|\lapv f\|_{2,\R^{n}}.
\]
The following is the corresponding ``zero order''-estimate, namely the operator $T := \lapv p_t\ast$ is a bounded map from $L^2(\R^n)$ to $L^2(\R^{n+1}_+)$.
\begin{lemma}\label{la:lapvptboundedness}
Denote for $f \in C_c^\infty(\R^n)$,
\[
 Tf(x,t) := p_t \ast \lapv f(x) \equiv \lapv (p_t)\ast f(x) \quad (x,t) \in \R^{n+1}_+.
\]
Then $T$ extends to a linear bounded operator $L^2(\R^n) \to L^{2}(\R^{n+1}_+)$, namely
\[
 \|Tf\|_{2,\R^{n+1}_+} \aleq \|f\|_{2,\R^{n}}.
\]
\end{lemma}
\begin{proof}
Denote by
\[
 \kappa := \lapv p,
\]
and set $\kappa_t = t^{-n} \kappa(\cdot/t)$. Then we can write
\[
 Tf(x,t) = t^{-\frac{1}{2}}\ \kappa_t \ast f(x).
\]
Thus, if we denote the square function with kernel $\kappa$ as
\[
 s_\kappa(f)(x) := \brac{\int_{0}^\infty |\kappa_t \ast f(x)|^2 \frac{dt}{t}}^{\frac{1}{2}},
\]
then
\[
 \|Tf\|_{2,\R^{n+1}_+} = \|s_\kappa(f)\|_{2,\R^n}.
\]
Observe that $\int \kappa = 0$ and that  $\kappa$ decays to zero sufficiently fast at infinity, see Lemma~\ref{la:ptest}, so that the theory of square functions is applicable, see \cite[Chapter I.C, \textsection 8.23, p.46]{Stein-Harmonic-Analysis}.
In particular we obtain
\[
 \|Tf\|_{2,\R^{n+1}_+} \aleq \|f\|_{2,\R^n}.
\]
\end{proof}

\subsection{Commutator estimates for extension operators}\label{ss:commies}
We want to study commutator estimates for extension operators. Recall that we denote by $f^h(x,t)$, $(x,t) \in \R^{n+1}_+$ the harmonic Poisson extension $p_t \ast f(x)$ of a function $f$ defined on $\R^n$. Observe that since the operator $()^h$ takes functions from $\R^n$ to functions on $\R^{n+1}_+$ the classical notion for commutators $[()^h,g](f)$ is meaningless, since $g\, f^h$ is not reasonably defined. Instead we will first consider the commutator
\[
 (fg)^h - f^h\, g^h.
\]
We then have the following estimate.
\begin{proposition}\label{pr:ptfgmptfptg}
\[
 \|(fg)^h - f^h\, g^h \|_{\infty,\R^{n+1}_+} \aleq [f]_{BMO}\ \|g\|_{\infty}.
\]
\end{proposition}
\begin{proof}
W.l.o.g. the constant of the Poisson kernel $p$ is chosen so that 
\begin{equation}\label{eq:ptzeq0}
 \int_{\R^n} p_t(z)\ dz = 1.
\end{equation}
Fix any $(x,t) \in \R^{n+1}_+$, and denote by
\[
 \tilde{g}(z) := g(z)-(g)_{B(x,t)}.
\]
Then
\[
 (fg)^h(x,t) - f^h(x,t)\, g^h(x,t) = (f\tilde{g})^h(x,t) + f^h(x,t)\, \brac{(g)_{B(x,t)}  - g^h(x,t)}.
\]
Consequently, with \eqref{eq:ptzeq0},
\[
\begin{split}
 &\left |(fg)^h(x,t) - f^h(x,t)\, g^h(x,t)\right |\\
 \leq& \brac{\|f\|_{\infty,\R} + \|f^h\|_{\infty,\R^{n+1}_+}}\ \int_{\R^n} p_t(x-z)\ \left |g(z)-(g)_{B(x,t)} \right |\ dz.
\end{split}
 \]
Since $f^h$ is harmonic and $\lim_{|(x,t)| \to \infty} |f^h(x,t)| = 0$, by the maximum principle,
\[
 \|f^h\|_{\infty, \R^{n+1}_+} \aleq \|f\|_{\infty,\R^n}.
\]
Moreover, see \cite[Lemma A.1]{Lenzmann-Schikorra-commutators}, we have
\begin{equation}\label{eq:ptmmv}
 \left |\int_{\R^n} p_t(x-z)\, |g(z)-(g)_{B(x,t)}|\ dz \right |  \aleq [g]_{BMO}.
\end{equation}
The proof of Proposition~\ref{pr:ptfgmptfptg} is finished.
\end{proof}
Now we state our main commutator estimate with regard to the operator $T$ from Lemma~\ref{la:lapvptboundedness}. 
Again, $T$ takes functions on $\R^n$ into a function on $\R^{n+1}_+$, so writing a commutator $[T,g]$ does not make sense. Instead we consider the commutator-like expression 
\[
 T[fg] - g^h\, Tf,
\]
where $g^h$ is the harmonic Poisson extension $p_t \ast g$. 
\begin{theorem}\label{th:extensioncomm}
Let $T$ be the operator from Lemma~\ref{la:lapvptboundedness}, and denote by $g^h := p_t \ast g$ the harmonic Poisson extension of a function $g$ defined on $\R^n$.
Then
for any $\alpha \in (0,n)$, $q_1,q_2 \in [1,\infty]$ so that $\frac{1}{2} = \frac{1}{q_1} + \frac{1}{q_2}$.
\[
 \|T[fg] - g^h\, T[f]\|_{2,\R^{n+1}_+} \aleq \|\laps{\alpha} g\|_{(\frac{n}{\alpha},q_1)}\ \|f\|_{(2,q_2)}.
\]
\end{theorem}
\subsection{Proof of Theorem~\ref{th:extensioncomm}}\label{s:proofth:extensioncomm}
We first need to fix the notation for a several maximal functions. By $\Max $ we denote the (uncentered) Hardy-Littlewood maximal function
\[
 \Max  f(x) = \sup_{B(y,r) \ni x} |B(y,r)|^{-1} \int_{B(y,r)} |f(z)|\ dz.
\]
By an abuse of notation, but for simplicity, we will not distinguish between finitely many repeated maximal functions, i.e. we will identify $\Max \Max  f$ with $\Max  f$.

We also have need the sharp maximal function,
\[
\Max ^\#f(x) := \sup_{B(y,r) \ni x} |B(y,r)|^{-1} \int_{B(y,r)} |f(z)-(f)_{B(y,r)}|\ dz,
\]
as well as the weighted maximal function, defined for $p \in [1,\infty)$ as
\[
 \Max _{\alpha,p}f(x) := \sup_{x \in B} \brac{|B|^{\frac{\alpha p}{n}-1} \int_B |f|^p}^{\frac{1}{p}}.
\]
Clearly,
\[
 \Max _{0,p}f(x) = \brac{\Max  |f|^p(x)}^{\frac{1}{p}}.
\]
Lastly, for a smooth kernel $\kappa \in L^1 \cap L^\infty(\R^n)$, $\int_{\R^n} \kappa = 0$ the square function $s_\kappa f$ is defined as
\begin{equation}\label{eq:squarefct}
 s_\kappa f(x) := \brac{\int_0^\infty |\kappa_t \ast f(x)|^2 \frac{dt}{t}}^{\frac{1}{2}}.
\end{equation}
Recall that we have the notation $\kappa_t(\cdot) := t^{-n} \kappa(\cdot/t)$.

It is well known (see, e.g., \cite{Bojarski-Hajlasz-1993, Hajlasz-1996}) that for any $p \in [1,\infty)$,
\[
 |f(x) - f(y)| \aleq |x-y|\ \brac{\Max |\nabla f|(x) + \Max |\nabla f|(y)}.
\]
The fractional version of this fact holds as well.
\begin{proposition}\label{pr:fxmfymax}
The following holds for any $\alpha \in (0,1)$ and for almost every $x,y \in \R^n$.
\begin{equation}\label{eq:pr:fxmfymax:1}
 |f(x)-f(y)| \aleq |x-y|^{\alpha}\ \brac{\Max (\laps{\alpha} f)(x)+\Max (\laps{\alpha} f)(y)},
\end{equation}
\begin{equation}\label{eq:pr:fxmfymax:2}
 \left |f(x) - (f)_{B(x,r)}| \right | \aleq r^{\alpha}\Max (\laps{\alpha} f)(x),
\end{equation}
and
\begin{equation}\label{eq:pr:fxmfymax:3}
 \left |f(x) - p_t \ast f(x) \right | \aleq t^{\alpha}\ \Max (\laps{\alpha} f)(x).
\end{equation}
\end{proposition}
For the proof of Proposition~\ref{pr:fxmfymax} we use the following Lemma.
\begin{lemma}\label{la:fxmfymaxest}
Let $\alpha \in (0,1)$ and $\Lambda > 0$. Then for almost every $x,y \in \R^n$,
\begin{equation}\label{eq:xmzllambda}
 |x-y|^{-\alpha} \int_{|x-z| < \Lambda |x-y|} |x-z|^{\alpha-n}\, |\laps{\alpha} f(z)|\ dz \aleq \Lambda^{\alpha}\, \Max \laps{\alpha} f(x),
\end{equation}
and
\begin{equation}\label{eq:xmzglambda}
 |x-y|^{1-\alpha} \int_{|x-z| > \Lambda |x-y|} |x-z|^{\alpha-1-n}\, |\laps{\alpha} f(z)|\ dz \aleq \Lambda^{\alpha-1}\, \Max \laps{\alpha} f(x).
\end{equation}
\end{lemma}
\begin{proof}
Regarding \eqref{eq:xmzllambda}, we split the integral and use the definition of the maximal function. Then
\[
\begin{split}
  &|x-y|^{-\alpha} \int_{|x-z| < \Lambda |x-y|} |x-z|^{\alpha-n}\, |\laps{\alpha} f(z)|\ dz \\
  \aeq &\sum_{k = -\infty}^0|x-y|^{-\alpha}\, \int_{|x-z| \aeq  2^k\Lambda |x-y|} |x-z|^{\alpha-n}\, |\laps{\alpha} f(z)|\ dz\\
  \aeq &\sum_{k = -\infty}^0 (2^k \Lambda)^{\alpha}\, \brac{2^k \Lambda |x-y|}^{-n} \int_{|x-z| \aeq  2^k\Lambda |x-y|} |\laps{\alpha} f(z)|\ dz \\
  \aleq &\sum_{k = -\infty}^0 (2^k \Lambda)^{\alpha}\, \Max \laps{\alpha} f(x) \aeq \Lambda^{\alpha}\, \Max \laps{\alpha} f(x). 
\end{split}
 \]
Similarly, regarding \eqref{eq:xmzglambda}, we compute
\[
\begin{split}
  &|x-y|^{\alpha-1} \int_{|x-z| > \Lambda |x-y|} |x-z|^{\alpha-1-n} \, |\laps{\alpha} f(z)|\ dz \\
    \aleq &\sum_{k = 0}^\infty (2^k \Lambda)^{\alpha-1}\ \Max \laps{\alpha} f(x) \aeq \Lambda^{\alpha-1}\, \Max \laps{\alpha} f(x). \\
\end{split}
 \]
Lemma~\ref{la:fxmfymaxest} is proven.
\end{proof}

\begin{proof}[Proof of Proposition~\ref{pr:fxmfymax}]
The second claim \eqref{eq:pr:fxmfymax:2} is a consequence of \eqref{eq:pr:fxmfymax:1}.

\eqref{eq:pr:fxmfymax:3} can be proven by hand arguing similar to Lemma~\ref{la:fxmfymaxest}, but it holds in general for a large class of radial kernels, see \cite[II.2, (16), p.57]{Stein-Harmonic-Analysis}. 

From the remaining claims \underline{we first establish \eqref{eq:pr:fxmfymax:1}}. By the definition of the Riesz potential $\lapms{\alpha} = (-\lap)^{-\frac{\alpha}{2}}$,
\[
 f(x) = \lapms{\alpha} \laps{\alpha} f(x) \equiv c\, \int_{\R^n} |x-z|^{\alpha-n}\, \laps{\alpha} f(z)\ dz.
\]
Consequently,
\[
  \frac{|f(x)-f(y)|}{|x-y|^{\alpha}} \aleq \int_{\R^n} \frac{\left | |x-z|^{\alpha-n} - |y-z|^{\alpha-n} \right |}{|x-y|^{\alpha}}\ |\laps{\alpha} f(z)|\ dz.
\]
We distinguish three regimes in the latter integral. The case where $|x-y|$ is relatively small compared to $|x-z|$ and $|y-z|$, the case where $|y-z|$ is relatively small, and the case where $|x-z|$ is relatively small.

More precisely we decompose $\R^n = A(x,y) \cup B(x,y)$, for
\[
 A(x,y) := \left \{z \in \R^n: \quad |x-y| \leq \frac{1}{2} |x-z| \mbox{ or } |x-y| \leq \frac{1}{2} |y-z| \right \}, 
\]
\[
 B(x,y) := \left \{z \in \R^n: \quad  |x-y| > \frac{1}{2} |x-z| \mbox{ and } |x-y| > \frac{1}{2} |y-z| \right \} ,
\]
In view of \eqref{eq:xmzllambda} we find
\[
\begin{split}
 &\int_{B(x,y)} \frac{\left | |x-z|^{\alpha-n} - |y-z|^{\alpha-n} \right |}{|x-y|^{\alpha}}\ |\laps{\alpha} f(z)|\ dz\\
 \leq& |x-y|^{-\alpha}\, \int_{|x-z| < 2|x-y|} |x-z|^{\alpha-n} \ |\laps{\alpha} f(z)|\ dz\\
 &+  |x-y|^{-\alpha}\, \int_{|y-z| < 2|x-y|} |y-z|^{\alpha-n} \ |\laps{\alpha} f(z)|\ dz\\
 \aleq&\Max  \laps{\alpha} f(x) + \Max  \laps{\alpha} f(y).
 \end{split}
\]
On the other hand, if $|x-y| < \frac{1}{2} |x-z|$ and $|y-x| < \frac{1}{2} |y-z|$, then $|x-z| \aeq |y-z|$ and consequently
\[
 \frac{\left | |x-z|^{\alpha-n} - |y-z|^{\alpha-n} \right |}{|x-y|^{\alpha}} \aleq |x-z|^{\alpha-n-1} |x-y|.
\]
This time we argue with \eqref{eq:xmzglambda} to find
\[
\begin{split}
 &\int_{A(x,y)} \frac{\left | |x-z|^{\alpha-n} - |y-z|^{\alpha-n} \right |}{|x-y|^{\alpha}}\ |\laps{\alpha} f(z)|\ dz\\
 \leq& |x-y|^{1-\alpha}\, \int_{|x-z| > 2|x-y|} |x-z|^{\alpha-1-n} \ |\laps{\alpha} f(z)|\ dz\\
 \aleq&\Max  \laps{\alpha} f(x).
 \end{split}
\]
This establishes \eqref{eq:pr:fxmfymax:1}.
\end{proof}
The main estimate needed for the proof of Theorem~\ref{th:extensioncomm} is contained in the following proposition.
\begin{proposition}\label{pr:extensioncomm:sharp}
Let $T$ be the operator from Lemma~\ref{la:lapvptboundedness}, and denote by $g^h := p_t \ast g$ the harmonic Poisson extension of a function $g$ defined on $\R^n$.
Let 
\[
 v(x):= \brac{\int_{0}^\infty \left |T[fg](x,t) - g^h(x,t)\, T[f](x,t) \right |^2\ dt}^{\frac{1}{2}}.
\]
then for any $\alpha \in (0,1)$, any $x_0 \in \R^n$, and for any $p, q \in (1,\infty)$
\[\begin{split}
  \Max ^\# v(x_0) \aleq&  \Max _{0,p}\laps{\alpha}g(x_0)\ \Max _{\alpha,p'} (s_\kappa f)(x_0) 
+\Max _{0,p}\laps{\alpha}g(x_0)\ \Max _{0,p'} (s_{\tilde{\kappa}} \lapms{\alpha}f)(x_0)\\
 &+ \Max _{0,qp}\Max  \laps{\alpha}g(x_0)\, \Max _{\alpha,qp'} f(x_0)\\
\end{split}
  \]
Here, $\kappa= \lapv p$ and $\tilde{\kappa} := \laps{\alpha} \kappa$ where $p$ is the Poisson kernel. 
\end{proposition}
\begin{proof}
The proof follows some ideas from commutator estimates on (interior) square functions, also called Lusin functions, see \cite{Torchinsky-Wang-1990,Chen-Ding-2009,GuliyevOmarovaSawano-2015}. This is responsible for the $L^2$-norm to appear on the left-hand side of the estimate of Theorem~\ref{th:extensioncomm}, cf. Remark~\ref{r:commiel2}.

Fix $x_0 \in \R^n$. For some $r > 0$ and $x \in \R^n$ so that $x_0 \in B(x,r)$ we want to find an estimate for
\[
 |B(x,r)|^{-1} \int_{B(x,r)} |v(y) - (v)_{B(x,r)}|\ dy.
\]
Set 
\[
 \tilde{g}(z) := g(z) - (g)_{B(x_0,r)}.
\]
For arbitrary $t \in (0,\infty)$, we split
\begin{equation}\label{eq:Tfgcomsplit}
 T[fg](y,t) - g^h(y,t)\, T[f](y,t) = -I(y,t) + II(y,t) + III(y,t) + IV(y,t),
\end{equation}
where
\[
\begin{split}
 I(y,t) :=& \tilde{g}(y)\, T[f](y,t),\\
 II(y,t) :=& \brac{g(y) - g^h(y,t)}\, T[f](y,t),\\
 III(y,t) :=& T[\chi_{B(x_0,10r)}f \tilde{g}](y,t),\\
 IV(y,t) :=& T[\chi_{\R^n \backslash B(x_0,10r)}f \tilde{g}](y,t).\\
  \end{split}
\]
We begin with the \underline{estimate of $I(y,t)$}.
We have
\[
 \brac{\int_{0}^\infty |I(y,t)|^2 dt}^{\frac{1}{2}} \aleq |\tilde{g}(y)|\, \brac{\int_{0}^\infty |T[f](y,t)|^2 dt}^{\frac{1}{2}}
\]
Recall that for $\kappa := \lapv p$ we can write the $t$-integral as square function \eqref{eq:squarefct},
\[
 \brac{\int_{0}^\infty |Tf(y)|^2 dt}^{\frac{1}{2}} = s_\kappa f(y).
\]
On the other hand by Proposition~\ref{pr:fxmfymax},
\[
 |\tilde{g}(y)| = |g(y) - (g)_{B(x_0,r)}| \aleq r^{\alpha}\brac{\Max  (\laps{\alpha}g)(y)+\Max  (\laps{\alpha}g)(x_0)}.
\]
Consequently, using H\"older inequality we find that for any $p > 1$,
\begin{equation}\label{eq:Iyt}
 \mvint_{B(x_0,r)} \brac{\int_{0}^\infty |I(y,t)|^2 dt}^{\frac{1}{2}}\ dy\aleq \Max _{0,p}\laps{\alpha}g(x_0)\ \Max _{\alpha,p'} (s_\kappa f)(x_0).
\end{equation}
\underline{As for $II(y,t)$} observe that by Proposition~\ref{pr:fxmfymax},
\[
 |g(y)-g^h(y,t)| \leq t^{\alpha}\Max (\laps{\alpha}g)(y).
\]
Consequently,
\[
\brac{\int_{0}^\infty |II(y,t)|^2 dt}^{\frac{1}{2}} \aleq \Max (\laps{\alpha}g)(y)\ \brac{\int_{0}^\infty |t^{\alpha}\kappa_t\ast f(y)|^2 \frac{dt}{t}}^{\frac{1}{2}}.
\]
Now with the Riesz potential $\lapms{\alpha} = (-\lap)^{-\frac{\alpha}{2}}$ we can write
\[
 t^{\alpha}\kappa_t \ast f = (\laps{\alpha}\kappa)_t \ast \lapms{\alpha} f.
\]
Therefore, for $\tilde{\kappa} := \laps{\alpha} \kappa$, we have found that for any $p \in (1,\infty)$
\begin{equation}\label{eq:IIyt}
 \mvint_{B(x_0,r)} \brac{\int_{0}^\infty |II(y,t)|^2 dt}^{\frac{1}{2}}\ dy\aleq \Max _{0,p}\laps{\alpha}g(x_0)\ \Max _{0,p'} (s_{\tilde{\kappa}} \lapms{\alpha}f)(x_0).
\end{equation}
Next we estimate \underline{estimate $III(y,t)$}, and have for any $q \geq 1$ by H\"older inequality
\[
 \mvint_{B(x_0,r)} \brac{\int_{0}^\infty |II(y,t)|^2 dt}^{\frac{1}{2}}\ dy \aleq r^{-\frac{n}{q}}\, \left \|s_\kappa \brac{f\chi_{B(x_0,10r)} \tilde{g}} \right \|_{q,\R^n}.
\]
Since the square function $s_\kappa$ is a bounded operator on $L^q(\R^n)$ whenever $q \in (1,\infty)$, we find for such $q$,
\[
 \mvint_{B(x_0,r)} \brac{\int_{0}^\infty |II(y,t)|^2 dt}^{\frac{1}{2}}\ dy \aleq r^{-\frac{n}{q}}\ \|f\chi_{B(x_0,10r)} \tilde{g}\|_{q,\R^n} = r^{-\frac{n}{q}}\ \|f\, \tilde{g}\|_{q,B(x_0,10r)}.
\]
In view of the definition of $\tilde{g}$ and Proposition~\ref{pr:fxmfymax}
\[
 \|f\, \tilde{g}\|_{q,B(x_0,10r)} \aleq r^{\alpha}\brac{\|f\, \Max \laps{\alpha} g\|_{q,B(x_0,10r)} + \|f\|_{q,B(x_0,10r)}\, \Max \laps{\alpha} g(x_0)}.
\]
In particular, for any $p > 1$,
\[
 r^{-\frac{n}{q}}\ \|f\, \tilde{g}\|_{q,B(x_0,10r)} \aleq \Max _{0,qp}\Max  \laps{\alpha}g(x_0)\, \Max _{\alpha,qp'} f(x_0)  + \Max \laps{\alpha} g(x_0)\, \Max _{\alpha,q}f(x_0).
\]
That is,
\begin{equation}\label{eq:IIIyt}
 \mvint_{B(x_0,r)} \brac{\int_{0}^\infty |III(y,t)|^2 dt}^{\frac{1}{2}}\ dy\aleq  \Max _{0,qp}\Max  \laps{\alpha}g(x_0)\, \Max _{\alpha,qp'} f(x_0)\  + \Max \laps{\alpha} g(x_0)\, \Max _{\alpha,q}f(x_0).
\end{equation}


\underline{It remains to treat $IV(y,t)$}.
Recall the Minkowski-inequality
\[
 \brac{\int_{0}^\infty \brac{\int_{\R^n} f(x,t)\ dx}^{2}\ dt}^{\frac{1}{2}} \aleq \int_{\R^n} \brac{\int_{0}^\infty f(x,t)^2\ dt}^{\frac{1}{2}}\ dx.
\]
Thus, for $y_1, y_2 \in B(x_0,r)$, with the definition of $Tf = \lapv p_t \ast f$,
\[
\begin{split}
 &\brac{\int_{0}^\infty |IV(y_1,t)-IV(y_2,t)|^2\, dt}^{\frac{1}{2}} \\
 \aleq  &
 \int_{\R^n \backslash B(x_0,10r)} \brac{\int_{0}^\infty  \brac{ \lapv p_t(y_1-z) -\lapv p_t(y_2-z)}^2\, dt}^{\frac{1}{2}}\ |\tilde{g}(z)|\ |f(z)|\ dz.
\end{split}
 \]
Observe that for $z \in \R^n \backslash B(x_0,10r)$ and $y_1,y_2 \in B(x_0,r)$ we have $|y_1-z| \aeq |y_2-z|$. Thus, in view of Lemma~\ref{la:ptest}, 
\[
 |\lapv p_t(y_1-z) -\lapv p_t(y_2-z)| \aleq |y_1-y_2|\, \brac{t^2 + |y_1-z|^2}^{-\frac{n+\frac{3}{2}}{2}}.
\]
Moreover,
\[
  \brac{\int_{0}^\infty \brac{t^2 + |y_1-z|^2}^{-\brac{n+\frac{3}{2}}}\ dt}^{\frac{1}{2}} \aeq |y_1-z|^{-n-1}.
\]
Consequently we have shown that for any $y_1, y_2 \in B(x_0,r)$
\begin{equation}\label{eq:IVest1}
 \brac{\int_{0}^\infty |IV(y_1,t)-IV(y_2,t)|^2\, dt}^{\frac{1}{2}} \aleq r\, \int_{\R^n \backslash B(x_0,10r)}  |y_1-z|^{-n-1}\ |\tilde{g}(z)|\ |f(z)|\, dz.
\end{equation}
We split this integral,
\[
 \int_{\R^n \backslash B(x_0,10r)}  |y_1-z|^{-n-1}\ |\tilde{g}(z)|\ |f(z)|\, dz  \aleq \sum_{\ell =2}^\infty (2^\ell r)^{-n-1} \| f\, \tilde{g} \|_{1,B(x_0,2^\ell r)}.
 \]
With the definition of $\tilde{g}$,
\[
 \|f \tilde{g}\|_{1,B(x_0,2^\ell r)} \aleq \|f\, \brac{g - (g)_{B(x_0,2^\ell r)}}\|_{1,B(x_0,2^\ell r)}  + \sum_{k=1}^\ell \|f\|_{1,B(x_0,2^\ell r)}\ \left |(g)_{B(x_0,2^kr)} - (g)_{B(x_0,2^{k-1}r)}\right |  
\]
On the one hand, for any $p > 1$ we have by H\"older inequality and in view of Proposition~\ref{pr:fxmfymax},
\[
 \|f\, \brac{g - (g)_{B(x_0,2^\ell r)}}\|_{1,B(x_0,2^\ell r)} \aleq \brac{2^{\ell}r}^{n}\ \Max _{0,p}\laps{\alpha}g(x_0)\, \Max _{\alpha,p'} f(x_0).
\]
On the other hand, observe that in view of Proposition~\ref{pr:fxmfymax}
\[
 \left |(g)_{B(x_0,2^kr)} - (g)_{B(x_0,2^{k-1}r)}\right |  \aleq (2^k r)^{\alpha}\Max  \laps{\alpha}g(x_0)
\]
and
\[
 \|f\|_{1,B(x_0,2^\ell r)} \aleq (2^\ell r)^{n-\alpha} \Max _{\alpha,1} f(x_0)
\]
and thus, since $\sum_{k=1}^\ell 2^{-\alpha(\ell-k)} \aleq 1$,
\[
 \sum_{k=1}^\ell \|f\|_{1,B(x_0,2^\ell r)}\ \left |(g)_{B(x_0,2^kr)} - (g)_{B(x_0,2^{k-1}r)}\right |   \aleq (2^\ell r)^{n} \Max  \laps{\alpha}g(x_0)\, \Max _{\alpha,1} f(x_0).
\]
Plugging these estimates into \eqref{eq:IVest1}, using that $\sum_{\ell =2}^\infty 2^{-\ell}  \aleq 1$, we have shown that for any $p > 1$,
\begin{equation}\label{eq:IVyt}
\begin{split}
 &\mvint_{B(x_0,r)}\mvint_{B(x_0,r)}\brac{\int_{0}^\infty |IV(y_1,t)-IV(y_2,t)|^2\, dt}^{\frac{1}{2}}\, dy_1\, dy_2\\
 \aleq& \Max _{0,p}\laps{\alpha}g(x_0)\, \Max _{\alpha,p'} f(x_0) +   \Max  \laps{\alpha}g(x_0)\, \Max _{\alpha,1} f(x_0).
 \end{split}
\end{equation}
\underline{We can now conclude} as follows. From the definition of $v$ in the statement of the proposition and the decomposition \eqref{eq:Tfgcomsplit} we find
\[
 \begin{split}
 &\mvint_{B(x_0,r)}\mvint_{B(x_0,r)}|v(y_1)-v(y_2)|dy_1\, dy_2\\
 \aleq& \mvint_{B(x_0,r)}\brac{\brac{\int_{0}^\infty |I(y,t)|^2\ dt}^{\frac{1}{2}} + \brac{\int_{0}^\infty |II(y,t)|^2\ dt}^{\frac{1}{2}} + \brac{\int_{0}^\infty |III(y,t)|^2\ dt}^{\frac{1}{2}}}\, dy\\
 &+ \mvint_{B(x_0,r)}\mvint_{B(x_0,r)}\brac{\int_{0}^\infty |IV(y_1,t)-IV(y_2,t)|^2\ dt}^{\frac{1}{2}} dy_1\, dy_2.
 \end{split}
\]
Thus, with the estimates \eqref{eq:Iyt}, \eqref{eq:IIyt}, \eqref{eq:IIIyt}, and \eqref{eq:IVyt} we obtain that for any $p, q > 1$,
\[
 \begin{split}
 &\mvint_{B(x_0,r)}\mvint_{B(x_0,r)}|v(y_1)-v(y_2)|dy_1\, dy_2\\
 \aleq& \Max _{0,p}\laps{\alpha}g(x_0)\ \Max _{\alpha,p'} (s_\kappa f)(x_0) 
+\Max _{0,p}\laps{\alpha}g(x_0)\ \Max _{0,p'} (s_{\tilde{\kappa}} \lapms{\alpha}f)(x_0)\\
 &+ \Max _{0,qp}\Max  \laps{\alpha}g(x_0)\, \Max _{\alpha,qp'} f(x_0) + \Max \laps{\alpha} g(x_0)\, \Max _{\alpha,q}f(x_0)\\
 &+  \Max  \laps{\alpha}g(x_0)\, \Max _{\alpha,1} f(x_0).
 \end{split}
\]
Observe that 
\[
 \Max  \laps{\alpha}g(x_0)\, \Max _{\alpha,1} f(x_0) + \Max \laps{\alpha} g(x_0)\, \Max _{\alpha,q}f(x_0) \aleq \Max _{0,qp}\Max  \laps{\alpha}g(x_0)\, \Max _{\alpha,qp'} f(x_0).
\]
Taking the supremum in $r>0$ we conclude.
\end{proof}
Finally, we need the following Lorentz-space estimates.
\begin{lemma}\label{la:maximalest}
Let $r \in (1,\infty)$. For $p_1,p_2 \in (r,\infty)$, $\alpha \in (0,n)$ so that $\frac{\alpha}{n} = \frac{1}{p_2}-\frac{1}{p_1}$. Then, for any $q \in [1,\infty]$,
\begin{equation}\label{eq:maximalest:1}
 \|\Max _{\alpha,r} f\|_{(p_1,q),\R^n} \aleq \|f\|_{(p_2,q),\R^n}
\end{equation}
In particular, for any $p_1 \in (r,\infty)$ and any $q \in[1,\infty]$,
\begin{equation}\label{eq:maximalest:2}
 \|\Max _{0,r} f\|_{(p_1,q),\R^n} \aleq \|f\|_{(p_1,q),\R^n}
\end{equation}
\end{lemma}
\begin{proof}
The second claim \eqref{eq:maximalest:2} is just the first claim \eqref{eq:maximalest:1} for $\alpha = 0$.

From \cite[Lemma 2]{Chanillo-1982} we obtain that
\[
 \|\Max _{\alpha,r} f\|_{p_1,\R^n} \aleq \|f\|_{p_2,\R^n}
\]
Because $\Max _{\alpha,r}$ is quasilinear, \eqref{eq:maximalest:1} follows for all $q \in [1,\infty]$ by interpolation, see \cite[Theorem 1.4.19.]{GrafakosC}.
\end{proof}

\begin{proof}[Proof of Theorem~\ref{th:extensioncomm}]
We may assume without loss of generality that $\alpha \in (0,\frac{1}{2})$, and thus in particular $\frac{n}{\alpha} > 2$. Indeed, if that is not the case we prove the claim for $\beta < \frac{1}{2}$ and observe that for $\beta < \frac{1}{2} \leq \alpha < n$, by Sobolev embedding we have
\[
 \|\laps{\beta} g\|_{(\frac{n}{\beta},q_1)} \aleq \|\laps{\alpha} g\|_{(\frac{n}{\alpha},q_1)}.
\]
So let $\alpha \in (0,\frac{1}{2})$, fix $q_1, q_2 \in [1,\infty]$ so that $\frac{1}{q_1} + \frac{1}{q_2} = \frac{1}{2}$. Also we pick in Proposition~\ref{pr:extensioncomm:sharp} $p \in (2,\frac{n}{\alpha})$ and $q>1$ so that $pq < \frac{n}{\alpha}$ and $p'q < \frac{2n}{n-2\alpha}$.

Now we estimate the terms on the right-hand side of the estimate of Proposition~\ref{pr:extensioncomm:sharp}.

Both, $\kappa$ and $\tilde{\kappa}$, satisfy the kernel condition for the square function estimates, and thus from \cite[Chapter I.C, \textsection 8.23, p.46]{Stein-Harmonic-Analysis}, for any $r \in (1,\infty)$,
\begin{equation}\label{eq:skappaest}
 \|s_\kappa h\|_{(r,q_2),\R^n}+\|s_{\tilde{\kappa}} h\|_{(r,q_2),\R^n} \aleq \|h\|_{(r,q_2),\R^n}.
\end{equation}
Moreover, recall that from the Sobolev inequality for Lorentz spaces we have
 \begin{equation}\label{eq:lorentzsobolev}
 \|\lapms{\alpha}f \|_{(\frac{2n}{n-2\alpha},q_2),\R^n} \aleq \|f\|_{(2,q_2),\R^n}.
\end{equation}
By H\"older inequality, Lemma~\ref{la:maximalest}, and \eqref{eq:skappaest} we find
\[
\begin{split}
 &\|\Max _{0,p}\laps{\alpha}g\ \Max _{\alpha,p'} (s_\kappa f) \|_{2,\R^n}\\
 \aleq& \|\Max _{0,p}\laps{\alpha}g\|_{(\frac{n}{\alpha},q_1),\R^n}\, \|\Max _{\alpha,p'} (s_\kappa f) \|_{(\frac{2n}{n-2\alpha},q_2),\R^n}\\
 \aleq& \|\laps{\alpha}g\|_{(\frac{n}{\alpha},q_1),\R^n}\, \|f\|_{(2,q_2),\R^n}.
\end{split}
 \]
By the same argument,
\[
 \|\Max _{0,qp}\Max  \laps{\alpha}g\, \Max _{\alpha,qp'} f\|_{2,\R^n} \aleq \|\laps{\alpha}g\|_{(\frac{n}{\alpha},q_1),\R^n}\, \|f\|_{(2,q_2),\R^n}.
\]
Using additionally \eqref{eq:lorentzsobolev} we have
\[
\begin{split}
 \|\Max _{0,p}\laps{\alpha}g\ \Max _{0,p'} (s_{\tilde{\kappa}} \lapms{\alpha}f)\|_{2,\R^n} \aleq& \|\laps{\alpha}g\|_{(\frac{n}{\alpha},q_1),\R^n}\, \|\lapms{\alpha}f \|_{(\frac{2n}{n-2\alpha},q_2),\R^n}\\
 \aleq & \|\laps{\alpha}g\|_{(\frac{n}{\alpha},q_1),\R^n}\, \|f \|_{(2,q_2),\R^n}.
\end{split}
 \]
Thus, for $v$ as in Proposition~\ref{pr:extensioncomm:sharp} we have shown
\begin{equation}\label{eq:mhvest}
 \|\Max ^\#v\|_{2,\R^n}. \aleq \|\laps{\alpha}g\|_{(\frac{n}{\alpha},q_1),\R^n}\, \|f \|_{(2,q_2),\R^n}.
\end{equation}
On the other hand, 
\begin{equation}\label{eq:TfvL2}
 \|T[fg] - g^h\, T[f]\|_{2,\R^{n+1}_+} = \|v\|_{2,\R^n}. 
\end{equation}
Moreover, by \cite[IV, 2.2, Theorem 2, p.148]{Stein-Harmonic-Analysis}, we have
\begin{equation}\label{eq:vmhv}
 \|v\|_{2,\R^n} \aleq \|\Max ^\#v\|_{2,\R^n}.
\end{equation}
Together, \eqref{eq:mhvest}, \eqref{eq:TfvL2} and \eqref{eq:vmhv} imply
\[
 \|T[fg] - g^h\, T[f]\|_{2,\R^{n+1}_+} \aleq \|\laps{\alpha}g\|_{(\frac{n}{\alpha},q_1),\R^n}\, \|f \|_{(2,q_2),\R^n}.
\]
Therefore, Theorem~\ref{th:extensioncomm} is proven.
\end{proof}
\begin{remark}\label{r:commiel2}
Observe that \eqref{eq:TfvL2} is true only for $L^2$, and fails for $L^p$ with $p \neq 2$. Consequently, it is at least dubious that there holds an $L^p$-version of Theorem~\ref{th:extensioncomm} of the form
\[
 \|T[fg] - g^h\, T[f]\|_{p,\R^{n+1}_+} \aleq \|\laps{\alpha} g\|_{(\frac{n}{\alpha},q_1)}\ \|f\|_{(p,q_2)}
\]
whenever $p \neq 2$. This is also related to the fact that $\|\nabla f\|_{2,\R^{n+1}_+} \aeq \|\lapv f\|_{2,\R^n}$, but $\|\nabla f\|_{p,\R^{n+1}_+} \not \aeq \|\lapv f\|_{p,\R^{n+1}_+}$ for $p \neq 2$. On the other hand, 
our arguments readily imply the following $L^p(\R^n)$-type version
\[
 \brac{\int_{\R^n} \brac{\int_0^\infty \brac{T[fg](x,t) - g^h(x,t)\, T[f](x,t)}^2 dt}^{\frac{p}{2}}\ dx}^{\frac{1}{p}} \aleq \|\laps{\alpha} g\|_{(\frac{n}{\alpha},q_1)}\ \|f\|_{(p,q_2)}.
\]
\end{remark}
\appendix
\section{Hardy space, div-curl quantities, and estimates on the halfspace}\label{s:hardyhalf}
Fix a kernel $\kappa \in C_c^\infty(B(0,1))$, $\int \kappa =1$ and denote by $\kappa_t(z):= t^{-n}\kappa(z/t)$.

A function $f \in L^1(\R^n)$ belongs to the Hardy space $\hardy(\R^n)$ if and only if
\[
 \sup_{t > 0}|\kappa_t \ast f| \in L^1(\R^n), 
\]
and the Hardy-space norm $\|\cdot\|_{\hardy}$ is given by
\[
 \|f\|_{\hardy(\R^n)} := \|f\|_{1,\R^n} + \| \sup_{t > 0}|\kappa_t \ast f|\|_{1,\R^n}.
\]
Different choices of $\kappa$ give equivalent norms of Hardy spaces. The interested reader is referred to the excellent survey on Hardy spaces and their implications for elliptic equations by Semmes, \cite{Semmes-1994}. 

The Hardy space $\hardy$ is important for the regularity theory for critical geometric equations, because of their duality-relation with $BMO$: the following Hardy-BMO-inequality holds
\begin{equation}\label{eq:hardybmodual}
 \int_{\R^n} f\ g \aleq \|f\|_{\hardy(\R^n)}\ [g]_{BMO}.
\end{equation}
Here the space BMO is defined by its norm
\[
 [g]_{BMO} := \sup_{B(x,r) \subset \R^n} r^{-n} \int_{B(x,r)} |g-(g)_{B(x,r)}|.
\]
Observe that by Sobolev-Poincar\'e embedding,
\[
 [g]_{BMO} \aleq \|\nabla g\|_{(n,\infty),(\R^n)},
\]
and more generally for any $\alpha > 0$,
\[
 [g]_{BMO} \aleq \|\laps{\alpha} g\|_{(\frac{n}{\alpha},\infty),(\R^n)},
\]

We state the celebrated result by Coifman, Lions, Meyer, Semmes \cite{CLMS-1993}, this is also very related to the Coifman-Rochberg-Weiss theorem, Theorem~\ref{th:CRW} . For proofs via harmonic extensions we refer to \cite{Lenzmann-Schikorra-commutators}.
\begin{theorem}[Coifman-Lions-Meyer-Semmes]\label{th:clms}
Let $F \in L^2(\R^n,\R^n)$ and $\nabla G \in L^2(\R^n)$. If $\div F = 0$ in $\R^n$, then for any $\Phi \in BMO$ 
\begin{equation}\label{eq:clmsglobal}
 \int_{\R^n} F \cdot \nabla G\ \Phi \aleq \|F\|_{2,\R^n}\ \|\nabla G\|_{2,\R^n}\ [\Phi]_{BMO}
\end{equation}
Moreover, we have the following localization. For any $r > 0$ and any $x_0 \in \R^n$
\begin{equation}\label{eq:clmslocal}
\begin{split}
 \int_{\R^n} F \cdot \nabla G\ \Phi \aleq& \|F\|_{2,B(x_0,r)}\ \|\nabla G\|_{2,B(x_0,r)}\ \|\nabla \Phi\|_{(2,\infty),B(x_0,r)}\\
 &+ \sum_{k=1}^\infty k\, \|F\|_{2,B(x_0,2^{k+5} r) \backslash B(x_0,2^{k-5} r)}\ \|\nabla G\|_{2,B(x_0,2^{k+5} r) \backslash B(x_0,2^{k-5} r)}\ \|\nabla \Phi\|_{(2,\infty),B(x_0,2^{k+5} r)}.
 \end{split}
\end{equation}
\end{theorem}
In \cite{CLMS-1993} it was shown that
\[
 F \cdot \nabla G \in \hardy(\R^n),
\]
which, in view of \eqref{eq:hardybmodual}, readily implies \eqref{eq:clmsglobal}. For \eqref{eq:clmslocal} we need the following adaption of their argument.
\begin{lemma}\label{la:localclmsest}
Let $\eta \in C_c^\infty(B(x_0,6R))$, and $\xi \in C_c^\infty(B(x_0,6R) \backslash B(x_0,5R))$ with \[
\int \eta \aeq \int \xi \aeq R^n,\mbox{ and } \|\nabla \eta\|_{\infty} + \|\nabla \xi\|_{\infty} \aleq R^{-1}.
\] 
Then for $F$, $G$ as in Theorem~\ref{th:clms},
\[
 \left \| \sup_{t \in (0,R)} |\kappa_t \ast \brac{\eta\, F \cdot \nabla G}|\right \|_{1,\R^n} \aleq \|F\|_{2,B(x_0,8R)}\ \|\nabla G\|_{2,B(x_0,8R)}
\]
and
\[
 \left \| \sup_{t \in (0,R)} |\kappa_t \ast \brac{\xi\, F \cdot \nabla G}|\right \|_{1,\R^n} \aleq \|F\|_{2,B(x_0,8R)\backslash B(x_0,3R)}\ \|\nabla G\|_{2,B(x_0,8R)\backslash B(x_0,3R)}.
\]
\end{lemma}
\begin{proof}
We only provide a proof for the $\xi$-case, the arguments for $\eta$ are the same.

Let $y_0 \in \R^n$ and $t \in (0,R)$. 
Integrating by parts and with $\div F = 0$,
\[
\begin{split}
 \kappa_t \ast \brac{\xi\, F \cdot G}(y_0)=-\int_{\R^n} \nabla \brac{\kappa_t(y_0-z)\ \xi(z)}\ F(z)\cdot (G(z)-(G)_{B(y_0,t)})\ dz,
 \end{split}
 \]
and thus for any $t \in (0,R)$,
\[ 
\begin{split}
\kappa_t \ast \brac{\xi\, F \cdot G}(y_0) \aleq&t^{-n-1} \int_{B(y_0,t)}|F(z)|\ |G(z)-(G)_{B(y_0,t)}|.
\end{split}
 \]
Pick any $p,q \in (1,2)$ so that $\frac{nq}{n-q} + \frac{1}{p} = 1$. Then H\"older and Poincar\'e-Sobolev inequality imply
\[
 |\kappa_t \ast \brac{\xi\, F \cdot G}(y_0)| \aleq \brac{t^{-n}\int_{B(y_0,t)}|F|^p}^{\frac{1}{p}}\ \brac{t^{-n}\int_{B(y_0,t)}|\nabla G|^q}^{\frac{1}{q}}
\]
\[
 \Max _R f := \sup_{t \in (0,R)} t^{-n}\int_{B(y_0,t)}|f|^p,
\]
we have found
\[
 \sup_{t \in (0,R)} |\kappa_t \ast \brac{\xi\, F \cdot G}(y_0)| \aleq \brac{\Max _R |F|^p(x_0)}^{\frac{1}{p}}\ \brac{\Max _R |\nabla G|^q(x_0)}^{\frac{1}{q}}
\]
On the other hand observe that $\kappa_t \ast \brac{\xi\, F \cdot G}(y_0) = 0$ for any $t \in (0,R)$ and any $y_0 \not \in B(x_0,7R) \backslash B(x_0,4R)$.

Consequently, with the maximal theorem, we conclude
\[
 \left \|\sup_{t \in (0,R)} |\kappa_t \ast \brac{\xi\, F \cdot G}(y_0)|\right \|_{1,\R^n} \aleq \|F\|_{2,B(x_0,8R)\backslash B(x_0,3R)}\ \|\nabla G\|_{2,B(x_0,8R)\backslash B(x_0,3R)}
\]
\end{proof}

\begin{lemma}\label{la:localtoglobalhardy}
For some $R > 0$ assume that $\eta$ and $\xi$ are given as in Lemma~\ref{la:localclmsest}.

If $g = \eta f$ or $g = \xi f \in L^1(\R^n)$ satisfies
\[
 \left \|\sup_{t \in (0,R)} |\kappa_t \ast g| \right \|_{1,\R^n} \leq \Lambda,
\]
then there is $\lambda > 0$ so that 
\begin{equation}\label{eq:lambdaest}
 |\lambda| \aleq R^{-n}\|g\|_{1}
\end{equation}
and
\[
 h := \eta (f-\lambda) \quad \mbox{or} \quad h := \xi (f-\lambda)
\]
belongs to the Hardy space with
\[
 \|h\|_{\hardy} \aleq \|g\|_{L^1} + \Lambda.
\]
\end{lemma}
\begin{proof}
This follows from the arguments in \cite[Proposition 1.92]{Semmes-1994}. We repeat the proof for the sake of completeness, again restricting our attention to the case of $\xi$, only.
Set
\[
 \lambda := \brac{\int \xi}^{-1}\ \int \xi f.
\]
Clearly, \eqref{eq:lambdaest} is satisfied. Moreover,
\begin{equation}\label{eq:clmshgest}
 \|h\|_{1,\R^n} \aleq \|g\|_{1,\R^n}.
\end{equation}

Firstly, we have
\[
 |\kappa_t \ast h |(y_0) \leq |\kappa_t \ast g|(y_0) + \|g\|_{1,\R^n}\ |\kappa_t| \ast |\xi|(y_0),
\]
and consequently,
\[
 \|\sup_{t \in (0,R)} |\kappa_t \ast h | \|_{1,\R^n} \aleq \Lambda + \|g\|_{1,\R^n}.
\]
On the other hand, since by the choice of $\lambda$ we have $\int h = 0$,
\begin{equation}\label{eq:kappatasth}
 |\kappa_t \ast h(y_0)| \leq \int_{\R^n} |\kappa_t(y_0-z) - \kappa_t(y_0-x_0)|\ h(z)\ dz \aleq t^{-n-1}\, R\, \|h\|_{1,\R^n}.
\end{equation}
Consequently,
\[
 \|\sup_{t \in (R,\infty)} |\kappa_t \ast h| \|_{1,B(x_0,20R)} \aleq \|h\|_{1,\R^n}.
\]
Moreover, since $\kappa_t \ast h(y_0) = 0$ for $\dist(y_0,\supp \xi) > t$ from \eqref{eq:kappatasth} we find for any $y_0 \in \R^n \backslash B(x_0,20R)$
\[
  \sup_{t \in (R,\infty)} |\kappa_t \ast h(y_0)| = \sup_{t \in (\dist(y_0,\supp \xi),\infty)} |\kappa_t \ast h(y_0)| \aleq |y_0-x_0|^{-n-1}\, R\, \|h\|_{1,\R^n}.
\]
Integrating this give
\[
 \|\sup_{t \in (R,\infty)} |\kappa_t \ast h| \|_{1,\R^n \backslash B(x_0,20R)} \aleq \|h\|_{1,\R^n}.
\]
Altogether, we find
\[
 \|\sup_{t \in (0,\infty)} |\kappa_t \ast h| \|_{1,\R^n} \aleq \|h\|_{1,\R^n} + \|g\|_{1,\R^n}+\Lambda,
\]
which in view of \eqref{eq:clmshgest} implies the claim.
\end{proof}

\begin{proof}[Proof of \eqref{eq:clmslocal}]
Let $\eta \in C_c^\infty(B(0,1))$ a typical bump function constantly one in $B(0,\frac{1}{2})$ and set
\[
 \eta_{k} := \eta((x_0-x)/{2^k r}),\ \xi_k := \eta_k - \eta_{k-1}.
\]
Also we denote
\[
 (f)_k := (f)_{B(x_0,2^{k} r)}.
\]
By $\div F = 0$, \[\int F\cdot \nabla G\, (\Phi)_1 = 0.\] Thus for any $(\lambda_k)_{k=0}^\infty \subset \R$, we may write
\[
\begin{split}
 \int_{\R^n} F \cdot \nabla G\ \Phi =& \int_{\R^n} \eta_0 (F \cdot \nabla G-\lambda_0)\ \brac{\Phi -(\Phi)_1} + \lambda_0 \int_{\R^n} \eta_0 \brac{\Phi -(\Phi)_1}\\
  &+\sum_{k=1}^\infty \int_{\R^n} \xi_k (F \cdot \nabla G-\lambda_k)\ (\Phi-(\Phi)_1)+ \sum_{k=1}^\infty \lambda_k \int_{\R^n} \xi_k (\Phi-(\Phi)_1).
\end{split}
 \]
By Lemma~\ref{la:localclmsest} and Lemma~\ref{la:localtoglobalhardy} we can choose $\lambda_k \in \R$, $k=0,1\ldots$, so that
\[
 \|\eta_0 (F \cdot \nabla G-\lambda_0)\|_{\hardy} + r^n |\lambda_0| \aleq \|F\|_{2,B(x_0,8r)}\ \|\nabla G\|_{2,B(x_0,8r)},
\]
and
\[
 \|\xi_k (F \cdot \nabla G-\lambda_k)\|_{\hardy} + (2^k r)^n |\lambda_k|  \aleq \|F\|_{2,B(x_0,2^{k+5}r) \backslash B(x_0,2^{k-5}r)}\ \|\nabla G\|_{2,B(x_0,2^{k+5}r) \backslash B(x_0,2^{k-5}r)}.
\]
Using the Hardy-BMO inequality \eqref{eq:hardybmodual} and the fact that $\eta_0 \equiv \eta_0 \eta_1$ as well as $\xi_k \equiv \xi_k \eta_{k+1}$,
\[
\begin{split}
 \left |\int_{\R^n} F \cdot \nabla G\ \Phi \right | 
 \aleq &\|F\|_{2,B(x_0,8r)}\ \|\nabla G\|_{2,B(x_0,8r)}\ [\eta_1(\Phi -(\Phi)_1)]_{BMO} \\
 &+ \sum_{k=1}^\infty \|F\|_{2,B(x_0,2^{k+5}r) \backslash B(x_0,2^{k-5}r)}\ \|\nabla G\|_{2,B(x_0,2^{k+5}r) \backslash B(x_0,2^{k-5}r)}. 
 [\eta_{k+1}(\Phi-(\Phi)_1)]_{BMO}\\
\end{split}
\]
By Poincar\'e-Sobolev embedding 
\[
 [\eta_1(\Phi -(\Phi)_1)]_{BMO} \aleq \|\nabla \Phi\|_{(2,\infty),B(x_0,4r)},
\]
Moreover,
\[
\begin{split}
 [\eta_{k+1}(\Phi-(\Phi)_1)]_{BMO} \aleq& [\eta_{k+1}(\Phi-(\Phi)_{k+1})]_{BMO} + \sum_{\ell = 2}^k | (\Phi)_{\ell}-(\Phi)_{\ell-1}|\\
  \aleq& k\  \|\nabla \Phi\|_{(2,\infty),B(x_0,2^{k+1}r)}.
\end{split}
 \]
The claim follows 
\end{proof}

\subsection{div-curl quantities on the half-space $\R^2_+$}
For notational convenience, we restrict our attention to the two-dimensional half-space $\R^2_+$.
\begin{theorem}\label{th:clmshalfspace}
Let $F \in L^2(\R^2_+,\R^2)$ and $G \in H^1(\R^2_+)$. If $\div F = 0$ in $\R^2_+$, that is
\[
 \int_{\R^2_+} F\cdot \nabla \Phi = 0  \quad \mbox{for any $\Phi \in C_c^\infty(\R^2_+)$}.
\]
Let $\nabla \Phi \in L^2(\R^2_+)$. Then if $\Phi = 0$ on $\R \times \{0\}$ in the sense of traces or $G =0$ on $\R \times \{0\}$ in the sense of traces, then
\begin{equation}\label{eq:clmsglobalhs}
 \int_{\R^2_+} F \cdot \nabla G\ \Phi \aleq \|F\|_{2,\R^2_+}\ \|\nabla G\|_{2,\R^2_+}\ \|\nabla \Phi\|_{(2,\infty),\R^2_+}.
\end{equation}
We also have the following localized estimate. For any $x_0 \in \R \times \{0\}$, $r > 0$
\begin{equation}\label{eq:clmshslocal2}
\begin{split}
 \int_{\R^2_+} F \cdot \nabla G\ \Phi \aleq& \|F\|_{2,B^+(x_0,r)}\ \|\nabla G\|_{2,B^+(x_0,r)}\ [ \Phi]_{BMO,B^+(x_0,r)}\\
 &+ \sum_{k=1}^\infty k\ \|F\|_{2,B^+(x_0,2^{k+5}r) \backslash B^+(x_0,2^{k-5}r)}\, \|\nabla G\|_{2,B^+(x_0,2^{k+5}r) \backslash B^+(x_0,2^{k-5}r)}\ [\Phi]_{BMO,B^+(x_0,2^{k+5}r)}.
\end{split}
 \end{equation}
Here, setting $\tilde{\Phi} := \eta_{B(x_0,R)} (\Phi-(\Phi)_{B(x_0,r)})$ where $\eta_{B(x_0,R)}(\cdot) = \eta((\cdot-x_0)/R)$ for the usual bump function $\eta$, we denote
\[
 [\Phi]_{BMO,B^+(x_0,R)} := [\tilde{\Phi}]_{BMO}.
\]
In particular, we have
\[
\begin{split}
 \int_{\R^2_+} F \cdot \nabla G\ \Phi \aleq& \|F\|_{2,B^+(x_0,r)}\ \|\nabla G\|_{2,B^+(x_0,r)}\ \|\nabla \Phi\|_{(2,\infty),B^+(x_0,r)}\\
 &+ \sum_{k=1}^\infty k\ \|F\|_{2,B^+(x_0,2^{k+5}r) \backslash B^+(x_0,2^{k-5}r)}\, \|\nabla G\|_{2,B^+(x_0,2^{k+5}r) \backslash B^+(x_0,2^{k-5}r)}\ \|\nabla \Phi\|_{(2,\infty),B^+(x_0,2^{k+5}r)}.
\end{split}
 \]
\end{theorem}
\begin{proof}
If $\Phi = 0$ on $\R \times \{0\}$ We extend $\Phi$ by zero to $\R^2_-$ and reflect $G$ and $F$. That is,
\[
 \tilde{G}(x,t) := G(x,|t|),
\]
and
\[
 \tilde{F}(x,t) := \left (\begin{array}{c} \frac{t}{|t|} F^1(x,|t|)\\ F^2(x,|t|)\end{array} \right ).
\]
Observe that $\div F = 0$ in $\R^2_+$ implies $\div \tilde{F} = 0$ in $\R^2$. 

If instead $G$ is zero on $\R \times \{0\}$ we extend $G$ by zero and $\Phi$ evenly, and otherwise proceed the same way.

By Theorem~\ref{th:clms},
\[
 \int_{\R^2_+} F \cdot \nabla G\ \Phi = \int_{\R^2} \tilde{F}\cdot \nabla \tilde{G}\ \Phi \aleq \|F\|_{2,\R^2_+}\ \|\nabla G\|_{2,\R^2_+}\ [\Phi]_{BMO}.
\]
That is the global estimate \eqref{eq:clmsglobalhs} follows Poincare-Sobolev embedding,
\[
 [\Phi]_{BMO} \aleq \|\nabla \Phi\|_{(2,\infty),\R^2}= \|\nabla \Phi\|_{(2,\infty),\R^2_+}.
\]
The localized estimate \eqref{eq:clmshslocal2}, follows directly from Theorem~\ref{th:clms}.
\end{proof}

\section{Hodge decomposition on the half-space}
On any star-shaped domain, Hodge decomposition tells us that a vector-field $F$ can be decomposed into a sum of a gradient part and divergence free part.

\begin{proposition}\label{pr:hodgewithestimate}
Let $F \in C^\infty(\overline{\R^{2}_+},\R^2)$ with compact support. Then we find a smooth function $\varphi \in C^\infty(\overline{\R^2_+})$ and a smooth vector field $H \in C^\infty(\overline{\R^2_+},\R^2)$ satisfying
\[
 \varphi\Big|_{\R \times \{0\}} \equiv 0,
\]
and
\[
 \div H = 0 \quad \mbox{in $\R^2_+$},
\]
and so that
\[
 F = \nabla \varphi + H \quad \mbox{in $\R^2_+$}.
\]
We also have the following estimates for any $p \in (1,\infty)$, $q \in [1,\infty]$
\[
 \|\nabla \varphi\|_{(p,q),\R^2_+} + \|H\|_{(p,q),\R^2_+}\aleq \|F\|_{(p,q),\R^2_+}.
\]
We also have the following localizing estimates:
For any $k \in \N_0$ and any $x \in \R^2_+$ with $\dist(x,\supp F) > \Lambda$,
\[
\begin{split}
 |\nabla^k\varphi(x)| \aleq \Lambda^{-k-1} \|F\|_{1,\R^2_+}.
\end{split}
 \]
and
 \[
\begin{split}
 |\nabla^{k} H| \aleq \Lambda^{-k-2} \|F\|_{1,\R^2_+}.
\end{split}
 \]
In particular, setting 
\[
 A_\Lambda := \left \{x \in \R^2_+:\  \dist(x,\supp F)>\Lambda \right \}
\]
for any $p > 2$,
\[
 \|\varphi\|_{p,A_\Lambda} \leq \Lambda^{\frac{2}{p}-1} \|F\|_{1,\R^2_+}
\]
and for any $p > 1$,
\[
 \|H\|_{p,A_\Lambda} + \|\nabla \varphi\|_{p,A_\Lambda} \leq \Lambda^{\frac{2}{p}-2} \|F\|_{1,\R^2_+}
\]
\end{proposition}
\begin{proof}
The Greens function on $\R^{2}_+$ is given by
\[
 G(x,y) := \log(|x-y|)-\log(|x^\ast-y|),
\]
where for $x = (x',t) \in \R^2_+$, $x^\ast$ denotes the reflected point over $\R \times \{0\}$, that is $x^\ast = (x',t)^\ast = (x',-t)$ .

Green's representation formula then tells us that we can find a solution $\varphi$
\[
\begin{cases}
  \lap \varphi = \dv F \quad &\mbox{in $\R^2_+$}\\
  \varphi = 0 \quad &\mbox{in $\R \times \{0\}$}.
\end{cases}
\]
Thus $H:=F-\nabla \varphi$ is divergence free. For a dimensional constant $c \in \R$, $\varphi$ is given by
\[
\begin{split}
 \varphi(x) :=& c \int_{\R^2_+} \nabla_y G(x,y)\ F(y)\ dy\\
 =& c \int_{\R^2_+} \brac{\frac{x-y}{|x-y|^2} - \frac{x^\ast -y}{|x^\ast -y|^2}}\cdot F(y)\ dy.
\end{split}
 \]
For $x \not \in \supp F$,
\[
\begin{split}
 |\nabla^k\varphi(x)| \aleq \int_{\R^2_+} |x-y|^{-1-k} |F(y)|\ dy.
\end{split}
 \]
In particular, if $\dist (x,\supp F) \geq \Lambda > 0$,
\[
\begin{split}
 |\nabla^k\varphi(x)| \aleq \Lambda^{-k-1} \|F\|_{L^1(\R^2_+)}.
\end{split}
 \]
As for the estimate observe that
\[
 \nabla^k H = \nabla^k F - \nabla^{k+1} \varphi
\]
and $\nabla^k F(x) = 0$ for $x \not \in \supp F$.
\end{proof}

\section{Localization estimates}
Even if $\supp f \subset B(x_0,r)$ there is no way to estimate the size of the support of $\Hz f$. However, the further away from the support of $f$ we estimate $\mathcal{H}f$ the smaller is the influence of $f$. Indeed, for $\dist(x,\supp f) \aeq \Lambda$,
\[
 |\mathcal{H}f(x)| = c|\int_{\R} \frac{x-z}{|x-z|^2} f(z)\, dz| \aleq \Lambda^{-1} \|f\|_{1,\R}.
\]
From this estimate, one obtains 
\begin{proposition}\label{pr:localHilbert}
Let $\Hz$ be the Hilbert transform. Then for any $x_0 \in \R$, $R > 0$
\[
 \|\Hz f\|_{(2,\infty),I(x_0,R)} \aleq \|f\|_{(2,\infty),I(x_0,2^{k_0} R)}+ \tail{\sigma}{\|f\|_{(2,\infty)}}{x_0}{R}{k_0}.
\]
Here $\sigma$ is a uniform constant.
\end{proposition}
A further quasi-local estimate involving the harmonic extension is the following.
\begin{proposition}\label{pr:VtBMO}
Let $\solu \in \dot{H}^{\frac{1}{2}}(\R,\R^N)$ and $\solV$ the Poisson extension of $\solu$ to $\R^2_+$, $\solV(x,t) := p_t\ast\solu(x)$.

Let $\solV^e$ be the even reflection of $\solV$ to $\R^2_-$. Then for any $k_0 \geq 10$ and for any $x_0 \in \R \times \{0\}$, any $R > 0$
\[
 \left [\eta_{B(x_0,R)} \brac{\solV^e -(\solV^e)_{B(x_0,R)}} \right ]_{BMO,\R^2} \aleq \|\lapv \solu\|_{(2,\infty),B(x_0,2^{k_0}R)} + \tail{\sigma}{\|\lapv \solu\|_{(2,\infty)}}{x_0}{R}{k_0}.
\]
Here $\sigma > 0$ is a uniform constant.
\end{proposition}
\begin{proof}
By Poincar\'e and Poincar\'e-Sobolev inequality, for any $p \in (1,2)$,
\[
\begin{split}
 &\left [\eta_{B(x_0,R)} \brac{\solV^e -(\solV^e)_{B(x_0,R)}} \right ]_{BMO(\R^2)}\\
 \aleq& \sup_{B(y,\rho) \subset B(x_0,2R)} \rho^{\frac{2}{p'}-1}\brac{\int_{h \in [-\rho,\rho]^2} |\nabla \solV^e(y_0+h)|^p\ dh}^{\frac{1}{p}}.
\end{split}
 \]
Now by Fubini and H\"older inequality,
\[
 \brac{\int_{h \in [-\rho,\rho]^2} |\nabla \solV^e(y_0+h)|^p\ dh}^{\frac{1}{p}} \aleq \rho^{1-\frac{2}{p'}} \left \|\brac{\int_{-\rho}^\rho |\nabla \solV^e (y_0 + (\cdot ,t))|^2\ dt}^{\frac{1}{2}}\right \|_{(2,\infty),I(0,\rho)},
\]
and thus, since $x_0 \in \R \times \{0\}$
\[
\begin{split}
 \left [\eta_{B(x_0,R)} \brac{\solV^e -(\solV^e)_{B(x_0,R)}} \right ]_{BMO(\R^2)}
 \aleq& \left \|\brac{\int_{0}^{3R} |\nabla \solV (\cdot,t))|^2\ dt}^{\frac{1}{2}}\right \|_{(2,\infty),I(x_0,3R)}.
\end{split}
 \]
As for the proof of Proposition~\ref{pr:V:err5} we set
\[
 T^1f(x,t) :=   \lapv p_t \ast \Hz f(x)
\]
\[
 T^2f(x,t) :=  -\lapv p_t \ast f(x).
\]
Then,
\[
 \nabla \solV = T(\lapv \solu) = T(\eta_{k_0}\lapv \solu)+\sum_{\ell = k_0}^\infty T(\xi_{\ell}\lapv \solu),
\]
where $\eta \in C_c^\infty(B(0,1))$ is a typical bump function constantly one in $B(0,\frac{1}{2})$ and 
\[
 \eta_{k} := \eta((x_0-x)/{2^k r}),\ \xi_k := \eta_k - \eta_{k-1}.
\]
Then we have
\[
\begin{split}
 &\left [\eta_{B(x_0,R)} \brac{\solV^e -(\solV^e)_{B(x_0,R)}} \right ]_{BMO(\R^2)}\\
 \aleq& \left \|\brac{\int_{0}^{\infty} |T[\eta_{k_0}\lapv \solu](\cdot,t)|^2\ dt}^{\frac{1}{2}}\right \|_{(2,\infty),\R} + \sum_{\ell =k_0}^\infty \|\eta_{4R}\, T(\xi_\ell\lapv \solu)\|_{2,\R^2_+}
\end{split}
 \]
For the first term, we write $T^1f = t^{-\frac{1}{2}} \kappa_t \ast \Hz f$ and $T^2f = t^{-\frac{1}{2}} \kappa_t \ast \Hz f$, then the square function estimate, see  \cite[Chapter I.C, \textsection 8.23, p.46]{Stein-Harmonic-Analysis}, implies
\[
 \left \|\brac{\int_{0}^{\infty} |T[\eta_{k_0}\lapv \solu]|^2\ dt}^{\frac{1}{2}}\right \|_{(2,\infty),\R} \aleq \|\eta_{k_0} \lapv \solu \|_{(2,\infty),\R}.
\]
For the remaining term observe that $\supp \xi_\ell \times \{0\}$ and $\supp \eta_{4R}$ are disjoint. Therefore, by Lemma~\ref{la:V:err5:disjointT}
\[
\|\eta_{4R}\, T(\xi_\ell\lapv \solu)\|_{2,\R^2_+}\aleq 2^{-\sigma\ell} \|\xi_\ell \lapv \solu\|_{(2,\infty),\R}.
\]
This concludes the proof.
\end{proof}

\section{On Poisson-type extension operators}
Recall that $p$ denotes the Poisson kernel
\[
 p(x) = \frac{1}{\brac{1+|x|^2}^{\frac{n+1}{2}}},
\]
and $p_t(x) = t^{-n} p(x/t)$, that is
\[
 p_t(x) = \frac{t}{\brac{t^2+|x|^2}^{\frac{n+1}{2}}}.
\]
In this section we prove a few, probably well-known, estimates on $p_t$. 
\begin{lemma}\label{la:ptest}
For $\alpha \in [0,1)$, $t > 0$, $z \in \R^n$ we have
\begin{equation}\label{eq:la:ptest:1}
 |\laps{\alpha} p_t(z)| \aleq \brac{t^2 + |z|^2}^{-\frac{n+\alpha}{2}},
\end{equation}
and
\begin{equation}\label{eq:la:ptest:2}
 |\laps{\alpha} \nabla p_t(z)| \aleq  \brac{t^2 + |z|^2}^{-\frac{n+\alpha+1}{2}}
\end{equation}
\end{lemma}
\begin{proof}
It suffices to prove \eqref{eq:la:ptest:1} and \eqref{eq:la:ptest:2} for $t =1$.

We treat \eqref{eq:la:ptest:1} first. We have 
\[
\begin{split}
 \laps{\alpha} p_1(z) =& \int_{\R^n} \frac{\frac{1}{\brac{1 + |z|^2}^{\frac{n+1}{2}}} - \frac{1}{\brac{1 + |y|^2}^{\frac{n+1}{2}}}}{|z-y|^{n+\alpha}}\ dy\\
 \end{split}
 \]
We split the integral
\[
\begin{split}
 \laps{\alpha} p_1(z)=&\int_{A(z)\dot{\cup} B(z)\dot{\cup} C(z)} \frac{\brac{1 + |y|^2}^{\frac{n+1}{2}} - \brac{1 + |z|^2}^{\frac{n+1}{2}} }{|z-y|^{n+\alpha}}\ \brac{1 + |y|^2}^{-\frac{n+1}{2}}\ \brac{1+ |z|^2}^{-\frac{n+1}{2}}\ dy,
 \end{split}
 \]
where
\[
 A(z) := \left \{y \in \R^n: \quad |z-y| \leq \frac{1}{2} (1+|z|) \mbox{ or } |z-y| \leq \frac{1}{2} (1+|y|) \right \}, 
\]
\[
 B(z) := \left \{y \in \R^n: \quad |z-y| > \frac{1}{2} (1+|z|) \mbox{ and } |z-y| > \frac{1}{2} (1+|y|) \mbox{ and } |z| \leq |y|\right \} ,
\]
and
\[
 C(z) := \left \{y \in \R^n: \quad |z-y| > \frac{1}{2} (1+|z|) \mbox{ and } |z-y| > \frac{1}{2} (1+|y|) \mbox{ and } |z| > |y|\right \}.
\]
Firstly, we observe that $y \in A(z)$ implies $1+|z| \aeq 1+|y|$, and consequently $1+|z|^2 \aeq 1+|y|^2$ and $|z-y| \aleq \sqrt{1+|z|^2}$. Thus, with the mean value theorem,
\[
\begin{split}
&\int_{A(z)} \frac{\left |\brac{1 + |y|^2}^{\frac{n+1}{2}} - \brac{1 + |z|^2}^{\frac{n+1}{2}} \right |}{|z-y|^{n+\alpha}}\ \brac{1 + |y|^2}^{-\frac{n+1}{2}}\ \brac{1+ |z|^2}^{-\frac{n+1}{2}}\ dy\\
\aleq &\int_{B(z,C\sqrt{1+|z|^2})} |z-y|^{-n+1-\alpha}\  \brac{1+ |z|^2}^{-\frac{n+2}{2}}dy\\
\aeq& \brac{1+ |z|^2}^{-\frac{n+1+\alpha}{2}}.\\
\end{split}
\]
Secondly, if $y \in B(z)$, $|z-y| \ageq \sqrt{1+|z|^2}$ and
\[
\begin{split}
 &\frac{\left |\brac{1 + |y|^2}^{\frac{n+1}{2}} - \brac{1 + |z|^2}^{\frac{n+1}{2}} \right |}{|z-y|^{n+\alpha}}\ \brac{1 + |y|^2}^{-\frac{n+1}{2}}\ \brac{1+ |z|^2}^{-\frac{n+1}{2}}\\ 
 \aleq &|z-y|^{-n-\alpha}\  \brac{1+ |z|^2}^{-\frac{n+1}{2}}.
\end{split}
 \]
Consequently,
\[
\begin{split}
&\int_{B(z)} \frac{\left |\brac{1 + |y|^2}^{\frac{n+1}{2}} - \brac{1 + |z|^2}^{\frac{n+1}{2}} \right |}{|z-y|^{n+\alpha}}\ \brac{1 + |y|^2}^{-\frac{n+1}{2}}\ \brac{1+ |z|^2}^{-\frac{n+1}{2}}\ dy\\
\aleq &\int_{\R^n \backslash B(z,c\sqrt{1+|z|^2})} |z-y|^{-n-\alpha}\  \brac{1+ |z|^2}^{-\frac{n+1}{2}}\, dy\\
\aeq& \brac{1+ |z|^2}^{-\frac{n+1+\alpha}{2}}.\\
\end{split}
\]
Lastly, if $y \in C(z)$, $|z-y| \ageq \sqrt{1+|z|^2}$ and thus
\[
\begin{split}
 &\frac{\left |\brac{1 + |y|^2}^{\frac{n+1}{2}} - \brac{1 + |z|^2}^{\frac{n+1}{2}} \right |}{|z-y|^{n+\alpha}}\ \brac{1 + |y|^2}^{-\frac{n+1}{2}}\ \brac{1+ |z|^2}^{-\frac{n+1}{2}}\\ 
 \aleq &|z-y|^{-n-\alpha}\  \brac{1+ |y|^2}^{-\frac{n+1}{2}} \aleq \brac{1+ |z|^2}^{-\frac{n+\alpha}{2}}\  \brac{1+ |y|^2}^{-\frac{n+1}{2}}.
\end{split}
 \]
Consequently,
\[
\begin{split}
&\int_{C(z)} \frac{\left |\brac{1 + |y|^2}^{\frac{n+1}{2}} - \brac{1 + |z|^2}^{\frac{n+1}{2}} \right |}{|z-y|^{n+\alpha}}\ \brac{1 + |y|^2}^{-\frac{n+1}{2}}\ \brac{1+ |z|^2}^{-\frac{n+1}{2}}\ dy\\
\aleq &\int_{B(0,|z|)} \brac{1+ |z|^2}^{-\frac{n+\alpha}{2}}\  \brac{1+ |y|^2}^{-\frac{n+1}{2}} dy\\
\aeq& \brac{1+ |z|^2}^{-\frac{n+\alpha}{2}}.\\
\end{split}
\]
This settles \eqref{eq:la:ptest:1}.

For \eqref{eq:la:ptest:2} we choose the following integral representation for the operator $\laps{\alpha}\nabla$, 
\[
 \laps{\alpha} \nabla f(z) = \int_{\R^n} (f(y) - f(z) - (y-z)\cdot\nabla f(z)) \frac{k(|z-y|)}{|z-y|^{n+\alpha+1}}\ dz. 
\]
Here $k \in C^\infty(\R^n \backslash \{0\},\R^n)$ is a zero-homogeneous bounded function, c.f. \cite[Proposition 2.4.7, Proposition 2.4.8]{GrafakosC}.
Recall that
\[
 \nabla p_1(z) = -(n+1)\, \frac{z}{\brac{1+|z|^2}^{\frac{n+3}{2}}}
\]
and
\[
 |\nabla^2 p_1(z)| \aleq \frac{1}{\brac{1+|z|^2}^{\frac{n+3}{2}}}
\]
Take $A(z)$, $B(z)$, $C(z)$ as above.
For $y \in A(z)$ we have $1+|z|^2 \aeq 1+|y|^2$, for all $y \in A(z)$,
\[
 \int_{A(z)} (p_1(y) - p_1(z) - (y-z)\nabla p_1(z)) \frac{1}{|z-y|^{n+\alpha+1}}\ k(|z-y|) \aleq \frac{1}{\brac{1+|z|^2}^{\frac{n+2+\alpha}{2}}}. 
\]
For $y \in B(z) \cup C(z)$, $|y-z| \aleq \sqrt{1+|z|^2}$ and thus
\[
 \int_{B(z) \cup C(z)} \left |(y-z)\nabla p_1(z) \right | \frac{1}{|z-y|^{n+\alpha+1}}\ k(|z-y|)\ dy \aleq \frac{1}{\brac{1+|z|^2}^{\frac{n+2+\alpha}{2}}}.
\]
Also, for $y \in B(z)$,
\[
  \left |\frac{1}{\brac{1+|y|^2}^{\frac{n+1}{2}}}-\frac{1}{\brac{1+|z|^2}^{\frac{n+1}{2}}} \right | \aleq \frac{1}{\brac{1+|z|^2}^{\frac{n+1}{2}}}
\]
Consequently,
\[
 \int_{B(z)} \left |p_1(y) - p_1(z) \right |\frac{1}{|z-y|^{n+\alpha+1}}\ k(|z-y|) \aleq \frac{1}{\brac{1+|z|^2}^{\frac{n+2+\alpha}{2}}}.
\]
For $y \in C(z)$,
\[
 \left |p_1(y) - p_1(z) \right |\frac{1}{|z-y|^{n+\alpha+1}}\ k(|z-y|) \aleq \frac{1}{\brac{1+|z|^2}^{n+1+\alpha}}\ \brac{1+|y|^2}^{-\frac{n+1}{2}},
 \]
and consequently,
\[
 \int_{C(z)} \left |p_1(y) - p_1(z) \right |\frac{1}{|z-y|^{n+\alpha+1}}\ k(|z-y|) \aleq \frac{1}{\brac{1+|z|^2}^{n+1+\alpha}}.
\]
This establishes \eqref{eq:la:ptest:2}.
\end{proof}
As a consequence of the decay estimates in Lemma~\ref{la:ptest} we obtain the following estimates for the harmonic extension $\varphi^h = p_t \ast f$ in points in $\R^{n+1}_+$ which are away from the support of $f$ in $\R^n$.
\begin{lemma}\label{la:harmextest}
Let $\varphi^h(x,t) = p_t \ast \varphi(x)$ the harmonic extension of $\varphi \in L^1(\R^n)$ to $\R^{n+1}_+$. Then if $(x,t) \in \R^{n+1}_+$ so that $\dist((x,t),\supp \varphi \times \{0\}) > \Lambda > 0$, we have
\begin{equation}\label{la:harmextest:1}
 |\varphi^h(x,t)| \aleq \Lambda^{-n} \|\varphi\|_{1,\R^n},
\end{equation}
\begin{equation}\label{la:harmextest:2}
 |\nabla \varphi^h(x,t)| \aleq \Lambda^{-n-1} \|\varphi\|_{1,\R^n}.
\end{equation}
In particular, for any $p > \frac{n+1}{n}$,
\begin{equation}\label{la:harmextest:3}
 \|\varphi^h\|_{p,\R^{n+1}_+ \backslash B(\supp \varphi \times \{0\},\Lambda)} \aleq \Lambda^{\frac{n+1}{p}-n} \|\varphi\|_{1,\R^n},
\end{equation}
and for any $p > 1$,
\begin{equation}\label{la:harmextest:4}
 \|\nabla \varphi^h\|_{p,\R^{n+1}_+ \backslash B(\supp \varphi \times \{0\},\Lambda)} \aleq \Lambda^{\frac{n+1}{p}-n-1} \|\varphi\|_{1,\R^n}.
\end{equation}
\end{lemma}
\begin{proof}
Assume that $(x,t) \in \R^{n+1}_+$ is so that 
\[\dist((x,t),\supp \varphi \times \{0\}) > \Lambda.\]
Then, by a direct computation, for any $y \in \supp \varphi$ we have
\[
 |p_t(x-y)|\aleq \Lambda^{-n},
\]
\[
 |\partial_x p_t (x-y)| \aleq \Lambda^{-n-1},
\]
and
\[
 |\partial_t p_t (x-y)| \aleq \Lambda^{-n-1}.
\]
This proves \eqref{la:harmextest:1} and \eqref{la:harmextest:2}, since for such $x$,
\[
 |\varphi^h(x,t)| = \left |\int_{\R^n} p_t(x-y)\ \varphi(y)\ dy\right | \aleq \Lambda^{-n} \|\varphi\|_{1,\R^n},
\]
and
\[
 |\nabla_{\R^2} \varphi^h(x,t)| \aleq \left |\int_{\R^n} \partial_x p_t (x-y)\ \varphi(y)\ dy\right | + \left |\int_{\R^n} \partial_t p_t (x-y)\ \varphi(y)\ dy\right | \aleq \Lambda^{-n-1} \|\varphi\|_{1,\R^n}.
\]
The $L^p$-estimates \eqref{la:harmextest:3} and \eqref{la:harmextest:4} follow by splitting the support of the $L^p$-norm
\[
 \|f\|_{p,\R^{n+1}_+ \backslash B(\supp \varphi \times \{0\},\Lambda)} \aleq \sum_{k=1}^\infty \|f\|_{p, B(\supp \varphi \times \{0\},2^{k} \Lambda) \backslash B(\supp \varphi \times \{0\},2^{k-1} \Lambda)}.
\]
and then apply H\"older inequality and then \eqref{la:harmextest:1} or \eqref{la:harmextest:2}, respectively.
\end{proof}

\section{Estimates on nonlocal operators and orthogonal projections}
\begin{lemma}\label{la:Piperplapvu}
For $\solu: \R \to \R^N$ so that $u(x) \in \mathcal{N}$ for almost every $x \in (-4,4)$. For any $r \in (0,1)$ and $x_0 \in (-1,1)$, 
\[
\begin{split}
  \|\Pi^\perp(\solu) \lapv \solu\|_{2,I(x_0,r)} \aleq&  \|\lapv \solu\|_{2,I(x_0,2^{k_0} r)}\, \|\lapv \solu\|_{(2,\infty),I(x_0,2^{k_0} r)}\\
 &+ \|\lapv \solu\|_{2,\R}\ \tail{\sigma}{\|\lapv \solu\|_{(2,\infty)}}{x_0}{r}{k_0} + r^{\frac{1}{2}}.
 \end{split}
 \]
\end{lemma}
\begin{proof}
For $x,y \in (-2,2)$ we have, for example by \cite[Lemma A.1.]{Mazowiecka-Schikorra-2017},
\[
 |\Pi^\perp(\solu(x)) (\solu(x)-\solu(y))|\aleq |\solu(x)-\solu(y)|^2.
\]
In particular, for any $x \in (-2,2)$ we find
\[
 |\Pi^\perp(\solu(x)) \lapv \solu(x)| \aleq \int_{\R} \frac{|\solu(x)-\solu(y)|^2}{|x-y|^{\frac{3}{2}}}\ dy + \int_{\R \backslash (-4,4)} \frac{|\solu(x)-\solu(y)|}{1+|y|^{\frac{3}{2}}}\ dy.
\]
Moreover, with Young inequality, again for any $x \in (-2,2)$,
\[
 \int_{\R \backslash (-4,4)} \frac{|\solu(x)-\solu(y)|}{1+|y|^{\frac{3}{2}}}\ dy \aleq \int_{\R \backslash (-4,4)} \frac{1+|\solu(x)-\solu(y)|^2}{1+|y|^{\frac{3}{2}}}\ dy \aleq 1 + \int_{\R} \frac{|\solu(x)-\solu(y)|^2}{|x-y|^{\frac{3}{2}}}\ dy.
\]
That is, for any $x \in (-2,2)$,
\[
 |\Pi^\perp(\solu(x)) \lapv \solu(x)| \aleq 1+\int_{\R} \frac{|\solu(x)-\solu(y)|^2}{|x-y|^{\frac{3}{2}}}\ dy.
\]
Now observe that with the notation
\[
 H_{\frac{1}{2}}(a,b) := \lapv (ab) - a\lap b-\lapv a\ b,
\]
we have
\[
 \int_{\R} \frac{|\solu(x)-\solu(y)|^2}{|x-y|^{2}}\ dy = c\left |H_{\frac{1}{2}}(\solu,\solu) \right |.
\]
Localizing the fractional Leibniz rule, see, e.g., \cite[Lemma A.7.]{Blatt-Reiter-Schikorra-2016}, 
\[
\begin{split}
 \|H_{\frac{1}{2}}(\solu,\solu)\|_{2,I(x_0,r)} \aleq& \|\lapv \solu\|_{2,I(x_0,2^{k_0} r)}\, \|\lapv \solu\|_{(2,\infty),I(x_0,2^{k_0} r)}\\
 &+ \|\lapv \solu\|_{2,\R}\ \tail{\sigma}{\|\lapv \solu\|_{(2,\infty)}}{x_0}{r}{k_0}.
\end{split}
 \]

\end{proof}

\section*{Acknowledgment}
A.S. would like to thank E. Kuwert for the reference to Douglas' work \cite{Douglas-1931}. The author is supported by the German Research Foundation (DFG) through grant no.~SCHI-1257-3-1. He receives funding from the Daimler and Benz foundation. A.S. is Heisenberg fellow.

\bibliographystyle{amsplain}
\bibliography{bib}%

\end{document}